\documentclass[11pt]{article}

\usepackage[T1]{fontenc}
\usepackage{multicol}
\usepackage{amssymb,amsmath,amsthm,bm}
\usepackage[colorlinks=true,urlcolor=blue,linkcolor=blue,citecolor=blue]{hyperref}
\usepackage{aliascnt}
\usepackage[nameinlink,capitalise,noabbrev]{cleveref}
\usepackage{mathrsfs,verbatim}
\usepackage{caption}
\usepackage{graphicx,float}
\usepackage{color,mathtools,tikz,xcolor}
\usepackage{enumerate,enumitem}
\usepackage{lineno}
\usepackage{comment}
\usepackage{pdfpages}
\usepackage{algorithm}
\usepackage{algorithmic}
\usepackage[mathscr]{euscript}
\usepackage{appendix}
\usepackage{geometry}
\usepackage{svg}

\usetikzlibrary{shapes.geometric,calc}
\numberwithin{equation}{section}

\newcommand{\eps}{\varepsilon}
\newcommand{\e}{\varepsilon}

\newcommand{\C}{\mathcal{C}}

\newcommand{\K}{\mathcal{K}}
\newcommand{\J}{\mathcal{J}}

\newcommand{\h}{H}
\newcommand{\cP}{\mathcal{P}}

\newcommand{\proofstep}[2]{%
  \par\medskip
  \noindent\textbf{Step #1.}\ \emph{#2.}%
  \par\nobreak\smallskip\noindent\ignorespaces
}

\newcommand{\sketchstep}[2]{\proofstep{#1}{#2}}

\theoremstyle{plain}

\newtheorem{theorem}{Theorem}[section]
\crefname{theorem}{Theorem}{Theorems}
\Crefname{theorem}{Theorem}{Theorems}

\newcommand{\newaliastheorem}[3]{%
  \newaliascnt{#1}{theorem}%
  \newtheorem{#1}[#1]{#2}%
  \aliascntresetthe{#1}%
  \crefname{#1}{#2}{#3}%
  \Crefname{#1}{#2}{#3}%
}

\newaliastheorem{problem}{Problem}{Problems}
\newaliastheorem{question}{Question}{Questions}
\newaliastheorem{exercise}{Exercise}{Exercises}
\newaliastheorem{coro}{Corollary}{Corollaries}
\newaliastheorem{conjecture}{Conjecture}{Conjectures}
\newaliastheorem{claim}{Claim}{Claims}
\newaliastheorem{proposition}{Proposition}{Propositions}
\newaliastheorem{lemma}{Lemma}{Lemmas}
\newaliastheorem{fact}{Fact}{Facts}
\newaliastheorem{observation}{Observation}{Observations}
\newaliastheorem{case}{Case}{Cases}

\theoremstyle{definition}
\newaliastheorem{defn}{Definition}{Definitions}
\newaliastheorem{example}{Example}{Examples}
\newaliastheorem{condition}{Condition}{Conditions}
\newaliastheorem{assumption}{Assumption}{Assumptions}
\newaliastheorem{solution}{Solution}{Solutions}

\newaliastheorem{remark}{Remark}{Remarks}
\newaliastheorem{note}{Note}{Notes}
\newaliastheorem{notes}{Notes}{Notes}

\crefname{section}{Section}{Sections}
\Crefname{section}{Section}{Sections}
\crefname{subsection}{Section}{Sections}
\Crefname{subsection}{Section}{Sections}
\crefname{figure}{Figure}{Figures}
\Crefname{figure}{Figure}{Figures}
\crefname{equation}{Equation}{Equations}
\Crefname{equation}{Equation}{Equations}

\newenvironment{proofclaim}[1][Proof of claim]
{
\begin{proof}[#1]}
  {
\end{proof}}

\setlist{nolistsep}

\title{\Large\bf Tight Hamiltonian Cycles in Uniformly Dense $3$-Graphs}

\author{
  Yaobin Chen \thanks{Extremal Combinatorics and Probability Group (ECOPRO), Institute for Basic Science (IBS), Daejeon, South Korea. Email: {\tt ybchen@ibs.re.kr}.}
  \and Jie Han \thanks{School of Mathematics and Statistics, Beijing Institute of Technology, Beijing, China. Email:~{\tt han.jie@bit.edu.cn}.}
  \and Xizhi Liu\thanks{School of Mathematical Sciences, University of Science and Technology of China, Hefei, China. Email:~{\tt liuxizhi@ustc.edu.cn}.}
}

\date{}

\begin{document}

\maketitle

\begin{abstract}
  We study minimum degree conditions for tight Hamiltonian cycles in uniformly dense $3$-uniform hypergraphs.
  We prove that for every $d,\alpha>0$, every sufficiently large $(\rho,d)$-dense $3$-graph on $n$ vertices with minimum codegree at least $(1/3+\alpha)n$ contains a tight Hamiltonian cycle.
  This resolves a problem of Aigner-Horev and Levy in a stronger form, and the constant $1/3$ is asymptotically best possible.
  We also show that uniform density does not lower the asymptotic vertex-degree threshold: there are $(\rho,d)$-dense $3$-graphs with minimum vertex degree $(5/9-o(1))\binom{n}{2}$ and no tight Hamiltonian cycle.
  Finally, we construct $(\rho,2-\sqrt{3})$-dense examples with minimum codegree $(2-\sqrt{3}-o(1))n$ and no tight Hamiltonian cycle, answering negatively a question of Ara{\'u}jo, Piga and Schacht.
\end{abstract}

\section{Introduction}

Hamiltonian cycles are a central object in extremal combinatorics.
The classical theorem of Dirac~\cite{Dirac1952} asserts that every graph on $n\ge 3$ vertices with minimum degree at least $n/2$ contains a Hamiltonian cycle.
For $k$-uniform hypergraphs, or $k$-graphs, there are several natural analogues of cycles.
A cycle is called \emph{loose} if consecutive edges intersect in exactly one vertex, and \emph{tight} if consecutive edges intersect in exactly $k-1$ vertices.
This work is concerned with tight Hamiltonian cycles.
Dirac-type degree conditions for tight Hamiltonian cycles in hypergraphs have been studied extensively over the past two decades.
For $3$-graphs, classical results of R\"odl, Ruci\'nski and Szemer\'edi~\cite{RRS2006,RRS2008} show that the minimum codegree threshold for tight Hamiltonian cycles is $n/2+o(n)$.
The corresponding minimum \emph{vertex-degree} threshold was a long-standing open problem, and was eventually shown by Reiher, R\"odl, Ruci\'nski, Schacht and Szemer\'edi~\cite{RRRSZ} to be $\frac{5}{9}\binom{n}{2}+o(n^2)$.
This vertex-degree result was later generalized to $k$-graphs independently by Lang and Sanhueza-Matamala~\cite{LS2022} and by Polcyn, Reiher, R\"odl and Sch\"ulke~\cite{PRRS2021}.

There has been growing interest in extremal problems under pseudorandomness assumptions, especially for uniformly dense hypergraphs~\cite{Reiher2020UniformlyDense,RRS2018Vanishing,RRS2018Mantel,RRS2018WeakQuasirandom,Schacht2023Restricted}.
The relevant notion, introduced by Erd\H{o}s and S\'os~\cite{ES}, is \emph{uniform density}.
In the $3$-uniform case, it can be formulated as follows.
For vertex sets $X_1,X_2,X_3\subseteq V(H)$, let $e_H(X_1,X_2,X_3)$ denote the number of ordered triples $(x_1,x_2,x_3)\in X_1\times X_2\times X_3$ whose entries are distinct and satisfy $x_1x_2x_3\in E(H)$.
An $n$-vertex $3$-graph $H$ is called $(\rho,d)$-dense if, for every three
sets $X_1,X_2,X_3\subseteq V(H)$, not necessarily pairwise disjoint,
\[
  e_H(X_1,X_2,X_3) \ge d|X_1||X_2||X_3|-\rho n^3.
\]
The systematic study of Hamiltonian cycles in uniformly dense hypergraphs was initiated by Lenz, Mubayi and Mycroft~\cite{LMM2016}, who showed that uniform density can substantially lower the degree conditions needed for loose Hamiltonian cycles.
Many related results have since been obtained for spanning subgraphs in pseudorandom and related hypergraph models; see, for example,~\cite{CHKPO,HZH,BAW,GH1,MR3853886,DHS}.

By contrast, tight Hamiltonicity in uniformly dense hypergraphs remains much less developed.
Motivated by their work on tight cycles in related pseudorandom settings, Aigner-Horev and Levy~\cite{AignerHorevLevy_cite} asked for the codegree threshold for tight Hamiltonian cycles in $(\rho,d)$-dense $3$-graphs, especially when $d>1/3$.
Their question suggests that the threshold should be close to $n/3$, substantially below the $n/2+o(n)$ threshold for general $3$-graphs.

Our first result shows that this heuristic is asymptotically correct.
In fact, no density threshold bounded away from zero is needed: for every fixed $d,\alpha>0$, all sufficiently large $(\rho,d)$-dense $3$-graphs with minimum codegree at least $(1/3+\alpha)n$ contain a tight Hamiltonian cycle.

\begin{theorem}\label{thm:main_codegree}
  For every $d,\alpha>0$, there exist $\rho>0$ and $n_0\in\mathbb{N}$ such that every $(\rho,d)$-dense $3$-graph $H$ on $n\ge n_0$ vertices with $\delta_2(H)\ge \left(\frac{1}{3}+\alpha\right)n$ contains a tight Hamiltonian cycle.
\end{theorem}

The following result shows that the constant $1/3$ in \cref{thm:main_codegree} is essentially tight.
It is witnessed by a construction obtained from a random $2$-edge-coloring of the complete graph, where the hyperedges are the monochromatic triangles.

\begin{theorem}\label{thm:codegree_lower_bound}
  For every $\gamma,\rho>0$, there exist $d>0$ and $n_0\in\mathbb{N}$ such that, for every $n\ge n_0$, there is an $n$-vertex $(\rho,d)$-dense $3$-graph $H$ with
  $
  \delta_2(H)\ge \left(\frac{1}{3}-\gamma\right)n
  $
  which contains no tight Hamiltonian cycle.
\end{theorem}

Recall that, by a result of Reiher, R\"odl, Ruci\'nski, Schacht and Szemer\'edi~\cite{RRRSZ}, for general $3$-graphs the asymptotic vertex-degree threshold for tight Hamiltonian cycles is $\frac{5}{9}\binom{n}{2}$.
In view of \cref{thm:main_codegree}, it is natural to ask whether uniform density lowers the vertex-degree threshold as well.
The following result shows that it does not.

\begin{theorem}\label{thm:vertex_degree_lower_bound}
  For every $\gamma,\rho>0$, there exist $d>0$ and $n_0\in\mathbb{N}$ such that, for every $n\ge n_0$, there is an $n$-vertex $(\rho,d)$-dense $3$-graph $H$ with
  $
  \delta_1(H)\ge \left(\frac{5}{9}-\gamma\right)\binom{n}{2}
  $
  which contains no tight Hamiltonian cycle.
\end{theorem}

We close the introduction with a related question of Ara{\'u}jo, Piga and Schacht~\cite{AraujoPigaSchacht_cite}.
They asked whether every $(\rho,1/4)$-dense $3$-graph with minimum codegree at least $n/4$ must contain a tight Hamiltonian cycle.
Set $d_\star \coloneqq 2-\sqrt{3}$.
The following theorem constructs counterexamples already at density $d_\star$, which is strictly larger than $1/4$.

\begin{theorem}\label{thm:araujo_piga_schacht}
  For every $\gamma,\rho>0$, there exists $n_0\in\mathbb{N}$ such that, for every $n\ge n_0$, there is an $n$-vertex $(\rho,d_\star)$-dense $3$-graph $H$ with $\delta_2(H)\ge (d_\star-\gamma)n$ which contains no tight Hamiltonian cycle.
\end{theorem}

\medskip

This work is organized as follows.
In \cref{pre}, we introduce notation and collect preliminary tools.
In \cref{section:regularity}, we review the regularity tools used later.
In \cref{sec:robusttight}, we define robust tight connectedness and derive the main structural case distinction.
In \cref{sec:rcc-absorption}, we state a common absorption lemma for
robust component classes using an RCC path-cover condition.
The switcher and non-switcher cases are proved in \cref{switcher-case} and \cref{non-switcher-case}, respectively.
The short final reduction to these cases is given at the end of \cref{non-switcher-case}.
In \cref{const}, we give the lower-bound constructions for \cref{thm:codegree_lower_bound,thm:vertex_degree_lower_bound,thm:araujo_piga_schacht}.
We finish with open problems in \cref{sec:concludingremark}.

\subsection*{Proof sketch of \texorpdfstring{\cref{thm:main_codegree}}{the main theorem}}

The proof of \cref{thm:main_codegree} occupies the main part of this work.
We now outline the proof strategy.
The argument combines a partition--index absorbing method with a robust connecting argument.
The absorbing part reduces the problem to finding an almost-spanning tight cycle with a controlled leftover set.
We first build a small absorbing structure, then cover almost all remaining vertices by a tight cycle, and finally absorb the leftover set, after arranging that its size is divisible by $3$.
Thus the main point of the proof is to construct the almost-spanning tight cycle in a way that leaves its endpoints easy to connect.
Let $H$ be a $(\rho,d)$-dense $3$-uniform hypergraph on $n$ vertices with $\delta_2(H)\ge \left(\frac{1}{3}+\alpha\right)n$.
The proof is organized into the following four steps.

\sketchstep{1}{Robust tight connectedness}
We first reserve a uniformly random set $A\subseteq V(H)$ of linear size.
By the minimum codegree condition and \cref{lem:chernoff}, with high probability every pair has codegree at least $(1/3+\alpha/2)|A|$ into $A$.
This reservoir will later contain the internal vertices of the short paths used to connect longer paths.
Since these connections are made one after another, some vertices of $A$ may already have been used when a later connection is needed.

For this purpose, we introduce \emph{robust tight connectedness} for ordered pairs.
Roughly speaking, two ordered pairs are robustly tightly connected with respect to $A$ if, after deleting any small forbidden subset of $A$, they can still be joined by a short tight path whose internal vertices lie in the remaining part of $A$.
This leaves enough connectors even after earlier connections have been made, and allows us to postpone the actual choice of connecting vertices until the end of the argument.
The formal definition is \cref{def:robustconnected}, and the main lemma producing robust connection components is \cref{equivalentclass}.

\sketchstep{2}{Connection classes from the regular slice lemma}
The endpoints of a tight path are ordered pairs.
We therefore organize the ordered shadow of $H$ according to which ordered pairs can be robustly connected through the reservoir.
We use the regular slice lemma, stated as \cref{thm:strengthened_regularity}, inside the reservoir to pass to a reduced $3$-graph.
The minimum codegree assumption forces this reduced $3$-graph to have at most two tight components, by \cref{twocom}.
Lifting the reduced connections back to $H$ in \cref{equivalentclass}, we obtain a partition of the ordered vertex pairs into at most two robust connection components, denoted by $\mathcal{C}_1$ and $\mathcal{C}_2$.
Any two vertex-disjoint ordered pairs in the same class can be robustly connected through $A$.
Thus the only possible difficulty is that two endpoints may lie in different classes.

According to these two classes, \cref{dividecase} distinguishes the following alternatives.

\begin{itemize}
  \item[\textbf{(A1)}] \emph{Switcher pairs.}
    There exist two disjoint unordered pairs $\{u_1,v_1\}$ and $\{u_2,v_2\}$ such that
    \begin{align*}
      \left\{ (u_1,v_1),~(u_2,v_2) \right\} \subseteq \mathcal{C}_1
      \quad\text{and}\quad
      \left\{ (v_1,u_1),~(v_2,u_2) \right\} \subseteq \mathcal{C}_2.
    \end{align*}

  \item[\textbf{(A2)}] \emph{Switcher edges.}
    There exist two disjoint edges $x_1y_1z_1$ and $x_2y_2z_2$ in $E(H)$ such that
    \begin{align*}
      \left\{ (x_1,y_1),~(x_2,y_2) \right\} \subseteq \mathcal{C}_1
      \quad\text{and}\quad
      \left\{ (z_1,y_1),~(z_2,y_2) \right\} \subseteq \mathcal{C}_2.
    \end{align*}
    These are the two attachment pairs of the $3$-vertex paths $x_jy_jz_j$.

  \item[\textbf{(A3)}] \emph{Non-switcher case.}
    Neither (A1) nor (A2) occurs.
\end{itemize}

The first two alternatives may overlap, so in the formal proof we treat them as switcher cases and reserve (A3) for the case when no switcher of either type exists.

\sketchstep{3}{The switcher cases}
In the switcher cases, the obstruction caused by the two connection classes can be routed around.
These cases are proved in \cref{maincase1}.
The path-cover step relies on \cref{coverall}, the covering lemma of Han~\cite{Han22}, which gives two vertex-disjoint tight paths covering almost all vertices under minimum codegree strictly greater than $n/3$.
The common absorbing tool is the indexed absorber, stated as
\cref{lem:indexed-absorber}.  Roughly speaking, given a bounded partition
into closed parts and an RCC path-cover property, it constructs at most $r$
short disjoint absorbing paths.  This path family can absorb every small
leftover whose distribution among the parts is a nonnegative integer
combination of frequently occurring edge distributions.  In the switcher
cases, the $3$-uniform specialization \cref{absorber1} gives at most two
short paths that absorb every sufficiently small leftover set whose size is
divisible by $3$.  The replaceable segments given by
\cref{tightcycleexist} allow the leftover size to be adjusted modulo $3$.

We first set aside the reservoir set $A$, the absorbing paths, and the vertices used by the chosen switchers.
In the remaining $3$-graph, a path-cover argument gives two vertex-disjoint tight paths covering almost all available vertices.
The four path ends must then be connected in pairs in order to form a cycle.  The ordered pair used by a connector is the attachment pair, namely the reverse of the corresponding end pair.
Whenever two attachment pairs lie in the same robust connection class, we connect them directly \textit{through} $A$.
When two attachment pairs lie in different classes, a switcher pair or switcher edge changes the class of one path end.
After this joining step, all required endpoint pairs can be placed in the same classes, and the robust connections through $A$ merge the two long paths into an almost-spanning tight cycle.

\sketchstep{4}{The non-switcher case}
The non-switcher case is where uniform density becomes essential.
This case is proved in \cref{nonswitchercase}.
If no switcher exists, then the two robust connection components cannot mix freely.
By \cref{propcolor1,prop:coloritem}, after deleting a constant-sized exceptional set, every edge of $H$ has all three boundary pairs in the same class.
This gives an $A$-connected $2$-coloring of the shadow graph: a pair $xy$ receives color $i$ if both ordered pairs $(x,y)$ and $(y,x)$ belong to $\mathcal{C}_i$.
With respect to this coloring, the shadow of every remaining hyperedge is monochromatic.

This restriction gives useful structure.
By \cref{lem:monosetexist}, one of the two colors has a linear-sized one-sided vertex set: for some color $i\in\{1,2\}$, there is a set $V_0$ of size $\Omega(n)$ such that, outside the exceptional set, every pair incident to a vertex of $V_0$ has color $i$.
We then focus on this color.
Using $V_0$ as an anchor, \cref{lem:shadowlocaldense} shows that the uniform density of $H$ forces the color-$i$ shadow graph to be uniformly dense.
Indeed, if there were too few color-$i$ pairs between two large sets, then applying the uniform density condition with $V_0$ as the third set would give too few hyperedges.

Now consider the $3$-graph formed by the color-$i$ monochromatic triangles, namely the hyperedges of $H$ whose three shadow pairs all have color $i$.
\Cref{lem:longpath} applies weak regularity to this color-$i$ $3$-graph.
Using the density of the color-$i$ shadow graph from \cref{lem:shadowlocaldense} and the original codegree condition, it obtains an almost-perfect matching in the corresponding weak reduced graph.
Each matched triple of clusters contains long tight paths, and these paths cover almost all remaining vertices.

All end pairs of these paths have color $i$, hence they belong to the same robust connection class.
We can therefore connect the paths using only a small number of vertices from the reservoir $A$.
\Cref{absorber2}, the one-class application of
\cref{lem:indexed-absorber}, gives an absorbing path whose end pairs have
color $i$.  It absorbs every sufficiently small leftover whose size is
divisible by $3$ and which, when there are two closed parts, is balanced
between them.  Meanwhile, \cref{lem:monochromatic_c10} gives two color-$i$
replaceable segments for the modulo-$3$ adjustment.  After the long paths
have been connected, these tools absorb the remaining vertices and complete
a tight Hamiltonian cycle.

\section{Preliminaries}\label{pre}

We collect here the notation, definitions, and probabilistic tools used throughout the paper.

Let $H=(V,E)$ be a $3$-graph with vertex set $V$ and edge set $E\subseteq \binom{V}{3}$.
We write $|H|$ for the number of vertices of $H$, and $e(H)$ for the number of hyperedges of $H$.
For a vertex $v\in V$, we denote by $\deg_H(v)$ the \emph{vertex degree} of $v$, namely the number of edges of $H$ containing $v$.
For distinct vertices $u,v\in V$, we denote by $\deg_H(u,v)$ the \emph{codegree} of the pair $\{u,v\}$, namely the number of edges of $H$ containing both $u$ and $v$.
We write $\delta_1(H)\coloneqq\min_{v\in V}\deg_H(v)$ and
$\delta_2(H)\coloneqq\min_{\{u,v\}\in\binom{V}{2}}\deg_H(u,v)$ for the
\emph{minimum vertex degree} and the \emph{minimum codegree} of $H$.

We also use degrees relative to a vertex set.
For $A\subseteq V$ and $\{u,v\}\in\binom{V}{2}$, set
$\deg_A(u,v)\coloneqq|\{w\in A\colon uvw\in E\}|$.  Similarly, for
$v\in V$, let $\deg_A(v)$ be the number of pairs
$\{x,y\}\in\binom{A}{2}$ with $vxy\in E$.
Equivalently, $\deg_A(v)$ is the number of edges in the link graph of $v$ induced by $A$.

The \emph{shadow} of a $3$-graph $H=(V,E)$ is the graph
$\partial H\coloneqq\{\{u,v\}\in\tbinom{V}{2}\colon\deg_H(u,v)\ge1\}$.
When orientations matter, $\overrightarrow{\partial}_2(H)$ denotes the set of ordered pairs $(u,v)$ such that $\{u,v\}\in\partial H$.
Thus both $(u,v)$ and $(v,u)$ are contained in $\overrightarrow{\partial}_2(H)$ and are treated as different ordered pairs.
For $U\subseteq V(H)$, we write $H[U]$ for the induced subhypergraph of $H$ on $U$; that is, its edge set is $E(H)\cap\binom{U}{3}$.
For $A\subseteq V(H)$, we write $H\setminus A$ for the induced subhypergraph obtained by deleting the vertices of $A$ (and deleting all edges containing any vertices from $A$).

A \emph{tight path} $P$ in a $3$-graph $H$ is a sequence of distinct vertices
$v_1,\dots,v_m$, where $m\ge2$, such that
$v_iv_{i+1}v_{i+2}\in E(H)$ for every $1\le i\le m-2$.
We explicitly allow the degenerate $2$-vertex tight path; it has no hyperedges.
The \emph{end pairs} of $P$ are $(v_2,v_1)$ and $(v_{m-1},v_m)$.
For an ordered pair $p=(x,y)$, write $p^{\mathrm{rev}}\coloneqq(y,x)$.
If $p$ is an end pair of $P$, then $p^{\mathrm{rev}}$ is called the
\emph{attachment pair} at that end.  Thus a connector attached at that end
has $p^{\mathrm{rev}}$ as one of its ordered end pairs.
The \emph{internal vertices} of $P$ are the vertices of $P$ not contained in
its two end pairs; in particular, a tight path on at most four vertices has no
internal vertex outside its end pairs.
A \emph{tight cycle} of length $m$ is a cyclic ordering of distinct vertices $v_1,v_2,\dots,v_m$, with $m\ge 4$, such that $v_iv_{i+1}v_{i+2}\in E(H)$ for every $i\in\mathbb{Z}/m\mathbb{Z}$.
A tight cycle is a \emph{tight Hamiltonian cycle} if it spans all vertices of $H$.
We write $C_m^{(3)}$ for the $3$-uniform tight cycle of length $m$; thus $C_{10}^{(3)}$ has vertices $v_1,\dots,v_{10}$ and edges $v_iv_{i+1}v_{i+2}$ for $i\in\mathbb Z/10\mathbb Z$.
For positive integers $a,b,c$, let $K_{a,b,c}^{(3)}$ denote the complete
$3$-partite $3$-graph with vertex-class sizes $a,b,c$.
For a fixed $3$-graph $F$ and an integer $s\ge 1$, the \emph{$s$-blow-up} $F[s]$ is obtained by replacing each vertex $x\in V(F)$ by a set $X_x$ of $s$ new vertices and replacing each edge $xyz\in E(F)$ by all triples with one vertex in each of $X_x,X_y,X_z$.
In particular, $C_{10}^{(3)}[2]$ is the $2$-blow-up of $C_{10}^{(3)}$.
Two edges $e$ and $f$ of a $3$-graph are \emph{tightly connected} if there is a sequence of edges $e=e_1,e_2,\dots,e_s=f$ such that $|e_j\cap e_{j+1}|=2$ for every $j\in[s-1]$.
Such a sequence is called a \emph{pseudo-path} from $e$ to $f$.
The equivalence classes of this relation on $E(H)$ are called the \emph{tight components} of $H$.

For not necessarily disjoint sets $X_1,X_2,X_3\subseteq V(H)$, let $e_H(X_1,X_2,X_3)$ denote the number of ordered triples $(x_1,x_2,x_3)\in X_1\times X_2\times X_3$ such that $x_1,x_2,x_3$ are distinct and $x_1x_2x_3\in E(H)$.

\begin{defn}[$(\rho,d)$-dense]\label{def:dense}
  Let $\rho>0$ and $d\in(0,1]$.
  An $n$-vertex $3$-graph $H$ is $(\rho,d)$-dense if, for all vertex
  sets $X_1,X_2,X_3\subseteq V(H)$,
  $e_H(X_1,X_2,X_3)\ge d|X_1||X_2||X_3|-\rho n^3$.
\end{defn}

For a $2$-graph $G$ and vertex sets $X,Y\subseteq V(G)$, let $e_G(X,Y)$ denote the number of ordered pairs $(x,y)\in X\times Y$ with $xy\in E(G)$.

\begin{defn}[$(\rho,d)$-dense $2$-graph]\label{def:dense2graph}
  Let $\rho>0$ and $d\in(0,1]$.
  An $n$-vertex $2$-graph $G$ is called $(\rho,d)$-dense if, for every
  two vertex sets $X,Y\subseteq V(G)$,
  $e_G(X,Y)\ge d|X||Y|-\rho n^2$.
\end{defn}
\medskip

We use the following Chernoff bound; see, for example, \cite[pp.~326--327]{MR3524748}.

\begin{lemma}[Chernoff bound]\label{lem:chernoff}
  Let $X \sim \mathrm{B}(n,p)$ be a binomial random variable with parameters $n$ and $p$, and let $x>0$.
  Then
  \[
    \mathbb{P}[X\ge np+x]
    \le \exp\left(-\frac{x^2}{2(np+x/3)}\right)
    \quad\text{and}\quad
    \mathbb{P}[X\le np-x]
    \le \exp\left(-\frac{x^2}{2np}\right).
  \]
\end{lemma}

We also use the following bounded differences inequality of McDiarmid~\cite{MR1036755}.

\begin{lemma}[Bounded differences inequality]\label{lem:bounded_differences}
  Let $X_1,\dots,X_N$ be independent random variables, and let $f=f(X_1,\dots,X_N)$ be a real-valued function.
  Suppose that there are constants $c_1,\dots,c_N$ such that, for every $i\in[N]$ and all vectors $\mathbf{x}=(x_1,\dots,x_N)$ and $\mathbf{x}'=(x_1',\dots,x_N')$ with $x_j=x_j'$ for every $j\ne i$, it holds that $|f(\mathbf{x})-f(\mathbf{x}')|\le c_i$.
  Then, for every $t>0$,
  \[
    \mathbb{P}\left[|f-\mathbb{E}[f]|\ge t\right]
    \le
    2\exp\left(-\frac{2t^2}{\sum_{i=1}^N c_i^2}\right).
  \]
\end{lemma}

\section{Hypergraph regularity and regular slices} \label{section:regularity}
This section collects the regularity tools used to pass between $H$ and suitable reduced hypergraphs.
The regular slice lemma provides the reduced $3$-graphs that underlie the robust tight connectedness argument in \cref{sec:robusttight}.
The extension and connector lemmas then allow us to lift short connections found at the reduced level back to the original hypergraph.

\subsection{Regular complexes}

Let $\cP$ be a partition of a vertex set $V$ into vertex classes $V_1,\dots,V_s$.
A set $S\subseteq V$ is \emph{$\cP$-partite} if $|S\cap V_i|\le 1$ for every $i\in[s]$.
A hypergraph is \emph{$\cP$-partite} if all its edges are $\cP$-partite, and it is \emph{$s$-partite} if it is $\cP$-partite for some partition $\cP$ with $|\cP|=s$.

A hypergraph $\h$ is a \emph{complex} if, whenever $e\in E(\h)$ and $\emptyset\ne e'\subseteq e$, we have $e'\in E(\h)$.
We assume throughout that complexes have no isolated vertices.
Equivalently, every vertex of a complex appears as a $1$-edge.
A complex $\h$ is a \emph{$3$-complex} if all edges of $\h$ have size at most $3$.
The edges of size $i$ are called the $i$-edges of $\h$, and we write $\h_i$ for the underlying $i$-graph of $\h$.
For $U\subseteq V(\h)$, we write $\h[U]$ for the induced subcomplex of $\h$ on $U$, obtained by keeping all edges of $\h$ contained in $U$.
Given $s\ge 3$, a \emph{$(3,s)$-complex} is an $s$-partite $3$-complex.
More generally, for $i\le j$, an \emph{$(i,j)$-graph} is a $j$-partite $i$-graph.

Let $\h$ be a $\cP$-partite $3$-complex.
For a set $X$ of vertex classes, write $\bigcup X$ for the union of the classes in $X$.
Let $\binom{\cP}{i}$ denote the collection of all $i$-element subsets of the partition $\cP$, that is, all choices of $i$ distinct vertex classes.
If $i\le 3$ and $X\in\binom{\cP}{i}$, then $\h_X$ denotes the $i$-uniform part of $\h$ on these classes; that is, $\h_X$ is the subgraph of $\h_i$ induced by $\bigcup X$.
Thus $\h_X$ is an $(i,i)$-graph.
We also need the lower-dimensional part of $\h$ supported on the same classes.
We denote it by $\h_X^{<}$; its vertex set is $\bigcup X$, and its edges are all edges of $\h$ that lie in a proper subcollection of the classes in $X$:
\[
  E(\h_X^{<})\coloneqq \bigcup_{X'\subsetneq X} E(\h_{X'}).
\]
For example, if $X\in\binom{\cP}{3}$, then $\h_X^{<}$ consists of the vertices and pairs on the three classes, and is a $(2,3)$-complex.

Roughly speaking, relative density is the proportion of the $i$-sets supported by $\h_{i-1}$ that are edges of $\h_i$.  Regularity says that
this proportion changes little when it is measured on any sufficiently large collection of supported $i$-sets.

Given $i\ge 2$, let $\h_i$ be an $(i,i)$-graph and let $\h_{i-1}$ be an $(i-1,i)$-graph on the same vertex set, both $i$-partite with respect to the same partition.
We write $\K_i(\h_{i-1})$ for the family of all partite $i$-sets which form a copy of the complete $(i-1)$-graph on $i$ vertices in $\h_{i-1}$.
The \emph{density of $\h_i$ with respect to $\h_{i-1}$} is
\[
  d(\h_i\mid \h_{i-1})
  \coloneqq
  \begin{cases}
    \displaystyle \frac{|\K_i(\h_{i-1})\cap E(\h_i)|}{|\K_i(\h_{i-1})|}, & \text{if } |\K_i(\h_{i-1})|>0,\\[1ex]
    0, & \text{otherwise}.
  \end{cases}
\]
More generally, if ${\bf Q}=(Q_1,\dots,Q_r)$ is a collection of $r$ subhypergraphs of $\h_{i-1}$, we set
\[
  \K_i({\bf Q})\coloneqq \bigcup_{j\in[r]} \K_i(Q_j)
\]
and define
\[
  d(\h_i\mid {\bf Q})
  \coloneqq
  \begin{cases}
    \displaystyle \frac{|\K_i({\bf Q})\cap E(\h_i)|}{|\K_i({\bf Q})|}, & \text{if } |\K_i({\bf Q})|>0,\\[1ex]
    0, & \text{otherwise}.
  \end{cases}
\]

We say that $\h_i$ is \emph{$(d_i,\e,r)$-regular with respect to $\h_{i-1}$} if for every $r$-tuple ${\bf Q}$ with $|\K_i({\bf Q})|>\e|\K_i(\h_{i-1})|$, we have
\[
  d(\h_i\mid {\bf Q})=d_i\pm \e.
\]
Instead of $(d_i,\e,1)$-regularity, we simply write \emph{$(d_i,\e)$-regularity}.
We also say that $\h_i$ is \emph{$(\e,r)$-regular with respect to $\h_{i-1}$} if there exists $d_i\ge0$ such that $\h_i$ is $(d_i,\e,r)$-regular with respect to $\h_{i-1}$.
Given an $i$-graph $G$ with $V(G)\supseteq V(\h_{i-1})$, we say that $G$ is \emph{$(d_i,\e,r)$-regular with respect to $\h_{i-1}$} if the $i$-partite subgraph of $G$ induced by the vertex classes of $\h_{i-1}$ is $(d_i,\e,r)$-regular with respect to $\h_{i-1}$.

Finally, let $s\ge 2$ and let $\h$ be a $(2,s)$-complex with vertex partition $\cP$.
We say that $\h$ is \emph{$(d_2,\e,r)$-regular} if, for every $A\in\binom{\cP}{2}$, the graph $\h_A$ is $(d_2,\e)$-regular with respect to $(\h_A^{<})_1$.
If $s\ge 3$ and $\h$ is a $(3,s)$-complex with vertex partition $\cP$, we say that $\h$ is \emph{$(d,d_2,\e_3,\e,r)$-regular} if both of the following conditions hold.
\begin{enumerate}[label={\rm(\roman*)}]
  \item For every $A\in\binom{\cP}{2}$, either $\h_A$ is $(d_2,\e)$-regular with respect to $(\h_A^{<})_1$, or $d(\h_A\mid (\h_A^{<})_1)=0$.

  \item For every $A\in\binom{\cP}{3}$, either $\h_A$ is $(d,\e_3,r)$-regular with respect to $(\h_A^{<})_2$, or $d(\h_A\mid (\h_A^{<})_2)=0$.
\end{enumerate}

\medskip

We need the following restriction lemma, which says that regular complexes remain regular after restricting each vertex class to a sufficiently large subset.

\begin{lemma}[Restriction Lemma, {\cite[Lemma~4.1]{KRD}}]
  \label{lem:reg_res}
  Let $s,r,m$ be positive integers and let $\alpha,d_2,d,\e$, $\e_3>0$ be real numbers satisfying
  \[
    1/m \ll \min\{1/r,~\e\} \le \max\{1/r,~\e\} \ll \min\{\e_3,d_2\}
    \le \e_3 \ll \alpha \ll \min\{ d,~1/s \}.
  \]
  Let $\h$ be a $(d,d_2,\e_3,\e,r)$-regular $(3,s)$-complex with vertex classes $V_1,\dots,V_s$, each of size $m$.
  For each $i\in[s]$, let $V_i'\subseteq V_i$ satisfy $|V_i'|\ge \alpha m$.
  Then the induced subcomplex $\h'=\h[V_1'\cup\cdots\cup V_s']$ is $(d,d_2,\sqrt{\e_3},\sqrt{\e},r)$-regular.
\end{lemma}

\subsection{The regular slice lemma}

We use the version of the regularity lemma due to Allen, B\"ottcher, Cooley and Mycroft~\cite{ABCM}, called the \emph{regular slice lemma}.
The notation in this subsection follows~\cite{ABCM}, with simplifications since we only need $3$-graphs.
A similar form was previously used by Haxell, {\L}uczak, Peng, R\"odl, Ruci\'nski and Skokan~\cite{HLOORRS}.

Let $t_0,t_1\in\mathbb{N}$ and $\e>0$.
A $2$-complex $\J$ is \emph{$(t_0,t_1,\e)$-equitable} if the following hold.
\begin{enumerate}[label={\rm(\roman*)}]
  \item There is a partition $\cP$ of $V(\J)$ into $t$ parts of equal size, for some $t_0\le t\le t_1$, such that $\J$ is $\cP$-partite.
    We call $\cP$ the \emph{ground partition} of $\J$, and its parts the \emph{clusters} of $\J$.

  \item There exists $d_2\ge 1/t_1$ with $1/d_2\in\mathbb{N}$ such that $\J$ is $(d_2,\e,1)$-regular.
\end{enumerate}

Let $X\in\binom{\cP}{3}$.
We write $\widehat{\J}_X$ for the $(2,3)$-graph $(\J_X^{<})_2$.
A $3$-graph $G$ on $V(\J)$ is \emph{$(\e_3,r)$-regular with respect to $\widehat{\J}_X$} if there exists $d\ge 0$ such that $G$ is $(d,\e_3,r)$-regular with respect to $\widehat{\J}_X$.
We write $d^\ast_{\J,G}(X)$ for this density, and simply write $d^\ast(X)$ when $\J$ and $G$ are clear from the context.
Thus $d^\ast(X)$ measures the density of the edges of $G$ supported by the triad $\widehat{\J}_X$; equivalently, it records how densely $G$ fills the three clusters in $X$ relative to the pairs in the slice $\J$.

Roughly speaking, a regular slice is an equitable collection of regular pairs for which almost every triple of clusters also behaves regularly in the $3$-graph.

\begin{defn}[Regular slice {\cite{ABCM}}]
  Given $\e,\e_3>0$, $r,t_0,t_1\in\mathbb{N}$, a $3$-graph $G$, and a $2$-complex $\J$ on $V(G)$, we call $\J$ a \emph{$(t_0,t_1,\e,\e_3,r)$-regular slice for $G$} if $\J$ is $(t_0,t_1,\e)$-equitable and $G$ is $(\e_3,r)$-regular with respect to all but at most $\e_3\binom{t}{3}$ triples of clusters of $\J$, where $t$ is the number of clusters of $\J$.
\end{defn}

Given a regular slice $\J$ for a $3$-graph $G$, we summarize these densities in a weighted reduced graph.
In this graph, the vertices are not the original vertices of $G$, but the clusters of the ground partition of $\J$.
For each triple of clusters $X$, the weight assigned to $X$ is the relative density $d^\ast(X)$ of $G$ on the corresponding triad.
Thus the reduced graph forgets the individual vertices and remembers only the
edge density on each triple of clusters.

\begin{defn}[Weighted reduced $3$-graph {\cite{ABCM}}]
  Given a $3$-graph $G$ and a $(t_0,t_1,\e)$-equitable $2$-complex $\J$ on $V(G)$ with ground partition $\cP$, we define $R_{\J}(G)$ to be the complete weighted $3$-graph whose vertex set is $\cP$.
  Each triple of clusters $X\in\binom{\cP}{3}$ is assigned the weight $d^\ast_{\J,G}(X)$.
  When $\J$ is clear from the context, we write $R(G)$ instead of $R_{\J}(G)$.
\end{defn}

The regular slice lemma gives a slice $\J$ for which $R(G)$ reflects the codegree structure of $G$.
Let $G$ be an unweighted $3$-graph on $n$ vertices.
For a pair $S\in\binom{V(G)}{2}$, define its \emph{relative degree} in $G$ by
\[
  \overline{\deg}(S;G)\coloneqq \frac{\deg_G(S)}{n-2}.
\]
If $G$ is a weighted $3$-graph with weight function $d^\ast$, then
\begin{align*}
  \overline{\deg}(S;G) = \frac{1}{n-2} \sum_{\substack{e\in E(G),~S\subseteq e}} d^\ast(e).
\end{align*}
For a collection $\mathcal S$ of pairs, we write $\overline{\deg}(\mathcal S;G) \coloneqq \frac{1}{|\mathcal S|} \sum_{S\in\mathcal S}\overline{\deg}(S;G)$.

Let $G$ be a $3$-graph with distinct root vertices $v_1,\dots,v_\ell$, and let $H$ be a $3$-graph with specified distinct root vertices $x_1,\dots,x_\ell$.
Let $n_H(G;v_1,\dots,v_\ell)$ be the number of injective maps $\psi:V(H)\to V(G)$ which embed $H$ into $G$ and satisfy $\psi(x_j)=v_j$ for each $j\in[\ell]$.
Define
\[
  d_H(G;v_1,\dots,v_\ell)
  \coloneqq
  \frac{n_H(G;v_1,\dots,v_\ell)}
  {\binom{v(G)-\ell}{v(H)-\ell}(v(H)-\ell)!}.
\]
Let $H^{\mathrm{skel}}$ be the $2$-complex obtained from the down-closure of $H$ by deleting the roots $x_1,\dots,x_\ell$ and deleting all edges of size $3$.
Given a $(t_0,t_1,\e)$-equitable $2$-complex $\J$ on $V(G)$, let $n_H(G;v_1,\dots,v_\ell,\J)$ be the number of labeled rooted copies of $H$ in $G$ such that each vertex of $H^{\mathrm{skel}}$ lies in a distinct cluster of $\J$ and the image of $H^{\mathrm{skel}}$ is contained in $\J$.
Let $n'_{H^{\mathrm{skel}}}(\J)$ be the number of labeled copies of $H^{\mathrm{skel}}$ in $\J$ with each vertex embedded in a distinct cluster.
Set
\[
  d_H(G;v_1,\dots,v_\ell,\J)
  \coloneqq
  \frac{n_H(G;v_1,\dots,v_\ell,\J)}
  {n'_{H^{\mathrm{skel}}}(\J)}.
\]

If $\J'$ is a complex on a subset $U\subseteq V(G)$, we use the notation
$d_H(G;v_1,\dots,v_\ell,\J')$ for the corresponding \emph{local} rooted
copy density: all non-root vertices are required to lie in $U$, and the
normalizing denominator is the number of labeled copies of
$H^{\mathrm{skel}}$ supported by $\J'$.

We shall use the following consequence of the Regular Slice Lemma and its
strengthened form due to Allen, B\"ottcher, Cooley and
Mycroft~\cite[Lemmas~6 and~22]{ABCM}.

\begin{theorem}[Local regular slice with a prescribed reservoir]
  \label{thm:strengthened_regularity}
  For all $q,t_0\in\mathbb N$, $\nu,\e_3>0$, and all functions
  $r:\mathbb N\to\mathbb N$ and $\e:\mathbb N\to(0,1]$, there exist
  $T_0,T_1,n_1\in\mathbb N$ such that the following holds for every
  $n\ge n_1$ divisible by $T_1!$.

  Let $G$ be a $3$-graph on a vertex set $V$ of size $n$, let $\mathcal Q$
  be a partition of $V$ into at most $q$ equal-sized parts, and let $A_0$ be
  a union of parts of $\mathcal Q$ with $|A_0|\ge\nu n$.
  Then there is a regular slice $\mathcal J$ for $G$ whose ground partition
  refines $\mathcal Q$ and which has the following properties.

  Let $\mathcal J_A=\mathcal J[A_0]$, let $\mathcal P_A$ be its ground
  partition, and put $t_A=|\mathcal P_A|$.  Let $\mathcal K_A$ be the
  complete $2$-complex on $A_0$ with ground partition $\mathcal P_A$.
  Then $t_0\le t_A\le T_1$, and:
  \begin{enumerate}[label={\rm(\roman*)}]
    \item \textbf{Regularity.}
      The complex $\mathcal J$ is a
      $(T_0,T_1,\e(T_1),\e_3,r(T_1))$-regular slice for $G$.
      Moreover, $\mathcal J_A$ is a
      $(t_0,T_1,\e(T_1),\e_3,r(T_1))$-regular slice for $G[A_0]$.

    \item \textbf{Local codegree.}
      For every pair $Y$ of clusters of $\mathcal J_A$,
      \[
        \overline{\deg}\bigl(Y;R_{\mathcal J_A}(G[A_0])\bigr)
        =
        \overline{\deg}\bigl((\mathcal J_A)_Y;G[A_0]\bigr)
        \pm \e_3.
      \]

    \item \textbf{Local rooted counting.}
      Let $P=x_1x_2x_3x_4$ be the $4$-vertex tight path with roots
      $x_1,x_2$.  For every two distinct vertices $u,v\in V$,
      \[
        \left|
        d_P(G;v,u,\mathcal J_A)
        -d_P(G;v,u,\mathcal K_A)
        \right|<\e_3.
      \]
      The roots may lie outside $A_0$, whereas the two non-root vertices
      are required to lie in $A_0$.
  \end{enumerate}
\end{theorem}

\begin{remark}
  Apply the Regular Slice Lemma~\cite[Lemma~6]{ABCM} with the prescribed
  initial partition $\mathcal Q$, with a global error parameter
  $\e_3'\ll\nu^4\e_3$, and with a global lower bound
  $T_0\ge 2t_0/\nu$.
  Its ground partition refines $\mathcal Q$, so $A_0$ is a union of
  clusters.  Since all clusters have the same size, at least a $\nu$
  proportion of them lie in $A_0$; hence $t_A\ge\nu T_0\ge t_0$.

  The restriction of the equitable $2$-complex to these clusters remains
  equitable.  At most $\e_3'\binom{t}{3}$ cluster triples are irregular in
  the global slice, where $t$ is the total number of clusters.  Since
  $t_A\ge\nu t$, our choice of $\e_3'$ implies that at most
  $\e_3\binom{t_A}{3}$ triples inside $A_0$ are irregular.  This proves
  part~(i), after weakening the displayed error parameters if necessary.

  Part~(ii) is the degree-preserving conclusion in
  \cite[Lemma~6(b)]{ABCM}, applied with the set of clusters contained in
  $A_0$; this set is the union of parts of the initial partition.

  For part~(iii), use the strengthened regularity lemma
  \cite[Lemma~22]{ABCM} in the proof of the rooted-counting conclusion
  \cite[Lemma~6(c)]{ABCM}.  That proof averages the rooted copy density over
  coarse cells and then over random fine slices.  If the non-root vertices
  are restricted to $A_0$, the same argument is unchanged because $A_0$ is
  a union of parts of the initial partition and hence a union of coarse and
  fine clusters.  The normalizing denominator is
  $\Omega_\nu(n^2)$ for the present rooted path, so choosing the global
  error $\e_3'$ sufficiently small in terms of $\nu$ and $\e_3$ gives the
  stated local error.  This is exactly the displayed formula in~(iii).
\end{remark}

\subsection{The \texorpdfstring{$d$}{d}-reduced
\texorpdfstring{$3$}{3}-graph and the extension lemma}
\label{subsection:reducedgraph}

The weighted reduced graph $R_{\J}(G)$ records the relative densities of $G$ on all triples of clusters of a regular slice.
In the applications below, however, we only use triples whose triad is
regular and whose relative density is at least a fixed threshold $d$.
We therefore pass to the unweighted $d$-reduced $3$-graph, and then use the extension lemma to lift structures found at the cluster level back to the original hypergraph.

\begin{defn}[The $d$-reduced $3$-graph {\cite{ABCM}}]
  Let $G$ be a $3$-graph and let $\J$ be a $(t_0,t_1,\e,\e_3,r)$-regular slice for $G$.
  For $d>0$, the \emph{$d$-reduced $3$-graph} $R_d(G)$ is the $3$-graph whose vertices are the clusters of $\J$ and whose edges are all triples of clusters $X$ such that $G$ is $(\e_3,r)$-regular with respect to $\widehat{\J}_X$ and $d^\ast(X)\ge d$.
  The graph $R_d(G)$ depends on the slice $\J$, which will always be clear from the context.
\end{defn}

We need two reduced-graph consequences of the codegree assumption.  To avoid
confusion with the $(\rho,d)$-density of the original hypergraph, we do not
introduce a second density terminology.  Instead, we state explicitly that all
but a small exceptional family of cluster pairs have large codegree.

For a nonempty $3$-graph $H$, its \emph{minimum positive codegree} is
\[
  \delta_2^+(H)
  \coloneqq
  \min\left\{\deg_H(S):S\in\binom{V(H)}2,\ \deg_H(S)>0\right\}
  =
  \min\left\{\deg_H(S):S\in\partial H\right\}.
\]
Thus $\delta_2^+(H)$ is the minimum codegree among the pairs that lie in at
least one edge of $H$.

The next lemma says that, apart from a small exceptional family, cluster pairs
inherit the codegree condition of the original hypergraph.

\begin{lemma}[{\cite[Lemma~8.7]{HLS}}]\label{lem:Rdeg}
  Let $1/n\ll 1/t_1\le 1/t_0\ll 1$, let $r\in\mathbb{N}$, and let $\mu,d,\e,\e_3>0$.
  Suppose that $G$ is a $3$-graph on $n$ vertices with $\delta_2(G)\ge \mu n$.
  Let $\J$ be a $(t_0,t_1,\e,\e_3,r)$-regular slice for $G$ such that, for all pairs $Y$ of clusters of $\J$,
  \[
    \overline{\deg}(Y;R(G))
    =
    \overline{\deg}(\J_Y;G)\pm\e_3.
  \]
  If $t$ is the number of clusters of $\J$, then all but at most
  $3\sqrt{\e_3}\binom t2$ pairs $Y$ of clusters satisfy
  \[
    \deg_{R_d(G)}(Y)
    \ge
    \bigl(\mu-d-\e_3-\sqrt{\e_3}\bigr)(t-2).
  \]
\end{lemma}

The following lemma passes from this almost-everywhere codegree
condition to a minimum positive codegree condition.

\begin{lemma}[{\cite[Lemma~8.8]{HLS}}]
  \label{lem:codegree-cleaning}
  Let $n\ge 6$ and $0<\mu,\theta<1$.  Suppose that $H$ is an $n$-vertex
  $3$-graph in which all but at most $\theta\binom n2$ pairs $S$ satisfy
  \(
    \deg_H(S)\ge \mu(n-2).
  \)
  Then $H$ contains a spanning subgraph $H'$ such that
  \(
    \delta_2^+(H')
    \ge
    \bigl(\mu-8\theta^{1/4}\bigr)(n-2)
  \)
  and
  \(
    \left|\binom{V(H')}{2}\setminus\partial H'\right|
    \le
    \bigl(\theta+\theta^{1/4}\bigr)\binom n2.
  \)
\end{lemma}

Suppose that $G$ is a $(3,\ell)$-complex with vertex classes $V_1,\dots,V_\ell$, and that $H$ is a $(3,\ell)$-complex with vertex classes $X_1,\dots,X_\ell$.
We say that $G$ \emph{respects the partition of $H$} if, whenever $H$ contains an $i$-edge with vertices in $X_{j_1},\dots,X_{j_i}$, the complex $G$ contains an $i$-edge with vertices in $V_{j_1},\dots,V_{j_i}$.
A labeled copy of $H$ in $G$ is \emph{partition-respecting} if each vertex in $X_i$ is embedded into $V_i$.
We write $|H|_G$ for the number of labeled partition-respecting copies of $H$ in $G$.

The Extension Lemma says that if $G$ is regular and $H$ is an induced subcomplex of $H'$, then almost all labeled copies of $H$ in $G$ extend to many labeled copies of $H'$.

\begin{lemma}[Extension Lemma, {\cite[Lemma~25]{ABCM}}]\label{lem:ext}
  Let $\ell,r,t,t',n_0$ be positive integers with $t<t'$, and let $c,\beta,d_2,d,\e,\e_3>0$ be constants such that $1/d_2,1/d\in\mathbb{N}$ and
  \[
    1/n_0 \ll \min\{1/r,\e\} \le   \max\{1/r,\e\} \ll c \ll \min\{\e_3,d_2\}
    \le \e_3 \ll \min\{\beta,d,1/\ell,1/t'\}.
  \]
  Then the following holds for all integers $n\ge n_0$.
  Let $H'$ be a $(3,\ell)$-complex on $t'$ vertices with vertex classes $Y_1,\dots,Y_\ell$, and let $H$ be an induced subcomplex of $H'$ on $t$ vertices.
  Suppose that $G$ is a $(3,\ell)$-complex with vertex classes $V_1,\dots,V_\ell$, all of size $n$, which respects the partition of $H'$.
  Suppose further that the underlying $(2,\ell)$-complex of $G$ is $(d_2,\e,1)$-regular, and that for each $3$-edge $e$ of $H'$ with index $A\in\binom{[\ell]}{3}$, the $(3,3)$-graph $G_A$ is $(d',\e_3,r)$-regular with respect to $(G_A^{<})_2$ for some $d'\ge d$.
  Then all but at most $\beta |H|_G$ labeled partition-respecting copies of $H$ in $G$ can be extended to at least $c n^{t'-t}$ labeled partition-respecting copies of $H'$ in $G$.
\end{lemma}

We shall use the Extension Lemma through the following connector statement, which turns a pseudo-path in the reduced graph into short tight paths between almost all prescribed terminal pairs.

\begin{lemma}[Connector, {\cite[Lemma~5.1]{Han22}}]\label{lem:connector}
  Let $\ell,r,n_0$ be positive integers, and let $c,\beta,d_2,d,\e,\e_3>0$ be constants such that $1/d_2,1/d\in\mathbb{N}$, $\beta\le d_2/18$, and
  \[
    1/n_0 \ll \min\{1/r,\e\} \le \max\{1/r,\e\} \ll c \ll \min\{\e_3,d_2\}
    \le \e_3 \ll \min\{\beta,d\}.
  \]
  Then the following holds for all integers $n\ge n_0$.
  Suppose that $G$ is a $(3,\ell)$-complex with vertex partition $\mathcal P=\{V_1,\dots,V_\ell\}$, each of size $n$, such that the underlying $(2,\ell)$-complex of $G$ is $(d_2,\e,1)$-regular.
  Let $R$ be a $3$-graph on $\mathcal P$ such that, for each triple $T\in E(R)$, the graph $G_T$ is $(d',\e_3,r)$-regular with respect to $(G_T^{<})_2$ for some $d'\ge d$.

  Let $S_1,S_0\in\binom{\mathcal P}{2}$ and $X_1,X_0\in\binom{\mathcal P}{3}$ satisfy $S_i\subseteq X_i\in E(R)$ for $i=0,1$.
  Suppose that there is a pseudo-path in $R$ connecting $X_1$ and $X_0$, and that $X_0$ is contained in a copy of $K_4^-$ in $R$.
  Then, for all but at most $\beta n^4$ choices of labeled $2$-edges $(v_1,v_2)\in G_{S_1}$ and $(v_3,v_4)\in G_{S_0}$, there exists a tight path $P'$ with at most $15+\ell^3$ vertices whose ordered end pairs are $(v_2,v_1)$ and $(v_3,v_4)$.
\end{lemma}
\section{Robust tight connectedness}\label{sec:robusttight}

In this section we describe the type of connection that will be used repeatedly in the absorption argument.
The goal is to decide when two ordered end pairs can be joined by a short tight path using only vertices from a fixed reservoir set $A$, even after a small forbidden subset of $A$ has been removed.
Recall that for a $3$-graph $H$,
$
\overrightarrow{\partial}_2(H) \coloneqq \{(x,y)\colon x\ne y \text{ and } \{x,y\}\in \partial H\}.
$

Roughly speaking, two ordered pairs are robustly tightly connected if a short path joins them through the reservoir even after any small forbidden set has been removed.

\begin{defn}\label{def:robustconnected}
  Let $n\in\mathbb N$, let $\zeta\ge0$, let $G$ be an $n$-vertex
  $3$-graph, and let $A\subseteq V(G)$.
  Let $p=(u,v)$ and $q=(u',v')$ be ordered pairs with
  $\{u,v\}\cap\{u',v'\}=\emptyset$.
  We say that $p$ and $q$ are
  \emph{$(A,\zeta,t)$-robustly tightly connected} if, for every set
  \[
    F\subseteq A\setminus\{u,v,u',v'\}
    \quad\text{with}\quad |F|\le\zeta n,
  \]
  there is a tight path in $G$ with at most $t$ vertices, with ordered end
  pairs $p$ and $q$, and with all internal vertices in $A\setminus F$.
\end{defn}

The reversal in the next elementary gluing statement is essential.

\begin{lemma}[Gluing and closing lemma]\label{lem:gluing}
  Let $G$ be a $n$-vertex $3$-graph. Then the following holds.

  \begin{enumerate}[label={\rm(\roman*)}]
    \item Let $P$ and $Q$ be vertex-disjoint tight paths, and choose one
      end pair $p$ of $P$ and one end pair $q$ of $Q$.  Suppose that a tight
      path $C$ has ordered end pairs $p^{\mathrm{rev}}$ and
      $q^{\mathrm{rev}}$, and that its internal vertices are disjoint from
      $V(P)\cup V(Q)$.  Then $P\cup C\cup Q$ is a tight path whose end pairs
      are the two unchosen end pairs of $P$ and $Q$.

    \item Let $P$ be a tight path with two vertex-disjoint end pairs
      $p,q$.  If a tight path $C$ has ordered end pairs
      $p^{\mathrm{rev}},q^{\mathrm{rev}}$ and has internal vertices disjoint
      from $V(P)$, then $P\cup C$ is a tight cycle.
  \end{enumerate}
\end{lemma}

\begin{proof}
  For~(i), traverse $P$ towards the chosen end, then traverse $C$, and
  finally traverse $Q$ away from the chosen end.  The reversal of the
  attachment pairs makes the two overlaps occur in the correct order, so
  every three consecutive vertices form an edge.  Part~(ii) is the same
  argument with the two ends of one path identified cyclically through
  $C$.
\end{proof}

The proof of the connection-class lemma uses the following
positive-codegree formulation of Han's tight-component
lemma.

\begin{lemma}[{\cite[Lemma~5.2]{Han22}}]\label{twocom}
  Let $\theta\in (0,1)$. Let $H$ be an $n$-vertex $3$-graph with $|\partial_2(H)|\ge \theta\binom{n}{2}$ and
  \(
    \delta_2^+(H)\ge \frac n3.
  \)
  Then $H$ has at most two tight components.
\end{lemma}

The next lemma divides the ordered shadow into at most two
classes such that any two disjoint pairs in the same class remain connectable through $A$ after a small forbidden set is removed.

\begin{lemma}\label{equivalentclass}
  Given $\alpha,\eta>0$, there exists $\zeta_0>0$ such that, for every
  $0<\zeta\le\zeta_0$, there exist
  $n_0=n_0(\alpha,\eta,\zeta)$ and
  $L_{\rm RCC}=L_{\rm RCC}(\alpha,\eta,\zeta)\in\mathbb N$ with the
  following property.
  Let $H$ be an $n$-vertex $3$-graph with $n\ge n_0$, and let
  $A\subseteq V(H)$ satisfy $|A|\ge\eta n$ and
  \[
    \deg_A(u,v)\ge\left(\frac13+\alpha\right)|A|
    \qquad\text{for every }\{u,v\}\in\binom{V(H)}2.
  \]
  Then there is a map
  \[
    \kappa:\overrightarrow{\partial}_2(H)\longrightarrow\{1,2\}
  \]
  such that any two vertex-disjoint ordered pairs $p,q$ with
  $\kappa(p)=\kappa(q)$ are
  $(A,\zeta,L_{\rm RCC})$-robustly tightly connected.
  Consequently, among any three pairwise vertex-disjoint ordered pairs,
  two are $(A,\zeta,L_{\rm RCC})$-robustly tightly connected.
  We may moreover assume that $L_{\rm RCC}\le\zeta^2n$.
  We call each nonempty $\mathcal C_i=\kappa^{-1}(i)$ a
  \emph{robust connection component}, abbreviated as an \emph{RCC}.
  Whenever a surjective codomain is required, we discard the empty
  fibres and relabel the nonempty fibres as $[r]$, where
  $r\in\{1,2\}$; we continue to write $\kappa$ and $\mathcal C_i$ for
  the resulting map and its fibres.
\end{lemma}

\begin{proof}
  Put $m=|A|$.  Choose integer constants $t_0,q$, positive constants
  $\mu,c,\e_3,d_2,\beta_{\rm con},d$, and functions $r,\e$. Here $r \colon \mathbb N\to\mathbb N$ and $\e \colon \mathbb N\to(0,1]$.
  We choose these functions with uniform bounds in the hierarchy below.
  The relation $\mu\ll\min\{1/r,\e\}$ means that $\mu$ is sufficiently small compared with both $1/r(s)$ and $\e(s)$ for every value of $s$ used below, in particular for the value $s=T_1$ given by \cref{thm:strengthened_regularity}. These parameters are chosen so that
  \begin{multline*}
    1/t_0\ll1/q\ll\mu\ll\min\{1/r(\cdot),\e(\cdot)\}
    \le\max\{1/r(\cdot),\e(\cdot)\}\ll c\ll\e_3\\
    \ll\beta_{\rm con}\ll\min\{d_2,d\}
    \ll\min\{\alpha,\eta\}.
  \end{multline*}
  Here the inequalities involving $r$ and $\e$ are required uniformly at
  every argument used below.

  Apply \cref{thm:strengthened_regularity} with $\nu=\eta/3$ and let
  $T_0,T_1,n_{\rm reg}$ be the constants given there.  Fix
  $\xi=1/50$, and then choose $\zeta_0>0$ so small that
  \[
    \sqrt{\zeta_0}\ll
    \min\left\{\frac{\eta}{T_1},\,\xi d_2\right\}.
  \]
  For a given $0<\zeta\le\zeta_0$, take $n_0$ sufficiently large that all
  applications below are valid and
  $25+T_1^3\le\zeta^2n$.
  Set $L_{\rm RCC}=25+T_1^3$.

  We make the divisibility adjustment explicitly.  Put
  \[
    L=T_1!\left\lceil\frac{n}{qT_1!}\right\rceil,
    \qquad \widetilde n=qL,
    \qquad a=\left\lfloor\frac{|A|}{L}\right\rfloor.
  \]
  Add $\widetilde n-n<qT_1!$ isolated vertices, all outside $A$, choose
  $A_0\subseteq A$ with $|A_0|=aL$, and partition the enlarged vertex set
  $\widetilde V$ into $q$ parts of size $L$, with exactly $a$ parts forming
  $A_0$.  Since $1/q\ll\mu\ll\eta$ and $n$ is sufficiently large, we have
  $|A\setminus A_0|<L\le\mu n$ and $\widetilde n-n\le\mu n$.
  Moreover $\widetilde n$ is divisible by $T_1!$, and the resulting partition
  $\mathcal Q$ has every part contained in $A_0$ or in
  $\widetilde V\setminus A_0$.
  Let $\widetilde H$ be the enlarged $3$-graph, and put $m_0=|A_0|$.
  Then $m_0\ge(\eta/3)|\widetilde V|$, and
  \[
    \deg_{A_0}(x,y)\ge
    \left(\frac13+\frac{3\alpha}{4}\right)m_0
    \qquad\text{for every }\{x,y\}\in\binom{V(H)}2.
  \]

  Apply \cref{thm:strengthened_regularity} to $\widetilde H$ with initial
  partition $\mathcal Q$.  Let $\mathcal J_A=\mathcal J[A_0]$, let
  $\mathcal P$ be its ground partition, put $t=|\mathcal P|$, and let
  $m_*=m_0/t$ be the cluster size.  Thus $t_0\le t\le T_1$.
  Let $R_d$ be the $d$-reduced graph of $H[A_0]$ with respect to
  $\mathcal J_A$.

  By the local codegree conclusion and \cref{lem:Rdeg}, all but at most
  $3\sqrt{\e_3}\binom t2$ pairs of clusters have degree at least
  \(
    \left(\frac13+\frac{2\alpha}{3}\right)(t-2)
  \)
  in $R_d$, where we used the hierarchy to absorb the regularity and
  density losses.  Apply \cref{lem:codegree-cleaning} with
  \(
    \mu=\frac13+\frac{2\alpha}{3}
    ~\text{and}~
    \theta=3\sqrt{\e_3}.
  \)
  We obtain a spanning subgraph $R\subseteq R_d$ such that
  \(
    \delta_2^+(R)
    \ge
    \left(\frac13+\frac\alpha2\right)(t-2)
  \)
  and
  \(
    \left|\binom{\mathcal P}{2}\setminus\partial R\right|
    \le \Theta\binom t2,
    ~
    \Theta=3\sqrt{\e_3}+(3\sqrt{\e_3})^{1/4}.
  \)
  Since $1/t_0\ll\alpha$ and $t\ge t_0$, the positive codegree bound is at
  least $t/3$.  Hence \cref{twocom} implies that $R$ has at most two tight
  components, denoted by $C_1$ and $C_2$.

  We next associate an edge of $R$ with every ordered pair
  $p=(u,v)\in\overrightarrow{\partial}_2(H)$.
  Let $P=x_1x_2x_3x_4$ be the $4$-vertex tight path, rooted at $x_1,x_2$,
  and root it at $x_1=v,x_2=u$.
  The codegree assumption gives at least
  $(1/3+\alpha/2)^2m_0^2$ labeled choices of $x_3,x_4\in A_0$ for which
  $vux_3x_4$ is a tight path, up to a negligible $O(m_0)$ error.
  Let $\mathcal K_A$ be the complete local $2$-complex on $A_0$ with ground
  partition $\mathcal P$.  Since $t\ge t_0$, choices with $x_3,x_4$ in the
  same cluster are negligible, and hence
  \[
    d_P(\widetilde H;v,u,\mathcal K_A)
    \ge\frac19+\frac\alpha4.
  \]
  The local rooted-counting conclusion of
  \cref{thm:strengthened_regularity} gives
  \[
    d_P(\widetilde H;v,u,\mathcal J_A)
    \ge\frac19+\frac\alpha5.
  \]
  It follows that at least
  $(1/9-\e)t(t-1)d_2m_*^2$ such rooted paths have their terminal pair in a
  $2$-edge of $\mathcal J_A$.

  All but at most $\Theta\binom t2$ pairs of clusters have positive degree
  in $R$.  The rooted paths terminating in the remaining cluster pairs
  number at most
  $\Theta\binom t2(d_2+\e)m_*^2$, which is less than half the preceding
  lower bound.  Therefore there is a cluster pair $S(p)$ of positive degree
  in $R$ supporting at least $\xi d_2m_*^2$ rooted paths starting at $p$.
  Let $\mathcal T(p)$ be the family of their ordered terminal pairs:

  \[
    \mathcal T(p)
    \coloneqq
    \bigl\{(x_3,x_4):vux_3x_4\text{ is one of these rooted paths
    and }x_3x_4\text{ is supported on }S(p)\bigr\}.
  \]
  Fix an edge $X(p)\in E(R)$ containing $S(p)$.

  Every edge of $R$ is contained in a copy of $K_4^-$ in $R$.  Indeed, if
  $X=abc\in E(R)$, then each of $ab,ac,bc$ has degree at least
  $(1/3+\alpha/2)(t-2)$. So some outside vertex belongs to at least two of them. Together with $abc$, this forms a copy of $K_4^-$.

  \begin{claim}
    Let $p=(u,v)$ and $q=(u',v')$ be vertex-disjoint ordered pairs such
    that $X(p)$ and $X(q)$ lie in the same tight component of $R$. Then $p$
    and $q$ are $(A,\zeta,L_{\rm RCC})$-robustly tightly connected.
  \end{claim}

  \begin{proofclaim}
    Let $F\subseteq A\setminus\{u,v,u',v'\}$ satisfy $|F|\le\zeta n$, and put
    \[
      F^\ast=F\cup\bigl(A_0\cap\{u,v,u',v'\}\bigr).
    \]
    Remove from each terminal family the rooted paths containing a root of
    the other ordered pair; explicitly, put
    \[
      \mathcal T_p
      =\bigl\{(a,b)\in\mathcal T(p):
      \{a,b\}\cap\{u',v'\}=\emptyset\bigr\},
      \qquad
      \mathcal T_q
      =\bigl\{(c,d)\in\mathcal T(q):
      \{c,d\}\cap\{u,v\}=\emptyset\bigr\}.
    \]
    This removes
    only $O(m_*)$ members from either family.

    Since $m_*\ge\eta n/(4T_1)$ and
    $\sqrt\zeta\ll\eta/T_1$, every cluster retains at least
    $(1-\sqrt\zeta)m_*$ vertices after deleting $F^\ast$, for sufficiently
    large $n$.  For each of $p$ and $q$, at most
    $2\sqrt\zeta m_*^2+O(m_*)$ further supported rooted paths are destroyed,
    so at least $(\xi/2)d_2m_*^2$ remain.

    In every cluster choose uniformly a subset of the common size
    $m_\circ=\lfloor(1-\sqrt\zeta)m_*\rfloor$ from the vertices outside
    $F^\ast$.  For either terminal family the expected retained size is a fixed
    positive proportion of its original size.  The bounded-differences
    inequality for hypergeometric sampling shows that each retained count is
    at least half its expectation with probability
    $1-\exp(-\Omega(m_*))$.  A union bound for the two terminal families
    therefore gives one common choice for which both have size at least
    $(\xi/4)d_2m_\circ^2$.

    By \cref{lem:reg_res}, the relevant complexes restricted to these equal
    subclusters remain regular.  The two terminal families yield at least
    $(\xi d_2/4)^2m_\circ^4$ ordered pairs of terminal edges.  After deleting
    the $O(m_\circ^3)$ choices with intersecting terminal edges, this is still
    more than $2\beta_{\rm con}m_\circ^4$.
    The edges $X(p)$ and $X(q)$ are joined by a pseudo-path in $R$, and
    $X(q)$ lies in a copy of $K_4^-$.  Thus \cref{lem:connector} gives a
    connector between a suitable pair of terminal edges, entirely inside the
    selected subclusters and hence avoiding $F^\ast$.

    To make the orientation explicit, write the two selected terminal end
    pairs as $r_p=(a,b)$ and $r_q=(c,d)$.  Apply \cref{lem:connector} to
    the first terminal edge labeled $(a,b)$ and the second labeled $(d,c)$.
    The connector then has ordered end pairs
    $(b,a)=r_p^{\mathrm{rev}}$ and $(d,c)=r_q^{\mathrm{rev}}$.
    Concatenating the rooted path from $p$ to $r_p$, this connector, and the
    reverse of the rooted path from $q$ to $r_q$ gives a tight path with
    ordered end pairs $p,q$, all internal vertices in $A_0\setminus F$, and
    at most $25+t^3\le L_{\rm RCC}$ vertices.
  \end{proofclaim}
  Define $\kappa(p)=i$ when $X(p)$ lies in the tight component $C_i$ of
  $R$; if $R$ has only one tight component, take $\mathcal C_2=\emptyset$.
  The claim proves the asserted robust connectedness within every
  nonempty fibre $\mathcal C_i=\kappa^{-1}(i)$.
  Finally, the second assertion follows from the pigeonhole
  principle applied to three pairwise vertex-disjoint ordered pairs.
\end{proof}

The robust connection components given by \cref{equivalentclass} lead to the following case distinction.
Informally, either two disjoint reversed pairs switch from one RCC to the
other, two disjoint edges provide such switches, or no switch of either type
exists.

\begin{proposition}\label{dividecase}
  Given $\alpha,\eta>0$, there exists $\zeta_0>0$ such that, for every
  $0<\zeta\le\zeta_0$, there exists $n_0=n_0(\alpha,\eta,\zeta)$ with the
  following property.
  Let $H$ be an $n$-vertex $3$-graph with $n\ge n_0$ and
  $\delta_2(H)\ge(1/3+\alpha)n$.
  Let $A\subseteq V(H)$ satisfy $\eta n\le |A|\le 2\eta n$ and
  \[
    \deg_A(u,v)\ge\left(\frac13+\frac\alpha2\right)|A|
    \qquad\text{for every }\{u,v\}\in\binom{V(H)}2.
  \]
  Let $\mathcal C_1,\mathcal C_2$ be the RCCs given by
  \cref{equivalentclass} with codegree slack $\alpha/2$ and connection
  parameter $\zeta$.  Then one of the
  following alternatives occurs.
  \begin{enumerate}[label=\rm(A\arabic*)]
    \item\label{case1} \textbf{Switcher-pair case.}
      There exist two vertex-disjoint unordered pairs
      $\{u_1,v_1\}$ and $\{u_2,v_2\}$ such that
      $(u_1,v_1),(u_2,v_2)\in\mathcal C_1$ and
      $(v_1,u_1),(v_2,u_2)\in\mathcal C_2$.

    \item\label{case2} \textbf{Switcher-edge case.}
      There exist two vertex-disjoint edges $x_1y_1z_1$ and $x_2y_2z_2$
      such that
      $(x_1,y_1),(x_2,y_2)\in\mathcal C_1$ and
      $(z_1,y_1),(z_2,y_2)\in\mathcal C_2$.

    \item\label{case3} \textbf{Non-switcher case.}
      Neither \textup{(A1)} nor \textup{(A2)} occurs. \qed
  \end{enumerate}
\end{proposition}

The common absorption argument for all three alternatives is given in
\cref{sec:rcc-absorption}.
The switcher-pair and switcher-edge cases are handled in \cref{switcher-case}.
The non-switcher case requires additional structural information from the absence of both types of switchers, and is handled in \cref{non-switcher-case}.

\section{Absorption for robust component classes}
\label{sec:rcc-absorption}

In a typical absorption argument for Hamiltonian cycles, one constructs a single absorbing path, covers almost all remaining vertices by a bounded collection of paths, and then joins these paths using suitable connectors. Such an approach usually relies on a sufficiently rich global tight-connectivity structure. In our setting, however, the hypergraph may have several robust connection components, so the exposed ends of a single absorbing path need not be compatible with all later connections. We therefore retain several absorbing paths and keep track of the RCC containing each of their ordered end pairs; these paths are then assembled with the path cover using robust connectors and, when necessary, switchers.


Recall that
a tight path is a sequence of distinct vertices
\(v_1\dots v_\ell\), where
\(\ell\ge k-1\), such that every \(k\) consecutive vertices form an edge;
paths of order \(k-1\) are allowed.  Put
\(\partial_{k-1}H=\{S\in\binom{V(H)}{k-1}:S\subseteq e\text{ for some }
e\in E(H)\}\), and let \(\overrightarrow{\partial}_{k-1}(H)\) contain all
orderings of the members of $\partial_{k-1}H$.  For \(p=(x_1,\dots,x_{k-1})\), put
\(p^{\rm rev}=(x_{k-1},\dots,x_1)\).  We call a set
\(\Omega\subseteq\overrightarrow{\partial}_{k-1}(H)\)
\emph{reversal-closed} if \(p^{\rm rev}\in\Omega\) for every
\(p\in\Omega\).  The ordered ends of
\(v_1\dots v_\ell\) are \((v_{k-1},\dots,v_1)\) and
\((v_{\ell-k+2},\dots,v_\ell)\); its internal vertices are those outside
the two ends.  A tight \(pq\)-path has ordered ends \(p,q\), and a path is
\(\Omega\)-ended if both ends belong to \(\Omega\).

We first define reservoir connections within an RCC.
Roughly speaking, any two disjoint ordered $(k-1)$-sets in the same class can
be joined by a short tight path whose internal vertices lie in $A\setminus F$,
where $F$ is any forbidden set of size $o(n)$.

\begin{defn}
  \label{def:rcc-connectable}
  Let \(A\subseteq V(H)\), let
  \(\Omega\subseteq\overrightarrow{\partial}_{k-1}(H)\) be
  reversal-closed, and let \(\kappa:\Omega\to[r]\) be surjective.
  For \(\chi>0\) and \(L\in\mathbb N\), we call
  \((\Omega,\kappa)\) \emph{\((A,\chi,L)\)-RCC-connectable} if,
  whenever \(p,q\in\Omega\) are vertex-disjoint with
  \(\kappa(p)=\kappa(q)\), every set
  \(F\subseteq A\setminus(V(p)\cup V(q))\) of size at most \(\chi n\)
  leaves a tight \(pq\)-path on at most \(L\) vertices whose internal
  vertices lie in \(A\setminus F\).
\end{defn}

Notice that reversal-closedness is required only for
\(\Omega\); the individual fibres of \(\kappa\) need not be
reversal-closed, and it is permitted that
\(\kappa(p)\ne\kappa(p^{\rm rev})\).
The next cover property says that any small family of disjoint
$\Omega$-ended paths can be extended to at most $r$ larger paths while avoiding some vertices.

\begin{defn}[RCC path-system cover]
  \label{def:rcc-cover}
  Let \(A,Z\subseteq V(H)\), let
  \(\Omega\subseteq\overrightarrow{\partial}_{k-1}(H)\) be
  reversal-closed, and let \(\kappa:\Omega\to[r]\) be surjective.  For
  \(\theta,\varepsilon>0\) and \(\omega\ge0\), we say that
  \((\Omega,\kappa)\) has the
  \emph{\((A,Z;\theta,\varepsilon,\omega)\)-RCC cover property}
  if every nonempty family \(\mathcal Q\) of vertex-disjoint
  \(\Omega\)-ended tight paths in \(H-(A\cup Z)\) with
  \(|V(\mathcal Q)|\le\theta n\) can be extended in \(H-Z\) to a family
  \(\mathcal P\) of \(1\le|\mathcal P|\le r\) vertex-disjoint tight paths
  such that each member of \(\mathcal Q\) is a subpath of exactly
  one member of \(\mathcal P\), every final end is from
  \(\mathcal Q\), \(|V(\mathcal P)|\le\varepsilon n\), and
  \(|A\cap V(\mathcal P)|\le\omega n\).
\end{defn}

Next, we state the definition of absorbers.

\begin{defn}
  \label{defn:Sabsorb}
  For \(S\subseteq V(H)\), an \emph{\(S\)-absorber} is a
  vertex-disjoint tight-path family \(\mathcal Q\), disjoint from \(S\),
  which can be replaced by a vertex-disjoint tight-path family
  \(\mathcal Q'\) with the same respective ordered ends and
  \(V(\mathcal Q')=V(\mathcal Q)\cup S\).

  Reachability for two vertices means that the same witness set can be completed with either
  vertex to path families having the same ordered ends.
  For \(A\subseteq V(H)\), \(\beta>0\), \(i\in\mathbb N\), and
  reversal-closed
  \(\Omega\subseteq\overrightarrow{\partial}_{k-1}(H)\), distinct vertices
  \(u,v\) are
  \emph{\((\beta,i;\Omega)\)-reachable in \(H-A\)} if there are at least
  \(\beta n^{(2k-1)i-1}\) sets
  \(T\subseteq V(H)\setminus(A\cup\{u,v\})\) of size
  \((2k-1)i-1\) such that both \(T\cup\{u\}\) and \(T\cup\{v\}\) decompose
  into \(i\) vertex-disjoint \(k\)-edge, \(\Omega\)-ended tight paths,
  paired with the same ordered ends.  A set is
  \emph{\((\beta,i;\Omega)\)-closed} if each pair of its vertices is
  reachable.  We omit \(\Omega\) when it is the full ordered shadow, and
  write \(\widetilde N_{A,\beta,i}(v)\) for the vertices
  \((\beta,i)\)-reachable to \(v\) in \(H-A\).

  The index vector records how a set is distributed among the parts of a
  partition.  For a partition \(\mathcal P=\{X_1,\dots,X_m\}\) and \(\mu>0\), put
  \(\mathbf i_{\mathcal P}(S)=(|S\cap X_1|,\dots,|S\cap X_m|)\).
  A subgraph \(J\subseteq H\) is \emph{\(\Omega\)-supported} if every
  ordering of every member of \(\partial_{k-1}J\) belongs to \(\Omega\).
  Let \(\mathcal I_\mu(J,\mathcal P)\) be the set of nonnegative
  \(m\)-vectors of coordinate sum \(k\) represented by at least
  \(\mu n^k\) edges of \(J\).  For
  \(\mathcal I\subseteq\mathbb Z_{\ge0}^m\), define
  \(
    \mathsf M(\mathcal I)
    \coloneqq
    \left\{\sum_{\mathbf v\in\mathcal I}a_{\mathbf v}\mathbf v:
    a_{\mathbf v}\in\mathbb Z_{\ge0}\right\}
  \)
  to be the set of all nonnegative integer combinations of the vectors in
  \(\mathcal I\).
\end{defn}

\begin{lemma}
  \label{lem:indexed-absorber}
  Fix \(k,m,r\in\mathbb N\) with \(k\ge3\), and let
  \(1/n\ll\nu\ll\gamma\ll
  \min\{\theta,\mu,\lambda,\varepsilon,1/k,1/m,1/r\}\).
  Let \(H\) be an \(n\)-vertex \(k\)-graph, let \(\omega\ge0\), and let
  \(A,Z\subseteq V(H)\) with \(|Z|\le\nu n\).
  Put \(H_0=H-(A\cup Z)\), and let
  \(\kappa:\Omega\to[r]\) be a surjection on a reversal-closed
  \(\Omega\subseteq\overrightarrow{\partial}_{k-1}(H)\).
  Suppose that \((\Omega,\kappa)\) has the
  \((A,Z;\theta,\varepsilon,\omega)\)-RCC cover property.
  Let \(\mathcal P=\{X_1,\dots,X_m\}\) be a partition of \(V(H)\) into
  \((\lambda,2;\Omega)\)-closed parts in \(H-A\), let
  \(J\subseteq H_0\) be
  \(\Omega\)-supported, and let
  \(\emptyset\ne\mathcal I\subseteq\mathcal I_\mu(J,\mathcal P)\).
  Then there are \(1\le t\le r\) vertex-disjoint \(\Omega\)-ended tight
  paths \(P_1,\dots,P_t\) in \(H-Z\) such that
  \begin{enumerate}[label=\rm(\roman*)]
    \item \(|V(P_1\cup\cdots\cup P_t)|\le\varepsilon n\) and
      \(|A\cap V(P_1\cup\cdots\cup P_t)|\le\omega n\);
    \item if
      \(L\subseteq V(H)\setminus V(P_1\cup\cdots\cup P_t)\) satisfies
      \(|L|\le\gamma n\) and
      \(\mathbf i_{\mathcal P}(L)\in\mathsf M(\mathcal I)\), the family $\{P_1,\dots P_t\}$ is an $L$-absorber.
  \end{enumerate}
\end{lemma}

\begin{remark}[Comparison with Han--Shu--Wang]
  \label{rem:hsw-absorption-comparison}
  The proof follows Han, Shu and Wang~\cite[Lemma~3.1] {HSW26}: count
  bounded local absorbers, choose a sparse subfamily, connect their ordinary
  forms, and then replace disjoint absorbers as needed.
  Here \((k-1)\)-tuple
  ends are indexed by
  \(r\) RCCs, the path family has at most \(r\) paths, and the condition of being
  a nonnegative integer combination of valid index vectors describes the
  admissible leftovers; the local counting condition is kept separate from
  the path-cover condition.
\end{remark}

The auxiliary lemmas and the proof of \cref{lem:indexed-absorber} are given
in Appendix~\ref{app:indexed-absorber}.
We now specialize to the \(3\)-uniform case by the above lemma.
\begin{coro}[The \(3\)-uniform specialization]
  \label{absorber1}
  Let \(1/n\ll\rho\ll\zeta\ll\eta\ll\nu\ll\gamma\ll\beta
  \ll\min\{\alpha,d\}\).  Let \(H\) be an \(n\)-vertex
  \((\rho,d)\)-dense \(3\)-graph with
  \(\delta_2(H)\ge(1/3+\alpha)n\).  Let \(A\subseteq V(H)\) satisfy
  \(\eta n\le|A|\le2\eta n\) and
  \(\deg_A(u,v)\ge(1/3+\alpha/2)|A|\) for every pair \(uv\).
  Let \(\kappa:\overrightarrow{\partial}_2(H)\to[r]\), where
  \(r\in\{1,2\}\), be the surjective RCC map from
  \cref{equivalentclass}.  If \(Z\subseteq V(H)\)
  satisfies \(|Z|\le\nu n\) and \(|A\cap Z|\le\zeta n/2\), then there are
  \(1\le t\le r\)
  vertex-disjoint tight paths \(P_1,\dots,P_t\) in \(H-Z\), using at most
  \(\beta n\) vertices in total and at most \(\zeta^2n\) vertices of
  \(A\); when \(r=2\), they lie in \(H-(A\cup Z)\).
  Moreover, for every
  \(S\subseteq V(H)\setminus V(P_1\cup\cdots\cup P_t)\) with
  \(|S|\le\gamma n\) and \(|S|\in3\mathbb N\), the family
  \(\{P_1,\dots,P_t\}\) is an \(S\)-absorber.
\end{coro}

The proof of \cref{absorber1} is also given in
Appendix~\ref{app:indexed-absorber}.

\section{The switcher cases}\label{switcher-case}

The main result of this section handles both switcher alternatives from
\cref{dividecase}.

\begin{lemma}\label{maincase1}
  Let
  $1/n\ll\rho\ll\zeta\ll\eta\ll\nu\ll\gamma\ll\beta
  \ll\min\{\alpha,d\}$.
  Let $H$ be an $n$-vertex $(\rho,d)$-dense $3$-graph with
  $\delta_2(H)\ge(1/3+\alpha)n$.
  Let $A\subseteq V(H)$ satisfy $\eta n\le |A|\le 2\eta n$ and
  $\deg_A(u,v)\ge(1/3+\alpha/2)|A|$ for every pair $uv$.
  Let $\mathcal C_1,\mathcal C_2$ be the RCCs with
  parameter $\zeta$, and suppose that Case~\ref{case1} or
  Case~\ref{case2} occurs.
  Then $H$ contains a tight Hamiltonian cycle.
\end{lemma}

\subsection{Path-cover and divisibility adjustment}

We use the following two-path cover lemma of Han~\cite{Han22}.

\begin{lemma}[{\cite[Lemma~3.2]{Han22}}]\label{coverall}
  Given $\alpha,\gamma>0$, there exists $n_0\in\mathbb{N}$ such that the following holds for all $n\ge n_0$.
  Let $H$ be an $n$-vertex $3$-graph with $\delta_2(H)\ge \left(\frac{1}{3}+\alpha\right)n$.
  Then $H$ contains two vertex-disjoint tight paths covering at least $(1-\gamma)n$ vertices.
\end{lemma}

The absorber from \cref{absorber1} can absorb only leftover sets whose sizes are divisible by $3$.
For the later divisibility adjustment, we use the following consequence of the
active-shadow blow-up lemma proved in \cref{app:monochromatic-c10}.

\begin{theorem}\label{tightcycleexist}
  Given any $\gamma>0$, there exists $n_0$ such that for every $n\ge n_0$ and every $n$-vertex $3$-graph $H$ with $\delta_2(H)\ge (1/3+\gamma)n$, the graph $H$ contains a copy of $C_{10}^{(3)}[2]$, that is, the $2$-blow-up of the tight cycle of length $10$.
\end{theorem}

\begin{proof}
  Apply \cref{lem:active-shadow-c10-blowup} with $G=H$, $R=K_n$, the
  given value of $\gamma$, and any fixed $\tau\in(0,1)$.  Since
  $\delta_2(H)>0$, we have $\partial H=K_n$; moreover $K_n$ is
  $(\rho,\tau)$-dense for every $\rho>0$.  The conclusion of that lemma is a copy of
  $C_{10}^{(3)}[2]$ in $H$.
\end{proof}

\subsection[Proof of the switcher case lemma]{Proof of \cref{maincase1}}

\begin{proof}[Proof of \cref{maincase1}]
  Let $\kappa(p)=j$ when $p\in\mathcal C_j$.
  Let $W_{\rm sw}$ be the set of switcher vertices, so $|W_{\rm sw}|\le6$.

  \proofstep{1}{Reserving replaceable tight-cycle segments}
  Since $\eta\ll\alpha$, two applications of \cref{tightcycleexist} give
  vertex-disjoint copies $C_1,C_2$ of $C_{10}^{(3)}[2]$ outside
  $A\cup W_{\rm sw}$.
  For $j\in[2]$, write
  \[
    V(C_j)=\{x_1^j,\dots,x_{10}^j,y_1^j,\dots,y_{10}^j\},
  \]
  and set
  \[
    R_j=x_1^jx_2^jx_3^jx_4^jx_5^j,
    \qquad
    R_j^+=x_1^jx_2^jy_3^j\cdots y_{10}^jy_1^jy_2^jx_3^jx_4^jx_5^j.
  \]
  These paths have the same ordered end pairs and
  $|V(R_j^+)|-|V(R_j)|=10$.

  \proofstep{2}{Constructing the absorber and the path cover}
  Apply \cref{absorber1} with reservoir $A$, RCC map $\kappa$, and
  \[
    Z=W_{\rm sw}\cup V(C_1\cup C_2).
  \]
  We obtain a family
  $\mathcal P_{\rm abs}=\{P_1,\dots,P_t\}$ of absorbing paths, where
  $t\le2$, $|V(\mathcal P_{\rm abs})|\le\beta n$, and
  $V(\mathcal P_{\rm abs})\cap A=\emptyset$.
  Delete
  \[
    A\cup W_{\rm sw}\cup V(C_1\cup C_2)\cup V(\mathcal P_{\rm abs})
  \]
  and call the remaining induced graph $H'$.  The hierarchy gives
  $\delta_2(H')\ge(1/3+\alpha/3)|V(H')|$, so \cref{coverall} gives two
  disjoint tight paths $B_1,B_2$ covering all but at most $\gamma n/10$
  vertices of $H'$.

  \proofstep{3}{Merging the path family into two paths}
  Start with the at most six disjoint paths
  \[
    \mathcal P_{\rm abs}\cup\{B_1,B_2,R_1,R_2\}.
  \]
  Throughout the connecting procedure, the forbidden set contains every
  reservoir vertex already used, except the vertices belonging to the two
  attachment pairs selected for the current connection.  It also contains
  every vertex of $W_{\rm sw}\cap A$ that is not one of those attachment
  vertices.  In particular, the switcher vertices are not used as internal
  vertices before the cycle is closed.

  While at least three paths remain, choose one end pair from each of any
  three current paths and consider the three attachment pairs obtained by
  reversing them.  These ordered pairs are pairwise vertex-disjoint, so
  \cref{equivalentclass} gives two in the same connection class.
  Robustly connect those two attachment pairs through unused vertices of
  $A$ and apply \cref{lem:gluing}.  After at most four such operations,
  two vertex-disjoint tight paths $Q_1,Q_2$ remain.  They contain all the
  original path segments as subpaths and have used at most
  $4\zeta^2n$ vertices of $A$.

  Roughly speaking, the attachment signature records the connection class of
  each available end without choosing an orientation.
  For a tight path $Q$ with end pairs $p,q$, define its \emph{attachment
  signature} to be the multiset
  \[
    \sigma(Q)=\{\kappa(p^{\rm rev}),\kappa(q^{\rm rev})\}.
  \]

  \proofstep{4}{Closing the cycle}
  If $\sigma(Q_1)=\sigma(Q_2)$, then orient the two paths so that corresponding
  attachment classes agree, and use two cross-connections between $Q_1$
  and $Q_2$.  By \cref{lem:gluing}, these connections form one tight cycle,
  not two separate cycles.

  Otherwise we use the reserved switchers.
  In Case~\ref{case1}, for $j\in[2]$, regard $u_jv_j$ as a $2$-vertex
  tight path $S_j$.  Its two attachment pairs are $(u_j,v_j)$ and
  $(v_j,u_j)$, so $\sigma(S_j)=\{1,2\}$.
  In Case~\ref{case2}, for $j\in[2]$, let $S_j=x_jy_jz_j$.  Its attachment
  pairs are $(x_j,y_j)$ and $(z_j,y_j)$, and again
  $\sigma(S_j)=\{1,2\}$.

  If the signatures of $Q_1,Q_2$ are $\{1,1\}$ and $\{2,2\}$, use both
  $S_1,S_2$ and connect the four segments cyclically in the order
  \[
    Q_1,S_1,Q_2,S_2,Q_1,
  \]
  matching class $1$ at the two incidences with $Q_1$ and class $2$ at the
  two incidences with $Q_2$.
  In every other unequal-signature case, after interchanging the class
  names if necessary, the signatures are $\{1,1\}$ and $\{1,2\}$.
  Use one switcher $S_1$: connect the two class-$1$ ends of the first path
  to the class-$1$ ends of the other two segments, and connect their two
  class-$2$ ends to each other.  This gives a three-segment cycle.

  Thus the closing step uses at most four further robust connections.
  When a connection is incident with a switcher, its two attachment
  vertices are omitted from the forbidden set, while all other reservoir
  vertices already present in the partial construction remain forbidden.
  Altogether at most eight connecting paths are used, so the set avoided at
  any step has size at most
  $8\zeta^2n+|W_{\rm sw}|\le\zeta n$.
  Hence robust connectedness can be used at every step.
  Denote the resulting tight cycle by $\Gamma$.

  \proofstep{5}{Divisibility adjustment and absorption}
  Let $L_0=V(H)\setminus V(\Gamma)$.  It is contained in the path-cover
  leftover, the unused vertices of $A$, the vertices of
  $C_1\cup C_2$ outside $R_1\cup R_2$, and any unused switcher vertices.
  Hence $|L_0|\le\gamma n/2$ for sufficiently large $n$.

  Replacing $R_j$ by $R_j^+$ decreases the leftover size by $10\equiv1
  \pmod3$.  Choose $I\subseteq [2]$ with $|I|\equiv|L_0|\pmod3$ and perform
  these replacements for $j\in I$.  The resulting cycle $\Gamma'$ has a
  leftover set $L$ with $|L|\in3\mathbb N$ and $|L|\le\gamma n$.
  Finally use the absorbing property of $\mathcal P_{\rm abs}$ to absorb
  $L$.  Since every member of $\mathcal P_{\rm abs}$ is a subpath of
  $\Gamma$ and the absorption step preserves all ordered end pairs, the
  resulting cycle is a tight Hamiltonian cycle of $H$.
\end{proof}

\section{The non-switcher case}\label{non-switcher-case}

This section handles the remaining case from \cref{dividecase}, where neither type of switcher exists.
The goal is the following non-switcher lemma.

\begin{lemma}\label{nonswitchercase}
  Let $1/n\ll\rho\ll\zeta\ll\eta\ll\min\{\alpha,d\}$.
  Let $H$ be an $n$-vertex $(\rho,d)$-dense $3$-graph with $\delta_2(H) \ge \left(\frac{1}{3}+\alpha\right)n$.
  Let $A\subseteq V(H)$ satisfy $\eta n\le |A|\le 2\eta n$ and, for every pair $uv\in\binom{V(H)}{2}$,
  $
  \deg_A(uv)\ge \left(\frac{1}{3}+\frac{\alpha}{2}\right)|A|.
  $
  Suppose that Case~\ref{case3} occurs with respect to $A$.
  Then $H$ contains a tight Hamiltonian cycle.
\end{lemma}

\subsection{Auxiliary coloring}

The absence of switchers imposes two elementary structural restrictions on the RCCs.

\begin{proposition}\label{propcolor1}
  Let $1/n \ll \zeta \ll \eta \ll \alpha$.
  Let $H$ be an $n$-vertex $3$-graph with $\delta_2(H) \ge \left(\frac{1}{3}+\alpha\right)n$.
  Let $A\subseteq V(H)$ satisfy $\eta n\le |A|\le 2\eta n$ and, for every pair $uv\in\binom{V(H)}{2}$,
  $
  \deg_A(uv)\ge \left(\frac{1}{3}+\frac{\alpha}{2}\right)|A|.
  $
  Let $\mathcal C_1$ and $\mathcal C_2$ be the RCCs given by \cref{equivalentclass} with codegree slack $\alpha/2$, and suppose that Case~\ref{case3} occurs with respect to $A$.
  Then the following statements hold.
  \begin{enumerate}[label=\rm(B\arabic*)]
    \item\label{propcoloritem1} There exists a set $W_1\subseteq V(H)$ with $|W_1|\le 3$ such that, for any distinct $u,v\in V(H)\setminus W_1$, the two ordered pairs $(u,v)$ and $(v,u)$ belong to the same class $\mathcal C_i$ for some $i\in[2]$.

    \item\label{propcoloritem2} There exists a set $W_2\subseteq V(H)\setminus W_1$ with $|W_2|\le 3$ such that, for every edge $xyz\in E(H[V(H)\setminus (W_1\cup W_2)])$, the ordered pairs $(x,y),(y,z)$ and $(z,x)$ belong to the same class $\mathcal C_i$ for some $i\in[2]$.
  \end{enumerate}
\end{proposition}

\begin{proof}
  We first prove~\ref{propcoloritem1}.
  Let $\mathcal S$ be the set of unordered pairs $\{u,v\}$ such that $(u,v)$ and $(v,u)$ belong to different RCCs.
  Since Case~\ref{case1} does not occur, $\mathcal S$ contains no two disjoint pairs.
  Hence $\mathcal S$ is either empty, a star, or a triangle.
  Let $W_1$ be empty in the first case, the center of the star in the second case, and the vertex set of the triangle in the third case.
  Then $|W_1|\le 3$, and every pair outside $W_1$ has its two orientations in the same class.

  Next, we prove~\ref{propcoloritem2}.
  Let $\mathcal E$ be the set of edges $xyz\in E(H-W_1)$ for which, for some cyclic ordering of its vertices, the ordered pairs $(x,y),(y,z),(z,x)$ are not all in the same class.

  We claim that $\mathcal E$ contains no two disjoint edges.
  Indeed, suppose that $e_1,e_2\in\mathcal E$ are disjoint.
  For each $j\in[2]$, since $e_j\subseteq V(H)\setminus W_1$, every pair of vertices in $e_j$ has its two orientations in the same class by~\ref{propcoloritem1}.
  Since $e_j\in\mathcal E$, the three boundary pairs of $e_j$ are not all in the same class.
  Hence, after possibly interchanging the names of $\mathcal C_1$ and $\mathcal C_2$, and after relabeling the vertices of $e_j$, we may write $e_j=x_jy_jz_j$ so that $(x_j,y_j)\in\mathcal C_1$ and $(z_j,y_j)\in\mathcal C_2$.
  Thus the two disjoint edges $x_1y_1z_1$ and $x_2y_2z_2$ give Case~\ref{case2}, a contradiction.
  Therefore $\mathcal E$ contains no two disjoint edges.

  If $\mathcal E=\emptyset$, set $W_2=\emptyset$.
  Otherwise choose one edge $e\in\mathcal E$ and set $W_2=e$.
  Then $W_2\subseteq V(H)\setminus W_1$ and $|W_2|\le 3$.
  Since $\mathcal E$ contains no two disjoint edges, every edge of $\mathcal E$ intersects $W_2$.
  Therefore no edge of $H[V(H)\setminus (W_1\cup W_2)]$ lies in $\mathcal E$.

  Consequently, for every edge $xyz\in E(H[V(H)\setminus (W_1\cup W_2)])$, the ordered pairs $(x,y),(y,z)$ and $(z,x)$ all belong to the same robust connection class.
\end{proof}

We will use an edge colouring to label edges of $H$. Roughly speaking, color $i$ marks a pair whose two orientations lie in
$\mathcal C_i$, whereas color $0$ marks a pair whose orientations lie in
different classes.  An exceptional set removes all color-$0$ pairs and makes
every remaining hyperedge monochromatic.

\begin{defn}[$A$-connected coloring]\label{defn:color}
  Let $H$ be a $3$-graph and let $A\subseteq V(H)$ be a reservoir set.
  Let $\mathcal C_1$ and $\mathcal C_2$ be the RCCs given by
  \cref{equivalentclass}.

  We define the \emph{$A$-connected coloring} as a function $\phi:\partial H\rightarrow \{0,1,2\}$ on unordered pairs in the shadow.
  For $xy\in\partial H$, define $\phi(xy)$ as follows:
  \begin{enumerate}[label=\rm(\roman*)]
    \item $\phi(xy)\coloneqq 0$ if the two ordered pairs $(x,y)$ and $(y,x)$ belong to different RCCs;
    \item $\phi(xy)\coloneqq i$ for $i\in\{1,2\}$ if both $(x,y)$ and $(y,x)$ belong to $\mathcal C_i$.
  \end{enumerate}
  We say that an ordered pair $(x,y)$ has color $\phi(xy)$.

  A set $W\subseteq V(H)$ is called an \emph{exceptional set} if both of the following properties hold:
  \begin{enumerate}[label=\rm(\roman*)]
    \item for every pair $xy\in\partial H[V(H)\setminus W]$, we have $\phi(xy)\ne 0$;
    \item for every edge $xyz\in E(H[V(H)\setminus W])$, the three boundary pairs are monochromatic under $\phi$, that is, $\phi(xy)=\phi(yz)=\phi(xz)$.
  \end{enumerate}
\end{defn}

After this coloring is defined, the previous proposition gives the basic properties used throughout the non-switcher argument.

\begin{lemma}\label{prop:coloritem}
  Let $1/n \ll \zeta \ll \eta \ll \alpha$.
  Let $H$ be an $n$-vertex $3$-graph with $\delta_2(H) \ge \left(\frac{1}{3}+\alpha\right)n$.
  Let $A\subseteq V(H)$ be a reservoir set satisfying $\eta n\le |A|\le 2\eta n$ and, for every pair $uv\in\binom{V(H)}{2}$, $\deg_A(uv)\ge \left(\frac{1}{3}+\frac{\alpha}{2}\right)|A|$.
  Let $\mathcal C_1,\mathcal C_2$ be the RCCs given by \cref{equivalentclass} with codegree slack $\alpha/2$, and suppose that Case~\ref{case3} holds with respect to $A$.
  Let $\phi$ be the $A$-connected coloring determine by $\C_1$, $\C_2$.
  Then there exists an exceptional set $W\subseteq V(H)$ such that:
  \begin{enumerate}[label=\rm(C\arabic*)]
    \item\label{coloritem0} $|W|\le 6$.

    \item\label{coloritem1} For every pair $xy\in\binom{V(H)\setminus W}{2}$, we have $\phi(xy)\ne 0$.

    \item\label{coloritem2} For every edge $xyz\in E(H-W)$, we have $\phi(xy)=\phi(yz)=\phi(xz)$.
      Moreover, for every pair $xy\subseteq V(H)\setminus W$, there are at least $\left(\frac{1}{3}+\frac{\alpha}{4}\right)n$ vertices $z\in V(H)\setminus W$ such that $xyz\in E(H)$ and $\phi(xy)=\phi(yz)=\phi(xz)$.

    \item\label{coloritem3} There exists a color $i\in\{1,2\}$ which spans $V(H)\setminus W$, meaning that every vertex of $V(H)\setminus W$ is incident to at least one pair of color $i$.
  \end{enumerate}
\end{lemma}

\begin{proof}
  By \cref{propcolor1}, choose sets $W_1,W_2\subseteq V(H)$ satisfying~\ref{propcoloritem1} and~\ref{propcoloritem2}, and let $W \coloneqq W_1\cup W_2$.
  Then $|W|\le 6$.

  Let $B \coloneqq V(H)\setminus W$.
  By \cref{propcolor1}~\ref{propcoloritem1}, the two orientations of every pair in $B$ belong to the same class.
  Hence every pair in $B$ has color $1$ or color $2$, proving~\ref{coloritem1}.
  By \cref{propcolor1}~\ref{propcoloritem2}, every edge of $H[B]$ has all three boundary pairs in the same class, and therefore in the same color under $\phi$.
  This proves the first part of~\ref{coloritem2}.

  For the moreover part, fix a pair $xy\subseteq B$.
  Since $|W|\le 6$ and $\delta_2(H)\ge(1/3+\alpha)n$, we have $\deg_B(xy)\ge (1/3+\alpha)n-6\ge (1/3+\alpha/4)n$ for all sufficiently large $n$.
  Every edge $xyz$ with $z\in B$ is monochromatic under $\phi$, so all such vertices $z$ satisfy the required color condition.

  It remains to prove~\ref{coloritem3}.
  Suppose that neither color spans $B$.
  Then there is a vertex $u\in B$ incident to no pair of color $1$, and a vertex $v\in B$ incident to no pair of color $2$.
  Since every pair in $B$ has color $1$ or color $2$, all pairs incident to $u$ have color $2$, while all pairs incident to $v$ have color $1$.
  This is impossible for the pair $uv$.
  Hence at least one of the two colors spans $B$.
\end{proof}

We also need a simple counting fact for $2$-colorings of complete graphs.

\begin{lemma}\label{LEMMA:counting-monotriangles}
  Let $\phi:E(K_n)\to[2]$ be a $2$-edge-coloring of the complete graph $K_n$.
  Let $G$ be the spanning subgraph consisting of all edges of color $1$.
  Suppose that, for every vertex $v\in V(K_n)$,
  $
  m\le d_G(v)\le n-1-m,
  $
  where $m\le (n-1)/2$.
  Then the number of monochromatic triangles in $K_n$ under $\phi$ is at most
  $
  \binom{n}{3}-\frac{mn(n-1-m)}{2}.
  $
\end{lemma}

\begin{proof}
  For each vertex $v$, write $r_v=d_G(v)$ and $b_v=n-1-r_v$.
  Let $T_{\mathrm{mono}}$ and $T_{\mathrm{bi}}$ be the numbers of monochromatic and non-monochromatic triangles, respectively.
  Every non-monochromatic triangle contains exactly two mixed-color $2$-edge stars.
  Hence $T_{\mathrm{bi}} = \frac{1}{2}\sum_{v\in V(K_n)} r_vb_v$.
  Since $m\le r_v\le n-1-m$, we have $r_vb_v\ge m(n-1-m)$ for every $v$.
  Thus $T_{\mathrm{bi}}\ge \frac{mn(n-1-m)}{2}$.
  Since $T_{\mathrm{mono}}+T_{\mathrm{bi}}=\binom{n}{3}$, the claim follows.
\end{proof}

Combining the coloring structure with the counting lemma, we find a linear-sized set of vertices whose incident pairs all have the same color.

\begin{lemma}\label{lem:monosetexist}
  Let $1/n \ll \zeta \ll \eta \ll \beta \ll \alpha \ll 1$.
  Let $H$ be an $n$-vertex $3$-graph with $\delta_2(H) \ge \left(\frac{1}{3}+\alpha\right)n$.
  Let $A\subseteq V(H)$ be a reservoir set satisfying $\eta n\le |A|\le 2\eta n$ and, for every pair $uv\in\binom{V(H)}{2}$,
  $
  \deg_A(uv)\ge \left(\frac{1}{3}+\frac{\alpha}{2}\right)|A|.
  $
  Let $\mathcal C_1,\mathcal C_2$ be the RCCs given by \cref{equivalentclass} with codegree slack $\alpha/2$, and suppose that Case~\ref{case3} holds with respect to $A$.
  Let $\phi$ be the $A$-connected coloring determine by $\mathcal C_1, \mathcal C_2$, and let $W$ be the exceptional set from \cref{prop:coloritem}.

  Let $i\in\{1,2\}$ be a spanning color on $V(H)\setminus W$.
  Then there exists a vertex set $V_0\subseteq V(H)\setminus W$ with $|V_0|\ge \beta n$ such that every vertex in $V_0$ is incident only to pairs of color $i$ in the complete graph on $V(H)\setminus W$ colored by $\phi$.
  Moreover, there exists a unique spanning color on $V(H)\setminus W$.
\end{lemma}

\begin{proof}
  Let $B=V(H)\setminus W$.
  By \cref{prop:coloritem}~\ref{coloritem3}, some color spans $B$.
  By relabeling the colors, we may assume that $i=1$ and that color $1$ spans $B$.
  Let $V_0$ be the set of vertices in $B$ incident only to pairs of color $1$.
  We show that $|V_0|\ge \beta n$; the proof will also imply the
  \emph{moreover} part of the lemma.

  Suppose for contradiction that $|V_0|<\beta n$, and put $B'\coloneqq B\setminus V_0$ and $N\coloneqq |B'|$.
  Let $T_{\mathrm{mono}}$ be the number of monochromatic triangles in the complete graph on $B'$ under the coloring $\phi$.

  By \cref{prop:coloritem}~\ref{coloritem2}, every edge of $H[B']$ gives a monochromatic triangle under $\phi$.
  Moreover, for every pair $xy\subseteq B'$,
  \[
    \deg_{B'}(xy)
    \ge
    \left(\frac{1}{3}+\frac{\alpha}{4}\right)n-|V_0|
    \ge
    \left(\frac{1}{3}+\frac{\alpha}{5}\right)n,
  \]
  by the hierarchy $\beta\ll\alpha$.
  Hence
  \[
    T_{\mathrm{mono}}
    \ge |E(H[B'])|
    =
    \frac{1}{3}\sum_{xy\in\binom{B'}{2}}\deg_{B'}(xy)
    \ge
    \frac{1}{3}\binom{N}{2}
    \left(\frac{1}{3}+\frac{\alpha}{5}\right)n.
  \]
  Since $N\ge n-|W|-|V_0|\ge (1-\beta)n-6$, and $\beta\ll\alpha$, this gives
  \[
    T_{\mathrm{mono}}
    \ge \left(\frac{1}{18}+\frac{\alpha}{40}\right)n^3
  \]
  for all sufficiently large $n$.

  We now upper-bound $T_{\mathrm{mono}}$.
  Let $G$ be the graph of color-$1$ pairs on $B'$.
  Fix $x\in B'$.
  Since color $1$ spans $B$, the vertex $x$ is incident to a pair of color $1$ in $B$.
  By \cref{prop:coloritem}~\ref{coloritem2}, this implies that $x$ has at least $(1/3+\alpha/4)n-|V_0|$ color-$1$ neighbors in $B'$.
  Since $x\notin V_0$, it is also incident to a pair of color $2$, and the same argument gives at least $(1/3+\alpha/4)n-|V_0|$ color-$2$ neighbors in $B'$.
  Therefore, for all sufficiently large $n$,
  \[
    \left(\frac{1}{3}+\frac{\alpha}{8}\right)N
    \le d_G(x)\le
    N-1-\left(\frac{1}{3}+\frac{\alpha}{8}\right)N.
  \]
  Applying \cref{LEMMA:counting-monotriangles} to the complete graph on
  $B'$ with $m=(1/3+\alpha/8)N$, we obtain
  \begin{align*}
    T_{\mathrm{mono}}
    &\le
    \binom{N}{3}-\frac{N}{2}m(N-1-m)\\
    &\le
    \left(\frac{1}{18}-\frac{\alpha}{48}
    +\frac{\alpha^2}{128}\right)N^3+O(N^2)\\
    &\le
    \left(\frac{1}{18}-\frac{\alpha}{64}\right)n^3,
  \end{align*}
  where the last inequality uses $\alpha\le10^{-3}$ and $n$ sufficiently
  large.
  This contradicts the lower bound above.
  Hence $|V_0|\ge \beta n$.  To prove uniqueness, choose $v\in V_0$.
  All pairs incident with $v$ have color $i$.
  Hence the other color has no edge incident with $v$, and therefore cannot
  span $V(H)\setminus W$.  Thus $i$ is the unique spanning color on
  $V(H)\setminus W$.
\end{proof}

Uniform density now turns the one-sided set $V_0$ into density of the color-$i$ shadow graph.

\begin{lemma}\label{lem:shadowlocaldense}
  Let $1/n \ll \rho \ll \rho' \ll \zeta \ll \eta \ll \min\{\alpha,d\} \le \max\{\alpha,d\} \ll 1$.
  Let $H$ be an $n$-vertex $(\rho,d)$-dense $3$-graph with $\delta_2(H) \ge \left(\frac{1}{3}+\alpha\right)n$.
  Let $A\subseteq V(H)$ be a reservoir set satisfying $\eta n\le |A|\le 2\eta n$ and, for every pair $uv\in\binom{V(H)}{2}$,
  $
  \deg_A(uv)\ge \left(\frac{1}{3}+\frac{\alpha}{2}\right)|A|.
  $
  Let $\mathcal C_1,\mathcal C_2$ be the RCCs given by \cref{equivalentclass} with codegree slack $\alpha/2$, and suppose that Case~\ref{case3} holds with respect to $A$.
  Let $\phi$ be the $A$-connected coloring determine by $\mathcal C_1, \mathcal C_2$, and let $W$ be the exceptional set from \cref{prop:coloritem}.

  Let $i\in\{1,2\}$ be the spanning color on $V(H)\setminus W$.
  Then the graph $\partial_2^i(H\setminus W)$ consisting of color-$i$ pairs on
  $V(H)\setminus W$ is $(\rho',d)$-dense, that is, for all
  $X,Y\subseteq V(H)\setminus W$,
  \[
    e_{\partial_2^i(H\setminus W)}(X,Y)
    \ge d|X||Y|-\rho'n^2.
  \]
\end{lemma}

\begin{proof}
  Choose $\beta$ such that $\eta\ll\beta\ll\alpha$, and assume $\rho\ll \rho'\beta$.
  By \cref{lem:monosetexist}, there exists a set $V_0\subseteq V(H)\setminus W$ with $|V_0|\ge\beta n$ such that every vertex of $V_0$ is incident only to pairs of color $i$ in $V(H)\setminus W$.

  Let $X,Y\subseteq V(H)\setminus W$.
  Then by the hierarchy $\rho\ll\rho'\beta$, we have
  \[
    e_H(X,Y,V_0)
    \ge d|X||Y||V_0|-\rho n^3
    \ge d|X||Y||V_0|-\rho'\beta n^3.
  \]
  Consider a triple counted by $e_H(X,Y,V_0)$, say with vertices $x\in X$, $y\in Y$, and $z\in V_0$.
  Since $z\in V_0$, the pairs $xz$ and $yz$ have color $i$.
  By \cref{prop:coloritem}~\ref{coloritem2}, the edge $xyz$ is monochromatic under $\phi$, so $xy$ also has color $i$.
  Thus every such hyperedge shows that $xy$ is a color-$i$ pair between $X$ and~$Y$.

  Each pair $xy\in X\times Y$ can arise from at most $|V_0|$ choices of $z\in V_0$.
  Hence
  \[
    e_{\partial_2^i(H\setminus W)}(X,Y)
    \ge
    \frac{e_H(X,Y,V_0)}{|V_0|}
    \ge
    d|X||Y|-\frac{\rho'\beta n^3}{|V_0|}
    \ge
    d|X||Y|-\rho'n^2.
  \]
  This proves the lemma.
\end{proof}
\subsection{Path cover lemma}

This subsection constructs a small collection of tight paths covering almost all vertices outside the reservoir.

\begin{lemma}\label{lem:longpath}
  Let $1/n\ll \rho\ll \zeta\ll \eta\ll \gamma\ll \beta\ll \min\{\alpha,d\}$.
  Let $H$ be an $n$-vertex $(\rho,d)$-dense $3$-graph with $\delta_2(H)\ge \left(\frac{1}{3}+\alpha\right)n$.
  Let $A\subseteq V(H)$ be a reservoir set satisfying $\eta n\le |A|\le 2\eta n$ and for every pair $uv\in\binom{V(H)}{2}$, $\deg_A(uv)\ge \left(\frac{1}{3}+\frac{\alpha}{2}\right)|A|$.
  Let $\mathcal C_1,\mathcal C_2$ be the RCCs given by \cref{equivalentclass} with codegree slack $\alpha/2$, and suppose that Case~\ref{case3} holds with respect to $A$.
  Let $\phi$ be the $A$-connected coloring determine by $\mathcal C_1, \mathcal C_2$, and let $W$ be the exceptional set from \cref{prop:coloritem}.
  Let $B_{\rm used}\subseteq V(H)$ be a set of size at most $\beta n$ such
  that $|A\cap B_{\rm used}|\le\zeta n/2$.
  Let $i\in\{1,2\}$ be the spanning color on $V(H)\setminus W$.

  Then $H-(B_{\rm used}\cup W)$ contains two vertex-disjoint tight paths
  $Q_1,Q_2$ with $|V(Q_1\cup Q_2)\cap A|\le \zeta |A|/2$ whose end pairs
  all have color $i$, covering all vertices of
  $V(H)\setminus(A\cup B_{\rm used}\cup W)$ except for at most $\gamma n$
  vertices.
\end{lemma}

The proof combines weak hypergraph regularity, an almost-perfect matching in the weak reduced graph, and the density of the color-$i$ shadow graph from \cref{lem:shadowlocaldense}.

\begin{defn}[$\eps$-regularity for $3$-graphs]
  Let $H=(V,E)$ be a $3$-graph, and let $V_1,V_2,V_3$ be disjoint subsets of $V$.
  We say that the triple $(V_1,V_2,V_3)$ is \emph{$\eps$-regular} if for all $A_i\subseteq V_i$ with $|A_i|\ge \eps |V_i|$ for $i\in[3]$, we have
  \[
    |d_H(A_1,A_2,A_3)-d_H(V_1,V_2,V_3)|\le \eps,
  \]
  where $d_H(U_1,U_2,U_3)\coloneqq e_H(U_1,U_2,U_3)/(|U_1||U_2||U_3|)$.
\end{defn}

We use the following weak regularity lemma to pass from the color-$i$ hypergraph to a reduced graph.

\begin{lemma}[Weak hypergraph regularity lemma~\cite{chung1990quasi}]
  \label{weakregular}
  For every $\eps>0$ and every positive integer $t_0$, there exists an integer $T=T(\eps,t_0)$ such that for every $3$-graph $H=(V,E)$ with $|V|\ge T$, there exists a partition $V=V_{\rm exc}\cup V_1\cup\dots\cup V_t$ such that:
  \begin{enumerate}[label=\rm(\roman*)]
    \item $t_0\le t\le T$;
    \item $|V_{\rm exc}|\le \eps |V|$ and $|V_1|=\dots=|V_t|$;
    \item all but at most $\eps\binom{t}{3}$ triples $\{i_1,i_2,i_3\}\subseteq[t]$ are $\eps$-regular.
  \end{enumerate}
\end{lemma}

\begin{defn}[Weak reduced graph]
  Let $H$ be a $3$-graph with a partition $V=V_{\rm exc}\cup V_1\cup\dots\cup V_t$.
  For $d,\eps>0$, the \emph{weak $d$-reduced $3$-graph} $R=R^w_{d,\eps}(H)$ is the $3$-graph on vertex set $[t]$, where $ijk\in E(R)$ if and only if $(V_i,V_j,V_k)$ is $\eps$-regular and has density at least $d$ in $H$.
\end{defn}

Two standard tools are then used on the reduced side: one finds an almost-perfect matching, and the other turns each regular triple in the matching into long tight paths.

\begin{lemma}[Almost perfect matching, {\cite[Lemma~4.3]{GH}}]
  \label{matchingexist}
  Let $0<1/n\ll \eps\ll \gamma\ll \alpha$.
  Let $H$ be a $3$-graph on $n$ vertices such that all but at most $\eps n^2$ pairs $S\subseteq V(H)$ satisfy $\deg_H(S)\ge \left(\frac{1}{3}+\alpha\right)n$.
  Then $H$ contains a matching that covers all but at most $\gamma n$ vertices.
\end{lemma}

\begin{lemma}[{\cite[Lemma~6.2]{schulke}}]
  \label{pairvertexdegreelongpathexist}
  For $\gamma>0$, $d>0$, $\eps>0$, and $n\in\mathbb N$ with $\frac{1}{2}(d\gamma^3-\eps)n\ge 1$, the following holds.
  Let $H=(U\cup V\cup W,E)$ be a $3$-graph with $|U|=|V|=|W|=n$, and suppose that $(U,V,W)$ is $\eps$-regular with density at least $d$ in $H$.
  Then at least $(1-\gamma)3n$ vertices of $H$ can be covered by vertex-disjoint tight paths, each with at least $\frac{1}{2}(d\gamma^3-\eps)n-2$ vertices.
\end{lemma}

We now prove the path cover lemma.

\begin{proof}[Proof of \cref{lem:longpath}]
  Choose the weak-regularity parameters $t_0,\eps,d'$ so that
  \[
    1/t_0\ll \eps\ll d'\gamma^3\ll d'
    \ll \min\{d,\alpha,\gamma\},
  \]
  and let
  $T=T(\eps,t_0)$ be given by \cref{weakregular}.  Put
  \[
    c_0=\frac14\bigl(d'(\gamma/4)^3-\eps\bigr)>0
    \quad\text{and}\quad
    M=\left\lceil\frac{2T}{c_0}\right\rceil.
  \]
  In addition to the displayed hierarchy, take $\zeta$ sufficiently small
  that $M\zeta\ll\eta$, and choose $\rho'>0$ such that
  $\rho'T^2\ll d$ and $\rho\ll\rho'$.

  By \cref{lem:shadowlocaldense}, the color-$i$ shadow graph $G_i\coloneqq \partial_2^i(H\setminus W)$ is $(\rho',d)$-dense as a $2$-graph.

  Let $V_{\rm work}=V(H)\setminus(A\cup B_{\rm used}\cup W)$.
  Let $H_i$ be the $3$-graph on $V_{\rm work}$ whose edges are the edges $xyz\in E(H[V_{\rm work}])$ such that $\phi(xy)=\phi(xz)=\phi(yz)=i$.
  Apply \cref{weakregular} to $H_i$ with parameters $\eps,t_0$, obtaining a partition $V_{\rm work}=V_{\rm exc}\cup V_1\cup\dots\cup V_t$, where $t_0\le t\le T$.
  Let $R=R^w_{d',\eps}(H_i)$ be the corresponding weak reduced graph, and write $r=|V_1|=\cdots=|V_t|$.

  The key reduced-graph point is that almost every pair of clusters has large degree in $R$.

  \begin{claim}
    All but at most $2\sqrt{\eps}t^2$ pairs $S\in\binom{[t]}{2}$ satisfy $\deg_R(S)\ge \left(\frac{1}{3}+\frac{\alpha}{16}\right)t$.
  \end{claim}

  \begin{proof}[Proof of Claim]
    Since $G_i$ is $(\rho',d)$-dense as a $2$-graph, for every pair of clusters $V_a,V_b$ we have
    \[
      e_{G_i}(V_a,V_b)\ge dr^2-\rho'n^2\ge dr^2/2,
    \]
    where the last inequality follows from $t\le T$ and $\rho'T^2\ll d$.
    Put $g_{ab}=e_{G_i}(V_a,V_b)/r^2$, so $g_{ab}\ge d/2$.

    Fix a pair of clusters $(V_a,V_b)$.
    For every color-$i$ pair $xy$ with $x\in V_a$ and $y\in V_b$, \cref{prop:coloritem}~\ref{coloritem2} gives at least $(1/3+\alpha/4)n$ vertices $z\in V(H)\setminus W$ such that $xyz\in E(H)$ and all three boundary pairs have color $i$.
    After excluding $A\cup B_{\rm used}\cup W\cup V_{\rm exc}\cup V_a\cup V_b$, at least $(1/3+\alpha/8)n$ such vertices remain, by the hierarchy $\max\{\eta,\beta,\eps,1/t_0\}\ll\alpha$.
    Therefore
    \[
      e_{H_i}\Big(V_a,V_b,\bigcup_{k\in[t]\setminus\{a,b\}}V_k\Big)
      \ge
      \left(\frac{1}{3}+\frac{\alpha}{8}\right)n\,g_{ab}r^2.
    \]

    Let $t_{ab}$ be the number of indices $k\in[t]\setminus\{a,b\}$ such that $d_{H_i}(V_a,V_b,V_k)\ge d'$.
    Since every edge of $H_i$ with one vertex in each of $V_a,V_b,V_k$ uses a color-$i$ pair between $V_a$ and $V_b$, we have
    \[
      e_{H_i}\Big(V_a,V_b,\bigcup_{k\in[t]\setminus\{a,b\}}V_k\Big)
      \le
      t_{ab}g_{ab}r^3+td'r^3.
    \]
    Combining the two bounds, using $n\ge tr$, $g_{ab}\ge d/2$, and $d'\ll \alpha d$, gives $t_{ab}\ge \left(\frac{1}{3}+\frac{\alpha}{10}\right)t$.

    By \cref{weakregular}, all but at most $\sqrt{\eps}t^2$ pairs of clusters are contained in at most $\sqrt{\eps}t$ irregular triples.
    Indeed, if more than $\sqrt{\eps}t^2$ pairs were contained in more than $\sqrt{\eps}t$ irregular triples, then the total number of incidences between pairs and irregular triples would exceed $3\eps\binom{t}{3}$ for sufficiently large $t$, contradicting \cref{weakregular}.
    For every non-exceptional pair, at least $\left(\frac{1}{3}+\frac{\alpha}{10}\right)t-\sqrt{\eps}t \ge \left(\frac{1}{3}+\frac{\alpha}{16}\right)t$ clusters give regular triples of density at least $d'$.
    Enlarging the exceptional set to $2\sqrt{\eps}t^2$ proves the claim.
  \end{proof}

  By \cref{matchingexist}, applied to $R$ with parameters $2\sqrt{\eps}$ and $\alpha/16$, the reduced graph $R$ contains a matching $M$ covering all but at most $\gamma t/8$ clusters.
  For each edge of $M$, \cref{pairvertexdegreelongpathexist}, applied with parameter $\gamma/4$, gives a collection of vertex-disjoint tight paths in the corresponding regular triple covering all but at most a $\gamma/4$ proportion of its vertices.
  Let $\mathcal P$ be the union of all these paths.
  Since all edges of these paths lie in $H_i$, every end pair of every path
  in $\mathcal P$ has color $i$.  By taking $t_0$ sufficiently large, the
  matching $M$ has at least two edges; each matched regular triple contributes
  at least one path.  Thus $|\mathcal P|\ge2$.

  Each path obtained from \cref{pairvertexdegreelongpathexist} has at least
  $\frac{1}{2}\left(d'(\gamma/4)^3-\eps\right)r-2$ vertices.  For all
  sufficiently large $n$, this is at least $c_0r$.  Since the paths in
  $\mathcal P$ are vertex-disjoint and $r\ge n/(2T)$, we obtain
  \[
    |\mathcal P|\le \frac{n}{c_0r}\le \frac{2T}{c_0}\le M.
  \]

  We now connect the paths in $\mathcal P$ using the reservoir $A$.
  Since every end pair has color $i$, its reverse, the attachment pair used
  for gluing, also lies in $\mathcal C_i$.  Choose ends on two distinct
  current paths, connect their attachment pairs, and apply
  \cref{lem:gluing}.  Repeating this operation leaves two paths.
  At each step, we use at most $\zeta^2n$ new vertices from $A$.  The
  forbidden set contains $A\cap(B_{\rm used}\cup W)$ and every previously
  used reservoir vertex, except the vertices of the two attachment pairs
  selected in the current connection.
  Since $|\mathcal P|\le M$, $|A|\ge\eta n$, and $M\zeta\ll\eta$,
  the total number of reservoir vertices used during this process is at most
  $M\zeta^2n\le \zeta |A|/2$.  Together with
  $|A\cap B_{\rm used}|\le\zeta n/2$ and $|W|=O(1)$, this is at most
  $\zeta n$ by the hierarchy.
  Hence the robust tight connectedness condition can be applied throughout.

  Let the resulting two paths be $Q_1,Q_2$.
  Their end pairs are inherited from paths in $\mathcal P$, so all of them have color $i$.
  The vertices of $V_{\rm work}$ not covered by $Q_1\cup Q_2$ consist of $V_{\rm exc}$, vertices in clusters not covered by the matching $M$, and vertices missed inside the matched triples.
  Hence the number of uncovered vertices is at most
  \[
    |V_{\rm exc}|+\frac{\gamma}{8}n+\frac{\gamma}{4}n
    \le \eps n+\frac{\gamma}{2}n\le \gamma n.
  \]
  This proves the lemma.
\end{proof}

\subsection{Monochromatic absorption}
\label{subsec:monochromatic-absorption}

Fix a $3$-graph $H$, a reservoir set $A\subseteq V(H)$, an $A$-connected
coloring $\phi$, and a color $i\in\{1,2\}$.
Roughly speaking, $\Omega_i$ consists of the ordered ends of color $i$, which
are the only ends used below.  Formally,
$
\Omega_i
\coloneqq
\{(x,y)\in\overrightarrow{\partial}_2(H):\phi(xy)=i\}.
$
The set $\Omega_i$ is reversal-closed.  We use
\emph{$(\beta,j,i)$-reachable} and \emph{$(\beta,j,i)$-closed} as shorthand
for $(\beta,j;\Omega_i)$-reachable and $(\beta,j;\Omega_i)$-closed,
respectively.  The notation $\tilde N_{A,\beta,j,i}(v)$ has the analogous
meaning.

To apply the RCC-class absorption lemma, we prove the required reachability
and index-vector conditions for these color-$i$ ends.  The
colored analogue of \cref{prop:absorbernumber} is the following.

\begin{proposition}\label{prop:absorbernumbernonswitcher}
  Let $1/n \ll \rho \ll \zeta \ll \eta \ll \beta \ll \min\{\alpha,d\} \le \max\{\alpha,d\} \ll 1$.
  Let $H$ be an $n$-vertex $(\rho,d)$-dense $3$-graph with $\delta_2(H) \ge (1/3+\alpha)n$.
  Let $A \subseteq V(H)$ be a reservoir set satisfying $\eta n\le |A|\le 2\eta n$ and, for every pair $uv \in \binom{V(H)}{2}$, we have $\deg_A(uv) \ge (1/3+\alpha/2)|A|$.
  Let $\mathcal C_1,\mathcal C_2$ be the RCCs given by \cref{equivalentclass} with codegree slack $\alpha/2$, and suppose that Case~\ref{case3} holds with respect to $A$.
  Let $\phi$ be the $A$-connected coloring determine by $\mathcal C_1, \mathcal C_2$.
  Let $W$ be the exceptional vertex set provided by \cref{prop:coloritem}, and let $i\in\{1,2\}$ be the spanning color on $V(H)\setminus W$.
  Then, for every $v\in V(H)$, $|\tilde N_{A,\beta,1,i}(v)|\ge (1/3+\alpha/8)n$.
\end{proposition}

\begin{proof}
  Choose constants $\xi$ and $\gamma$ such that $\beta\ll \gamma\ll \xi\ll \alpha$, and let $B=V(H)\setminus W$.
  For each vertex $v$, its color-$i$ link shadow records the pairs that form an
  edge with $v$ and can serve as color-$i$ ends.  Formally, define
  \[
    L_v\coloneqq
    \left\{\{x,y\}\in\tbinom{B\setminus A}{2}\colon
    vxy\in E(H)\text{ and }\phi(xy)=i\right\}.
  \]
  We first show that a large common color-$i$ link shadow gives one-step color-$i$ reachability.
  If $u\ne v$ and $|L_u\cap L_v|\ge \gamma n^2$, then $u$ and $v$ are $(\beta,1,i)$-reachable.
  Indeed, let $G$ be the graph on $(B\setminus A)\setminus\{u,v\}$ with edge set $(L_u\cap L_v)\cap\binom{(B\setminus A)\setminus\{u,v\}}2$.
  For sufficiently large $n$, we have $e(G)\ge \gamma n^2/2$.
  By supersaturation, $G$ contains at least $c n^4$ labeled paths $abxy$, where $c=c(\gamma)>0$.
  For each such path, the sequences $a,b,u,x,y$ and $a,b,v,x,y$ are tight paths in $H$.
  They have the same ordered end pairs $(b,a)$ and $(x,y)$, and these end pairs have color $i$.
  Since each $4$-set gives at most a constant number of labeled paths, choosing $\beta\ll c$ gives at least $\beta n^4$ distinct witness sets.
  Hence $u$ and $v$ are $(\beta,1,i)$-reachable.

  Now fix $v\in V(H)$.
  By \cref{lem:monosetexist}, applied with $\xi$ in place of $\beta$, there exists a set $V_0\subseteq B$ with $|V_0|\ge \xi n$ such that every vertex in $V_0$ is incident only to pairs of color $i$ inside $B$.
  Let $V_0'=V_0\setminus(A\cup\{v\})$.
  Since $\eta\ll\xi$, we have $|V_0'|\ge \xi n/2$.

  For every $x\in V_0'$, at least $(1/3+\alpha)n-|A|-|W|-1$ vertices $y\in B\setminus A$ satisfy $vxy\in E(H)$.
  For each such $y$, the pair $xy$ has color $i$ because $x\in V_0$.
  Thus $xy\in L_v$.
  Each unordered pair is counted at most twice in this way, and therefore
  \[
    |L_v|\ge
    \frac12 |V_0'|\left((1/3+\alpha)n-|A|-|W|-1\right)
    \ge \frac{\xi}{10}n^2.
  \]
  Let $B_0=V(H)\setminus(\tilde N_{A,\beta,1,i}(v)\cup W\cup\{v\})$.
  By the first paragraph, every $u\in B_0$ satisfies $|L_u\cap L_v|<\gamma n^2$.
  Hence $\sum_{u\in B_0}|L_u\cap L_v|<\gamma n^3$.
  On the other hand,
  \[
    \sum_{u\in B_0}|L_u\cap L_v|
    =
    \sum_{S\in L_v}\deg^*_{B_0}(S),
  \]
  where $\deg^*_{B_0}(S)$ denotes the number of vertices $u\in B_0$ such that $uS\in E(H)$.

  Suppose, for a contradiction, that $|\tilde N_{A,\beta,1,i}(v)|<(1/3+\alpha/8)n$.
  For every $S\in L_v$, the minimum codegree condition gives
  \[
    \deg^*_{B_0}(S)
    \ge
    (1/3+\alpha)n-(1/3+\alpha/8)n-|W|-1
    \ge \alpha n/2.
  \]
  Consequently,
  \[
    \sum_{u\in B_0}|L_u\cap L_v|
    \ge |L_v|\cdot \frac{\alpha}{2}n
    \ge \frac{\alpha\xi}{20}n^3,
  \]
  contradicting $\gamma\ll \alpha\xi$.
  Therefore $|\tilde N_{A,\beta,1,i}(v)|\ge(1/3+\alpha/8)n$, as desired.
\end{proof}

We next turn colored reachability into the partition needed for absorption.
This is the usual reachability partition argument applied to the color-$i$ reachability relation.

\begin{lemma}\label{lem:partitionnonswitcher}
  Let
  $1/n\ll\rho\ll\zeta\ll\eta\ll\beta\ll\beta_0
  \ll\min\{\alpha,d\}$.
  Under the hypotheses of \cref{prop:absorbernumbernonswitcher}, there is a
  partition $\mathcal P_0$ of $V(H)$ into one or two parts such that every
  part has size at least $n/3$ and is $(\beta,2,i)$-closed in $H-A$.
\end{lemma}

\begin{proof}
  Put $\delta=1/3+\alpha/8$.
  By \cref{prop:absorbernumbernonswitcher}, every vertex has at least
  $\delta n$ vertices which are $(\beta_0,1,i)$-reachable from it.
  The proof of the reachability partition lemma \cref{lem:partition}
  applies  to the color-$i$ reachability relation: concatenation
  and inflation preserve the requirement that every exposed ordered end
  pair has color $i$.  We therefore obtain at most $1/\delta<3$ parts, each
  of size at least $(\delta-o(1))n\ge n/3$, and each
  $(\beta,2,i)$-closed.
\end{proof}

We now prove the absorbing lemma in the form needed for the final step.
The additional forbidden set $F$ will later contain vertices that must be kept away from the absorbers, including the exceptional set and the reserved monochromatic cycles.

\begin{lemma}\label{absorber2}
  Let $1/n \ll \rho \ll \zeta \ll \eta \ll \nu \ll \gamma
  \ll \beta \ll \beta_0 \ll \min\{\alpha,d\}
  \le \max\{\alpha,d\} \ll 1$.
  Let $H$ be an $n$-vertex $(\rho,d)$-dense $3$-graph with $\delta_2(H)\ge(1/3+\alpha)n$.
  Let $A\subseteq V(H)$ be a reservoir set satisfying $\eta n\le |A|\le 2\eta n$ and $\deg_A(uv)\ge (1/3+\alpha/2)|A|$ for every pair $uv\in\binom{V(H)}2$.
  Suppose that Case~\ref{case3} holds with respect to $A$.
  Let $\phi$ be the $A$-connected coloring determine by $\mathcal C_1, \mathcal C_2$, let $W$ be the exceptional set from \cref{prop:coloritem}, and let $i\in\{1,2\}$ be the spanning color on $V(H)\setminus W$.
  Let $F\subseteq V(H)$ satisfy $|F|\le \nu n$ and
  $|A\cap(F\cup W)|\le\zeta n/2$.
  Then one can find a partition $\mathcal P_0$ of $V(H)$ into either one
  or two parts and a tight path $P_{\rm abs}$ in $H-(F\cup W)$ such that
  the following hold.
  \begin{enumerate}[label=\rm(\roman*)]
    \item Every part of $\mathcal P_0$ has size at least $n/3$.

    \item $|V(P_{\rm abs})|\le \beta n$ and
      $|V(P_{\rm abs})\cap A|\le\zeta^2n$.

    \item Both ordered end pairs of $P_{\rm abs}$ have color $i$.

    \item For every set $L\subseteq V(H)\setminus V(P_{\rm abs})$ with $|L|\le \gamma n$ and $|L|\in3\mathbb N$ which, when $\mathcal P_0=\{X,Y\}$ and $L\ne\emptyset$, also satisfies the balance condition
      \[
        \frac{|L\cap X|}{|L|},\frac{|L\cap Y|}{|L|}
        \in
        \left[
          \frac13+\frac\alpha4,\,
          \frac23-\frac\alpha4
        \right],
      \]
      there exists a tight path $P_{\rm abs}'$ with the same ordered end
      pairs as $P_{\rm abs}$ and
      $V(P_{\rm abs}')=V(P_{\rm abs})\cup L$.
  \end{enumerate}
\end{lemma}
\begin{proof}
  Choose auxiliary constants so that
  \[
    \max\{\rho,\zeta,\eta,\nu,\gamma\}
    \ll\mu\ll\lambda\ll\theta\ll\varepsilon^2
    \ll\varepsilon\ll\beta\ll\beta_0\ll\min\{\alpha,d\}.
  \]
  Let \(B=V(H)\setminus W\), and let \(H_i\) be the \(3\)-graph on
  \(B\) whose edges are those \(xyz\in E(H[B])\) for which
  \(\phi(xy)=\phi(xz)=\phi(yz)=i\).

  Apply \cref{lem:partitionnonswitcher}, with \(\lambda\) as its closure
  parameter, to obtain a partition
  \(\mathcal P_0\) of \(V(H)\) into one or two parts, each of size at least
  \(n/3\), and each \((\lambda,2,i)\)-closed in \(H-A\).
  We shall use only those edge distributions that occur many times.
  We call an index vector \(\mathbf v\) valid if \(H_i\) has at least
  \(\mu n^3\) edges \(e\) with
  \(\mathbf i_{\mathcal P_0}(e)=\mathbf v\).

  \begin{claim}\label{claim:nonswitcher-valid-index-vectors}
    If \(\mathcal P_0\) has one part, then \((3)\) is valid.  If
    \(\mathcal P_0=\{X,Y\}\), then both \((2,1)\) and \((1,2)\) are
    valid.
  \end{claim}

  \begin{proof}
    Let \(V_0\subseteq B\) be the set given by
    \cref{lem:monosetexist}; thus \(|V_0|\ge\beta n\), and every pair in
    \(B\) incident with a vertex of \(V_0\) has color \(i\).

    If \(\mathcal P_0\) has one part, then
    \[
      e_H(V_0,B,B)
      \ge d|V_0||B|^2-\rho n^3
      \ge10\mu n^3.
    \]
    Since each edge is counted at most six times, \(H[B]\) has at least
    \(\mu n^3\) edges containing a vertex of \(V_0\).  Every such edge
    belongs to \(H_i\), because it is monochromatic and its two pairs
    incident with \(V_0\) have color \(i\).  Thus \((3)\) is valid.

    Suppose now that \(\mathcal P_0=\{X,Y\}\).  Put
    \(X_B=X\cap B\) and \(Y_B=Y\cap B\); both have size at least \(n/4\).
    One of \(V_0\cap X\) and \(V_0\cap Y\) has size at least
    \(\beta n/2\).  If \(|V_0\cap X|\ge\beta n/2\), uniform density gives
    \[
      \begin{split}
        e_H(V_0\cap X,X_B,Y_B)&\ge10\mu n^3,\\
        e_H(V_0\cap X,Y_B,Y_B)&\ge10\mu n^3.
      \end{split}
    \]
    Consequently \(H_i\) has at least \(\mu n^3\) edges of each of the
    index types \((2,1)\) and \((1,2)\).  If
    \(|V_0\cap Y|\ge\beta n/2\), use instead
    \[
      (V_0\cap Y,X_B,X_B)
      \quad\text{and}\quad
      (V_0\cap Y,X_B,Y_B).
    \]
    This proves that both mixed vectors are valid.
  \end{proof}

  Put \(Z_0=F\cup W\), \(J=H_i-(A\cup Z_0)\), and
  \(\Omega=\Omega_i\).  Deleting all edges incident with \(A\cup Z_0\)
  destroys at most \(O((\eta+\nu)n^3)\) edges of any fixed index.
  Consequently, after replacing \(\mu\) by \(\mu/2\), the index set
  \[
    \mathcal I=
    \begin{cases}
      \{(3)\},&|\mathcal P_0|=1,\\
      \{(2,1),(1,2)\},&\mathcal P_0=\{X,Y\},
    \end{cases}
  \]
  lies in \(\mathcal I_{\mu/2}(J,\mathcal P_0)\).
  The graph \(J\) is \(\Omega\)-supported.  Moreover every member of
  \(\Omega_i\), and also its reverse, lies in \(\mathcal C_i\).  Let
  \(L_{\rm RCC}\) be the path-length bound given by
  \cref{equivalentclass}.  The constant map
  \(\kappa_i:\Omega_i\to[1]\) makes
  \((\Omega_i,\kappa_i)\)
  \((A,\zeta,L_{\rm RCC})\)-RCC-connectable by
  \cref{equivalentclass}.  Finally,
  \[
    \delta_2\bigl(H-(A\cup Z_0)\bigr)
    \ge(1/3+\alpha/2)|V(H)\setminus(A\cup Z_0)|
  \]
  by the hierarchy.  Since $L_{\rm RCC}$ is independent of $n$, it is at
  most $\zeta^2n$ for sufficiently large $n$.  Hence
  \cref{prop:3graph-rcc-cover}, applied to $H$ with $Z=Z_0$, $H_0=J$,
  $\sigma=\alpha/2$, $\chi=\zeta$, and $\omega=\zeta^2$, keeping $\theta$
  and $\varepsilon$ as above, gives the
  $(A,Z_0;\theta,\varepsilon,\zeta^2)$-RCC cover property.
  Apply \cref{lem:indexed-absorber} with $k=3$ and
  \(\mathcal P_0,\Omega_i,\kappa_i,J,\mathcal I\), using \(3\nu\) as its
  forbidden-set parameter.  Since all pairs in $\Omega_i$ belong to the same RCC,
  the lemma gives one color-\(i\) path \(P_{\rm abs}\), of size at most
  \(\varepsilon n\le\beta n\) and using at most $\zeta^2n$ vertices of
  $A$.

  It remains only to check that the balance condition implies membership in
  \(\mathsf M(\mathcal I)\).  Recall that
  \(\mathbf w\in\mathsf M(\mathcal I)\) means that
  \(\mathbf w=\sum_{\mathbf v\in\mathcal I}a_{\mathbf v}\mathbf v\) for
  some coefficients \(a_{\mathbf v}\in\mathbb Z_{\ge0}\).
  If \(|\mathcal P_0|=1\), then
  \(|L|\in3\mathbb N\) gives
  \(\mathbf i_{\mathcal P_0}(L)\in\mathsf M(\{(3)\})\).
  Suppose that \(\mathcal P_0=\{X,Y\}\), and put
  \[
    x=|L\cap X|,\qquad y=|L\cap Y|,\qquad q=|L|/3.
  \]
  A decomposition into \(a\) triples of type \(XXY\) and \(b\) triples
  of type \(XYY\) must, and can only, have
  \[
    a=\frac{2x-y}{3}=x-q,
    \qquad
    b=\frac{2y-x}{3}=y-q.
  \]
  Hence such a decomposition exists exactly when
  \(x+y\equiv0\pmod3\) and \(x/2\le y\le2x\), equivalently
  \(q\le x,y\le2q\).  The balance hypothesis gives these inequalities, so
  \[
    \mathbf i_{\mathcal P_0}(L)
    =a(2,1)+b(1,2)\in\mathsf M(\mathcal I).
  \]
  The absorbing conclusion is now
  \cref{lem:indexed-absorber}\textup{(ii)}.
\end{proof}

As in the switcher case, the last step needs short tight cycles whose lengths allow a modulo-$3$ adjustment.
Here the cycles must be compatible with the color-$i$ connecting structure, so we use the following monochromatic version.
Since the proof is a standard modification of the argument of Piga, Sanhueza-Matamala and Schacht~\cite{PIGA20261}, together with removal and supersaturation, we defer it to \cref{app:monochromatic-c10}.

\begin{lemma}\label{lem:monochromatic_c10}
  Let $1/n \ll \rho \ll \zeta \ll \eta \ll \beta \ll \beta_1 \ll \min\{\alpha,d\} \le \max\{\alpha,d\} \ll 1$.
  Let $H$ be an $n$-vertex $(\rho,d)$-dense $3$-graph with $\delta_2(H) \ge (1/3+\alpha)n$.
  Let $A \subseteq V(H)$ be a reservoir set satisfying $\eta n\le |A|\le 2\eta n$ and, for every pair $uv \in \binom{V(H)}{2}$, we have $\deg_A(uv) \ge (1/3+\alpha/2)|A|$.
  Let $\mathcal C_1,\mathcal C_2$ be the RCCs given by \cref{equivalentclass} with codegree slack $\alpha/2$ and $\phi$ be the $A$-connected coloring determined by $\mathcal C_1,\mathcal C_2$.
  Suppose that Case~\ref{case3} holds with respect to $A$.
  Let $W$ be the exceptional vertex set provided by \cref{prop:coloritem}, and let $i\in\{1,2\}$ be the spanning color on $V(H)\setminus W$.
  Then $H[V(H)\setminus(A\cup W)]$ contains two vertex-disjoint copies of $C_{10}^{(3)}[2]$ such that every shadow pair of each copy has color $i$; equivalently, for every edge $xyz$ in either copy, we have $\phi(xy)=\phi(xz)=\phi(yz)=i$.
\end{lemma}

\subsection[Proof of the non-switcher case lemma]{Proof of \cref{nonswitchercase}}

\begin{proof}
  Choose constants in the following order:
  \[
    1/n \ll \rho \ll \zeta \ll \eta\ll\nu
    \ll \gamma_{\rm cov}
    \ll \gamma_{\rm abs}
    \ll \beta \ll \beta_1 \ll \min\{\alpha,d\} \le \max\{\alpha,d\} \ll 1.
  \]
  Let $\phi$ be the $A$-connected coloring determine by $\mathcal C_1, \mathcal C_2$.
  Let $W$ be the exceptional set given by \cref{prop:coloritem}, so $|W|\le6$.
  Fix the spanning color $i\in\{1,2\}$ from \cref{lem:monosetexist}.
  Thus $i$ spans $V(H)\setminus W$, and both orientations of every color-$i$ pair lie in the same robust connection class.

  \proofstep{1}{Reserving monochromatic replaceable segments}
  By \cref{lem:monochromatic_c10}, there are two vertex-disjoint copies $C_1,C_2$ of $C_{10}^{(3)}[2]$ in $V(H)\setminus(A\cup W)$ such that every shadow pair of each copy has color $i$.
  For $j\in[2]$, write
  \[
    V(C_j)=\{x_1^j,x_2^j,\dots,x_{10}^j,y_1^j,\dots,y_{10}^j\},
  \]
  and let
  \begin{align*}
    R_j \coloneqq x_1^jx_2^jx_3^jx_4^jx_5^j
    \quad\text{and}\quad
    R_j^+ \coloneqq x_1^jx_2^jy_3^jy_4^j\cdots y_{10}^jy_1^jy_2^jx_3^jx_4^jx_5^j.
  \end{align*}
  Then $R_j$ and $R_j^+$ are tight paths with the same ordered end pairs $(x_2^j,x_1^j)$ and $(x_4^j,x_5^j)$, and $|V(R_j^+)|-|V(R_j)|=10$.
  Moreover, these ordered end pairs have color $i$.
  Hence any tight cycle containing $R_j$ as a segment remains a tight cycle after replacing $R_j$ by $R_j^+$.

  \proofstep{2}{Constructing the absorber}
  Apply \cref{absorber2} with $F=V(C_1\cup C_2)$, forbidden-set
  parameter $\nu$, and absorption parameter $\gamma_{\rm abs}$.
  We obtain a partition $\mathcal P_0$ of $V(H)$ into either one or two
  parts and one tight path $P_{\rm abs}$ in
  $H-(W\cup V(C_1\cup C_2))$ such that
  $|V(P_{\rm abs})|\le\beta n$, both ordered end pairs have color $i$,
  $|V(P_{\rm abs})\cap A|\le\zeta^2n$,
  and $P_{\rm abs}$ has the absorbing property stated in
  \cref{absorber2} with respect to $\mathcal P_0$.

  \proofstep{3}{Choosing a balancing set}
  We choose a set $A_{\rm bal}\subseteq V(H)\setminus(A\cup W\cup V(C_1\cup C_2)\cup V(P_{\rm abs}))$ such that $\gamma_{\rm abs}n/3\le |A_{\rm bal}|\le \gamma_{\rm abs}n/2$.
  If $\mathcal P_0=\{X,Y\}$ has two parts, we additionally choose $A_{\rm bal}$ so that
  \[
    \frac{|A_{\rm bal}\cap X|}{|A_{\rm bal}|},
    \frac{|A_{\rm bal}\cap Y|}{|A_{\rm bal}|}
    \in
    \left[
      \frac13+\frac\alpha3,\,
      \frac23-\frac\alpha3
    \right].
  \]
  Such a choice is possible because each part has size at least $n/3$, while the previously forbidden set has size at most $2\beta n$; hence each part still contains linearly many available vertices, far more than the $O(\gamma_{\rm abs}n)$ vertices requested here.

  \proofstep{4}{Covering almost all remaining vertices}
  Apply \cref{lem:longpath} with the fixed color $i$ and with $B_{\rm used}=V(P_{\rm abs})\cup V(C_1\cup C_2)\cup A_{\rm bal}$.
  Since $|B_{\rm used}|\le \beta n+40+\gamma_{\rm abs}n/2\le \beta_1 n$
  and $|A\cap B_{\rm used}|\le\zeta^2n\le\zeta n/2$ by the hierarchy,
  \cref{lem:longpath} gives two vertex-disjoint tight paths $Q_1,Q_2$ in
  $H-(W\cup B_{\rm used})$ whose ordered end pairs have color $i$, and
  which cover all but at most $\gamma_{\rm cov}n$ vertices of
  $V(H)\setminus(A\cup W\cup B_{\rm used})$.
  Let $L_{\rm cov}$ denote this uncovered set, so $|L_{\rm cov}|\le\gamma_{\rm cov}n$.

  \proofstep{5}{Connecting the paths into a cycle}
  Arrange the five paths $P_{\rm abs},Q_1,Q_2,R_1,R_2$ in an arbitrary cyclic
  order.  Every end pair has color $i$, so every attachment pair also lies
  in $\mathcal C_i$.  Use four robust connections between consecutive
  segments and \cref{lem:gluing}(i) to obtain one tight path containing all
  five segments.  A fifth robust connection between its two attachment
  pairs closes it into a tight cycle by \cref{lem:gluing}(ii).

  At each step, the forbidden set in the definition of robust tight
  connectedness contains all reservoir vertices already used except the
  vertices in the two attachment pairs selected for that connection.
  The paths $Q_1,Q_2$ may already contain at most $\zeta |A|/2$ vertices
  from $A$, coming from the connections used in \cref{lem:longpath}, and
  $P_{\rm abs}$ uses at most $\zeta^2n$ reservoir vertices in its final
  RCC cover.  In the present step we use at most five further
  reservoir connections.  Since
  $|A|\le2\eta n$ and $\eta\ll1$, the total number of reservoir vertices
  that must be avoided is at most
  $6\zeta^2n+\zeta |A|/2\le\zeta n$.
  This yields a tight cycle $\Gamma$ containing
  $P_{\rm abs},Q_1,Q_2,R_1,R_2$ as subpaths.

  Let $L_0=V(H)\setminus V(\Gamma)$.
  Then $L_0$ is contained in the union of the reservoir vertices not used by the connecting paths, the exceptional set $W$, the balancing set $A_{\rm bal}$, the leftover set $L_{\rm cov}$, and the vertices of $C_1\cup C_2$ not covered by $R_1\cup R_2$.
  Hence
  \[
    |L_0|\le 2\eta n+6+|A_{\rm bal}|+\gamma_{\rm cov}n+30
    \le \gamma_{\rm abs}n.
  \]

  \proofstep{6}{Adjusting divisibility}
  Replacing $R_j$ by $R_j^+$ increases the number of covered vertices by $10$, and therefore decreases the size of the leftover set by $10$.
  Since $10\equiv1\pmod3$, choose $J\subseteq\{1,2\}$ such that $|J|\equiv |L_0|\pmod3$.
  For every $j\in J$, replace the segment $R_j$ of $\Gamma$ by $R_j^+$.
  Let the resulting tight cycle be $\Gamma'$, and put $L=V(H)\setminus V(\Gamma')$.
  Then $|L|\in3\mathbb N$ and $|L|\le |L_0|\le \gamma_{\rm abs}n$.

  Note that $A_{\rm bal}\subseteq L$.
  If $\mathcal P_0$ has two parts, say $\mathcal P_0=\{X,Y\}$, then $L$ satisfies the balance condition required by \cref{absorber2}.
  Indeed, $A_{\rm bal}$ has both part ratios in $\left[ \frac13+\frac\alpha3,\, \frac23-\frac\alpha3 \right]$,
  while the contribution from $L\setminus A_{\rm bal}$ has size at most $2\eta n+\gamma_{\rm cov}n+O(1)$.
  Since $\max\{\eta,\gamma_{\rm cov}\}\ll\gamma_{\rm abs}\ll\alpha$, and $|A_{\rm bal}|\ge \gamma_{\rm abs}n/3$, it follows that
  \[
    \frac{|L\cap X|}{|L|},
    \frac{|L\cap Y|}{|L|}
    \in
    \left[
      \frac13+\frac\alpha4,\,
      \frac23-\frac\alpha4
    \right].
  \]
  If $\mathcal P_0$ has one part, no balance condition is needed.

  \proofstep{7}{Absorbing the leftover}
  The cycle $\Gamma'$ contains $P_{\rm abs}$ as a subpath.  By the
  absorbing property of \cref{absorber2}, applied to $L$, there is a path
  $P_{\rm abs}'$ with the same ordered end pairs and
  $V(P_{\rm abs}')=V(P_{\rm abs})\cup L$.  Replacing $P_{\rm abs}$ by
  $P_{\rm abs}'$ in $\Gamma'$ gives a tight Hamiltonian cycle of $H$.
\end{proof}

We now are ready to prove \cref{thm:main_codegree}.

\begin{proof}[Proof of \cref{thm:main_codegree}]
  Fix $d,\alpha>0$.  Since decreasing either $d$ or $\alpha$ only weakens
  the hypotheses, we may replace them by smaller positive constants and
  assume throughout that $\max\{d,\alpha\}\le 10^{-3}$.
  Choose constants so that all applications below are valid and
  \(
    1/n\ll \rho\ll \zeta\ll \eta\ll \min\{\alpha,d\}.
  \)
  In particular, choose $\zeta$ below the threshold in
  \cref{equivalentclass,dividecase}, and choose $\rho$ below the thresholds
  required in \cref{maincase1,nonswitchercase}.

  Let $H$ be an $n$-vertex $(\rho,d)$-dense $3$-graph with
  $\delta_2(H)\ge(1/3+\alpha)n$.
  Put $m=\lceil\eta n\rceil$ and choose $A$ uniformly at random among all
  $m$-subsets of $V(H)$.  Thus
  \(
    \eta n\le |A|\le 2\eta n
  \)
  for all sufficiently large $n$.
  For each pair $uv$, the random variable $\deg_A(u,v)$ is hypergeometric
  with mean
  \(
    \frac{|A|}{n}\deg_H(u,v)
    \ge \left(\frac13+\alpha\right)|A|.
  \)
  The hypergeometric Chernoff bound and a union bound over the at most $n^2$
  pairs imply that there is a choice of $A$ such that
  \[
    \deg_A(u,v)\ge
    \left(\frac13+\frac\alpha2\right)|A|
    \qquad\text{for every }uv\in\binom{V(H)}2.
  \]

  Apply \cref{equivalentclass} with parameters $\alpha/2$ and $\eta$ to
  obtain at most two robust connection components, and then apply
  \cref{dividecase}.
  If Case~\ref{case1} or Case~\ref{case2} occurs, \cref{maincase1} gives a
  tight Hamiltonian cycle in $H$.
  If Case~\ref{case3} occurs, \cref{nonswitchercase} gives a tight Hamiltonian
  cycle in $H$.
  These cases exhaust all possibilities.
\end{proof}
\section{Lower bound constructions}\label{const}
This section gives the lower-bound constructions used to show that the main thresholds are asymptotically sharp.
The common idea is to start with a random $2$-edge-coloring of the complete graph and then take hyperedges from the monochromatic triangles.
In the vertex-degree construction, we additionally delete one class of monochromatic triangles to create a divisibility obstruction.

The following concentration lemma is used repeatedly to check uniform density in these random coloring constructions.
It is stated separately because the triangle indicators are not independent.

\begin{lemma}\label{lem:coloring_density_concentration}
  Let $V$ be an $n$-vertex set and let $S\subseteq V$ be fixed with $|S|\le 1$.
  Suppose that a $3$-graph $H$ on $V$ is obtained from an edge-coloring of the complete graph on $V$ in which the colors of the edges in a fixed set $F\subseteq\binom V2$ are prescribed and the colors of the remaining edges are mutually independent random variables.
  Assume that there is a constant $p_0>0$ such that every triple $T\in\binom{V\setminus S}{3}$ becomes an edge of $H$ with probability at least $p_0$.
  Then, for every $\rho>0$, asymptotically almost surely,
  $e_H(U_1,U_2,U_3) \ge p_0|U_1||U_2||U_3|-\rho n^3$
  for all sets $U_1,U_2,U_3\subseteq V$, not necessarily pairwise disjoint.
\end{lemma}

\begin{proof}
  Fix sets $U_1,U_2,U_3\subseteq V$ and let
  $Z\coloneqq e_H(U_1,U_2,U_3)$.
  Among the ordered triples in $U_1\times U_2\times U_3$, at most $3n^2$
  have two equal entries and at most $3n^2$ contain the possible vertex of
  $S$.  Thus at least $|U_1||U_2||U_3|-6n^2$ of them have distinct entries
  and avoid $S$.
  Hence
  \[
    \mathbb{E}[Z]\ge p_0\bigl(|U_1||U_2||U_3|-6n^2\bigr).
  \]

  Changing the color of one random pair can affect at most $6n$ ordered
  triples counted by $Z$.
  By \cref{lem:bounded_differences}, for every fixed $\tau>0$,
  \[
    \mathbb{P}\left[Z<\mathbb{E}[Z]-\tau n^3\right]
    \le \exp(-\Omega_\tau(n^2)).
  \]
  There are at most $8^n$ choices for the ordered triple $(U_1,U_2,U_3)$.
  Taking $\tau=\rho/2$ and applying a union bound, we get that asymptotically almost surely, for all such triples $(U_1,U_2,U_3)$,
  \[
    Z\ge p_0|U_1||U_2||U_3|-6p_0n^2-\frac{\rho}{2}n^3.
  \]
  Since $6p_0n^2\le \rho n^3/2$ for all sufficiently large $n$, the desired inequality follows.
\end{proof}

Before treating the individual constructions, we record a simple consequence of using monochromatic triangles as hyperedges.

\begin{fact}\label{fact:monochromatic-tight-cycle}
  Let $H$ be a $3$-graph obtained from an edge-coloring of a complete graph by taking some monochromatic triangles as hyperedges.
  Then every tight cycle in $H$ is monochromatic in the underlying edge-coloring.
\end{fact}

\begin{proof}
  Two consecutive hyperedges of a tight cycle share a graph edge.
  Since each hyperedge is a monochromatic triangle, the two hyperedges have
  the same color.  This equality propagates around the cycle.
\end{proof}

\subsection[Proof of the codegree lower bound]{Proof of \cref{thm:codegree_lower_bound}}

The first construction shows that the threshold in \cref{thm:main_codegree} cannot be lowered below $n/3$ asymptotically.
The construction is shown in \cref{coloring3unicodegree}.

\begin{figure}[htbp]\centering
  \begin{tikzpicture}[scale=0.6, transform shape]\tikzstyle{set} = [draw, circle, minimum size=3.2cm, thick, fill=gray!5]\tikzstyle{vertex} = [fill, circle, inner sep=1.8pt]\tikzstyle{rededge} = [draw, red, thick]\tikzstyle{blueedge} = [draw, blue, thick]\tikzstyle{reddashed} = [draw, red, thick, dashed]\tikzstyle{bluedashed} = [draw, blue, thick, dashed]

    \node[set] (X) at (0, 3.8) {};
    \node[set] (Y) at (-3.5, 0) {};
    \node[set] (Z) at (3.5, 0) {};

    \node at (0, 5.8) {\Large \textbf{$X$}};
    \node at (-5, -1.5) {\Large \textbf{$Y$}};
    \node at (5, -1.5) {\Large \textbf{$Z$}};

    \node[vertex] (v) at (0, 0) {};
    \node[above right, font=\large] at (v) {$v$};

    \node[vertex] (x1) at (-0.8, 3.8) {};
    \node[vertex] (x2) at (0.8, 3.8) {};
    \node[above left] at (x1) {$x_1$};
    \node[above right] at (x2) {$x_2$};

    \node[vertex] (y1) at (-2.5, -0.8) {};
    \node[vertex] (y2) at (-4.2, -0.2) {};
    \node[above right] at (y1) {$y_1$};
    \node[below left] at (y2) {$y_2$};

    \node[vertex] (z1) at (2.5, -0.8) {};
    \node[vertex] (z2) at (4.2, -0.2) {};
    \node[above left] at (z1) {$z_1$};
    \node[below right] at (z2) {$z_2$};

    \draw[rededge] (v) -- (x1);
    \draw[rededge] (v) -- (x2);
    \draw[rededge] (v) -- (y1);
    \draw[rededge] (v) -- (z1);

    \draw[blueedge] (x1) -- (y1);
    \draw[blueedge] (x2) -- (z1);
    \draw[blueedge] (y1) to[bend right=20] (z1);

    \draw[reddashed] (x1) to[bend left=45] node[above, text=black, font=\small] {$1-\varepsilon$} (x2);
    \draw[bluedashed] (x1) to[bend right=45] node[below, text=black, font=\small] {$\varepsilon$} (x2);

    \draw[reddashed] (y2) to[bend left=45] (y1);
    \draw[bluedashed] (y2) to[bend right=45] (y1);

    \draw[reddashed] (z1) to[bend left=45] (z2);
    \draw[bluedashed] (z1) to[bend right=45] (z2);
  \end{tikzpicture}
  \caption{Illustration of the random $2$-coloring $\chi$. Edges incident to $v$ are red, cross-edges between different parts are blue, and internal edges are red with probability $1-\varepsilon$ and blue with probability $\varepsilon$. The hyperedges of $H$ are the monochromatic triangles under $\chi$.}
  \label{coloring3unicodegree}

\end{figure}
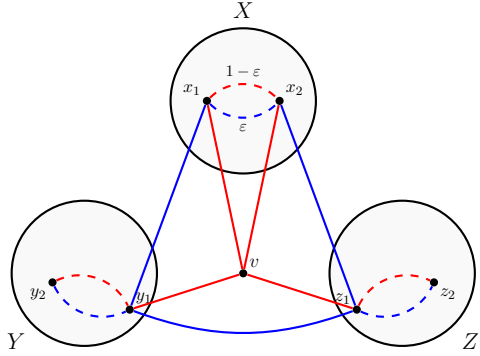
\begin{proof}[Proof of \cref{thm:codegree_lower_bound}]
  Choose $0<\varepsilon\ll\gamma$ and set $d=\varepsilon^3$.
  Let $V=X\cup Y\cup Z\cup\{v\}$, where $|X|,|Y|,|Z|$ are as equal as possible.
  Thus each of $|X|,|Y|,|Z|$ is at least $n/3-2$.

  We color the pairs of $V$ as follows.
  Pairs inside $X,Y,Z$ are colored red with probability $1-\varepsilon$ and blue with probability $\varepsilon$, independently.
  Pairs between two different parts among $X,Y,Z$ are colored blue.
  Pairs incident to $v$ are colored red.
  Let $H$ be the $3$-graph whose edges are the monochromatic triangles.

  We first check the minimum codegree.
  Let $\{x,y\}$ be a pair of vertices.
  If $x$ and $y$ lie in two different parts among $X,Y,Z$, then every vertex in the third part completes a blue triangle with $x$ and $y$.
  Hence $\deg_H(x,y)\ge n/3-2$.

  Suppose next that $x,y\in X$; the cases inside $Y$ and $Z$ are the same.
  If $\chi(xy)=B$, then every vertex of $Y\cup Z$ completes a blue triangle with $x$ and $y$, so $\deg_H(x,y)\ge 2n/3-4$.
  If $\chi(xy)=R$, then the number of vertices $z\in X\setminus\{x,y\}$ such that $xz$ and $yz$ are both red is distributed as $\mathrm{Bin}(|X|-2,(1-\varepsilon)^2)$.
  Since $\varepsilon\ll\gamma$, Chernoff's bound implies that this number is at least $(1/3-\gamma)n$ with probability $1-\exp(-\Omega(n))$.

  Finally, suppose the pair is $\{v,x\}$ with $x\in X$; the cases $x\in Y$ and $x\in Z$ are identical.
  The number of vertices $z\in X\setminus\{x\}$ such that $xz$ is red is distributed as $\mathrm{Bin}(|X|-1,1-\varepsilon)$.
  Again by Chernoff's bound, this number is at least $(1/3-\gamma)n$ with probability $1-\exp(-\Omega(n))$.

  Taking a union bound over all pairs, we get $\delta_2(H)\ge (1/3-\gamma)n$ asymptotically almost surely.

  We now check uniform density.
  Every triple contained in $V\setminus\{v\}$ becomes a hyperedge with probability at least $\varepsilon^3$.
  Indeed, a triple inside one part is monochromatic with probability $(1-\varepsilon)^3+\varepsilon^3$, a triple meeting exactly two parts is blue with probability $\varepsilon$, and a triple meeting all three parts is always blue.
  Applying \cref{lem:coloring_density_concentration} with $S=\{v\}$ and $p_0=\varepsilon^3$, we obtain that $H$ is $(\rho,d)$-dense asymptotically almost surely.

  It remains to rule out tight Hamiltonian cycles.
  Suppose that $H$ contains a tight Hamiltonian cycle $C$.
  By \cref{fact:monochromatic-tight-cycle}, all hyperedges of $C$ have the same color.
  If $C$ is blue, then it cannot contain $v$, because $v$ is isolated in the blue graph.

  If $C$ is red, then the cyclic ordering of the vertices of $C$ gives a spanning cycle in the underlying red graph: every pair of consecutive vertices of $C$ is contained in a red hyperedge and hence is a red graph edge.
  However, after deleting $v$, the red graph has no edges between the parts $X,Y,Z$.
  Thus $v$ is a cut vertex of the red graph, and a graph with a cut vertex cannot contain a spanning cycle.
  This is a contradiction.

  The codegree and density events both hold with probability $1-o(1)$.
  Hence there exists a coloring with all the required properties.
\end{proof}

\subsection[Proof of the vertex-degree lower bound]{Proof of \cref{thm:vertex_degree_lower_bound}}

We next show that uniform density does not lower the vertex-degree threshold.
Here the obstruction comes from an unbalanced two-part partition.
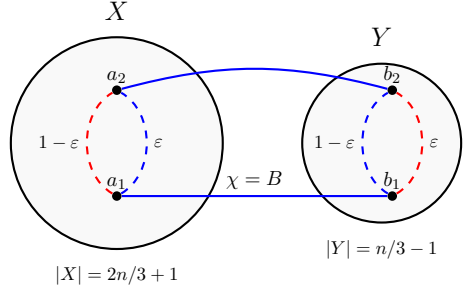
\begin{figure}[htbp]\centering
  \begin{tikzpicture}[scale=0.7, transform shape]\tikzstyle{setA} = [draw, circle, minimum size=4cm, thick, fill=gray!5]\tikzstyle{setB} = [draw, circle, minimum size=3cm, thick, fill=gray!5]\tikzstyle{vertex} = [fill, circle, inner sep=1.8pt]\tikzstyle{blueedge} = [draw, blue, thick]\tikzstyle{reddashed} = [draw, red, thick, dashed]\tikzstyle{bluedashed} = [draw, blue, thick, dashed]

    \node[setA] (A) at (-2.5, 0) {};
    \node[setB] (B) at (2.5, 0) {};

    \node at (-2.5, 2.5) {\Large \textbf{$X$}};
    \node at (2.5, 2) {\Large \textbf{$Y$}};
    \node at (-2.5, -2.5) {\small $|X|=2n/3+1$};
    \node at (2.5, -2) {\small $|Y|=n/3-1$};

    \node[vertex] (a1) at (-2.5, -1) {};
    \node[vertex] (a2) at (-2.5, 1) {};
    \node[above] at (a1) {$a_1$};
    \node[above] at (a2) {$a_2$};

    \node[vertex] (b1) at (2.7, -1) {};
    \node[vertex] (b2) at (2.7, 1) {};
    \node[above] at (b1) {$b_1$};
    \node[above] at (b2) {$b_2$};

    \draw[blueedge] (a1) -- (b1) node[midway, above, text=black] {$\chi=B$};
    \draw[blueedge] (a2) to[bend left=15] (b2);

    \draw[reddashed] (a1) to[bend left=60] node[left, text=black, font=\small] {$1-\varepsilon$} (a2);
    \draw[bluedashed] (a1) to[bend right=60] node[right, text=black, font=\small] {$\varepsilon$} (a2);

    \draw[bluedashed] (b1) to[bend left=60] node[left, text=black, font=\small] {$1-\varepsilon$} (b2);
    \draw[reddashed] (b1) to[bend right=60] node[right, text=black, font=\small] {$\varepsilon$} (b2);
  \end{tikzpicture}
  \caption{Illustration of the random $2$-coloring $\chi$. Cross-edges are blue. Inside $X$, edges are red with probability $1-\varepsilon$. Inside $Y$, edges are blue with probability $1-\varepsilon$. The hyperedges of $H$ are the monochromatic triangles under $\chi$, except that blue triangles contained entirely in $X$ are deleted.}
  \label{fig:construction_vertex_degree}

\end{figure}
\begin{proof}[Proof of \cref{thm:vertex_degree_lower_bound}]
  Choose $0<\varepsilon\ll\gamma$ and set $d=\varepsilon$.
  Let $V=X\cup Y$, where $|X|=\lfloor 2n/3\rfloor+1$ and $|Y|=n-|X|=\lceil n/3\rceil-1$.
  Color pairs inside $X$ red with probability $1-\varepsilon$ and blue with probability $\varepsilon$, color pairs inside $Y$ blue with probability $1-\varepsilon$ and red with probability $\varepsilon$, and color all pairs in $X\times Y$ blue.
  Let $H$ consist of all monochromatic triangles under this coloring, except that blue triangles contained entirely in $X$ are deleted.

  We first check the minimum vertex degree.
  For $x\in X$, we count red triangles inside $X$ and blue triangles with $x$ and two vertices of $Y$.
  This gives
  \[
    \mathbb{E}[\deg_H(x)]
    \ge
    \binom{|X|-1}{2}(1-\varepsilon)^3
    +\binom{|Y|}{2}(1-\varepsilon)
    \ge
    \left(\frac{5}{9}-\frac{\gamma}{2}\right)\binom{n}{2}
  \]
  for sufficiently large $n$ and $\varepsilon\ll\gamma$.
  This is only a lower bound, since there may also be blue triangles of type $XXY$.

  For $y\in Y$, we count blue triangles inside $Y$ and blue triangles with one vertex in $X$ and two vertices in $Y$.
  Thus
  \[
    \mathbb{E}[\deg_H(y)]
    \ge
    \binom{|Y|-1}{2}(1-\varepsilon)^3
    +|X|(|Y|-1)(1-\varepsilon)
    \ge
    \left(\frac{5}{9}-\frac{\gamma}{2}\right)\binom{n}{2}.
  \]
  Fix a vertex $v$.  Changing the color of a random pair incident with
  $v$ can affect $\deg_H(v)$ by at most $n$, and there are at most $n-1$
  such random pairs.  Changing the color of a random pair not incident with
  $v$ affects at most the single triple obtained by adding $v$, and hence
  changes $\deg_H(v)$ by at most $1$.  Therefore
  \[
    \sum_e c_e^2\le(n-1)n^2+\binom{n-1}{2}=O(n^3).
  \]
  By \cref{lem:bounded_differences}, a deviation of order $\gamma n^2$
  has probability $\exp(-\Omega_\gamma(n))$.
  A union bound over all vertices gives $\delta_1(H)\ge \left(\frac{5}{9}-\gamma\right)\binom{n}{2}$ asymptotically almost surely.

  We now check uniform density.
  Every triple of vertices becomes an edge of $H$ with probability at least $\varepsilon$.
  For triples inside $X$, a red triangle appears with probability $(1-\varepsilon)^3\ge \varepsilon$.
  For triples with two vertices in $X$ and one in $Y$, the two cross-pairs are blue and the internal pair of $X$ is blue with probability $\varepsilon$.
  The other cases have probability at least $(1-\varepsilon)^3$.
  Applying \cref{lem:coloring_density_concentration} with $S=\emptyset$ and $p_0=\varepsilon$, we obtain that $H$ is $(\rho,d)$-dense asymptotically almost surely.

  It remains to rule out tight Hamiltonian cycles.
  Suppose that $H$ contains a tight Hamiltonian cycle $C$.
  By \cref{fact:monochromatic-tight-cycle}, all hyperedges of $C$ have the same color.
  If $C$ is red, then the underlying graph cycle of $C$ lies in the red graph.
  This is impossible because all pairs between $X$ and $Y$ are blue, so the red graph is disconnected.

  If $C$ is blue, then every hyperedge of $C$ contains at least one vertex from $Y$, since all blue triangles contained entirely in $X$ were deleted.
  A tight Hamiltonian cycle on $n$ vertices has exactly $n$ hyperedges, and each vertex lies in exactly three of them.
  Hence
  \[
    3|Y|=\sum_{e\in E(C)}|e\cap Y|\ge n,
  \]
  which implies $|Y|\ge n/3$.
  This contradicts $|Y|=\lceil n/3\rceil-1<n/3$.
  Thus $H$ contains no tight Hamiltonian cycle.

  The degree and density events hold simultaneously with positive probability, and the construction follows.
\end{proof}

\subsection[Proof of the Ara{\'u}jo--Piga--Schacht lower bound]{Proof of \cref{thm:araujo_piga_schacht}}

Finally, we give a second codegree construction.

\begin{figure}[htbp]\centering
  \begin{tikzpicture}[scale=0.7, transform shape]\tikzstyle{set} = [draw, circle, minimum size=3.5cm, thick, fill=gray!5]\tikzstyle{vertex} = [fill, circle, inner sep=1.8pt]\tikzstyle{rededge} = [draw, red, thick]\tikzstyle{blueedge} = [draw, blue, thick]\tikzstyle{reddashed} = [draw, red, thick, dashed]\tikzstyle{bluedashed} = [draw, blue, thick, dashed]

    \node[set] (X) at (-3.5, -1) {};
    \node[set] (Y) at (3.5, -1) {};

    \node at (-3.5, 1.2) {\Large \textbf{$X$}};
    \node at (3.5, 1.2) {\Large \textbf{$Y$}};

    \node[vertex] (v) at (0, 3) {};
    \node[above right, font=\large] at (v) {$v$};

    \node[vertex] (x1) at (-4.5, -1) {};
    \node[vertex] (x2) at (-2.5, -1) {};
    \node[below] at (x1) {$x_1$};
    \node[below] at (x2) {$x_2$};

    \node[vertex] (y1) at (2.5, -1) {};
    \node[vertex] (y2) at (4.5, -1) {};
    \node[below] at (y1) {$y_1$};
    \node[below] at (y2) {$y_2$};

    \draw[rededge] (v) -- (x1);
    \draw[rededge] (v) -- (x2);
    \draw[rededge] (v) -- (y1);
    \draw[rededge] (v) -- (y2);

    \draw[blueedge] (x2) -- (y1) node[midway, above, text=black] {$\chi=B$};

    \draw[reddashed] (x1) to[bend left=60] node[above, text=black, font=\small] {$\sqrt{3}-1$} (x2);
    \draw[bluedashed] (x1) to[bend right=60] node[below, text=black, font=\small] {$2-\sqrt{3}$} (x2);

    \draw[reddashed] (y1) to[bend left=60] node[above, text=black, font=\small] {$\sqrt{3}-1$} (y2);
    \draw[bluedashed] (y1) to[bend right=60] node[below, text=black, font=\small] {$2-\sqrt{3}$} (y2);
  \end{tikzpicture}
  \caption{Illustration of the random $2$-coloring $\chi$. The vertex set is $X\cup Y\cup\{v\}$. Edges incident to $v$ are red, cross-edges between $X$ and $Y$ are blue, and internal edges are red with probability $\sqrt{3}-1$ and blue with probability $2-\sqrt{3}$.}
  \label{fig:construction_simon}
\end{figure}
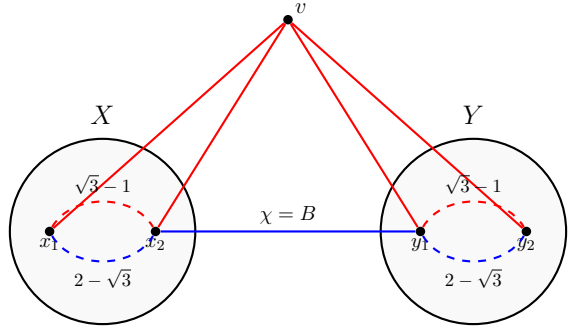
\begin{proof}[Proof of \cref{thm:araujo_piga_schacht}]
  Put $q \coloneqq 2-\sqrt{3}$ and $p \coloneqq \sqrt{3}-1$, so $p+q=1$ and $p^2=2q$.
  Let $V=X\cup Y\cup\{v\}$, where $|X|$ and $|Y|$ are as equal as possible.
  Color pairs inside $X$ and inside $Y$ red with probability $p$ and blue with probability $q$, independently.
  Color all pairs in $X\times Y$ blue, and all pairs incident to $v$ red.
  Let $H$ be the $3$-graph consisting of all monochromatic triangles.

  We first check the minimum codegree.
  Let $\{x,y\}$ be a pair.
  First suppose that $x,y$ lie in the same part, say $X$.
  If $\chi(xy)=R$, then the number of vertices $z\in X\setminus\{x,y\}$ such that $xz$ and $yz$ are both red is distributed as $\mathrm{Bin}(|X|-2,p^2)$.
  Since $p^2|X|=(q-o(1))n$, Chernoff's bound gives at least $(q-\gamma)n$ such vertices asymptotically almost surely.
  If $\chi(xy)=B$, then every vertex of $Y$ completes a blue triangle with $x$ and $y$, so $\deg_H(x,y)\ge |Y|\ge(q-\gamma)n$ for all sufficiently large $n$.

  If $x\in X$ and $y\in Y$, then $\chi(xy)=B$.
  A vertex $z\in X\setminus\{x\}$ completes a blue triangle exactly when $xz$ is blue, and a vertex $z\in Y\setminus\{y\}$ completes a blue triangle exactly when $yz$ is blue.
  The total number of such vertices has expectation $(q-o(1))n$, and Chernoff's bound gives degree at least $(q-\gamma)n$ asymptotically almost surely.

  Finally, suppose that the pair is $\{v,x\}$ with $x\in X$; the case $x\in Y$ is the same.
  The number of vertices $z\in X\setminus\{x\}$ such that $xz$ is red has expectation $(p/2-o(1))n$.
  Since $p/2>q$, Chernoff's bound gives at least $(q-\gamma)n$ such vertices asymptotically almost surely.
  A union bound over all pairs yields $\delta_2(H)\ge(q-\gamma)n$ asymptotically almost surely.

  We now check uniform density.
  We check triples not containing $v$.
  If a triple is contained in $X$ or in $Y$, then it is monochromatic with probability $p^3+q^3>q$.
  If it has vertices in both $X$ and $Y$, then the two cross-pairs are blue, and the triple is blue exactly when the unique internal pair is blue, which happens with probability $q$.
  Therefore every triple in $V\setminus\{v\}$ becomes an edge with probability at least $q$.
  Applying \cref{lem:coloring_density_concentration} with $S=\{v\}$ and $p_0=q$, we obtain that $H$ is $(\rho,q)$-dense asymptotically almost surely.

  It remains to rule out tight Hamiltonian cycles.
  Suppose that $H$ contains a tight Hamiltonian cycle $C$.
  By \cref{fact:monochromatic-tight-cycle}, all hyperedges of $C$ have the same color.
  If $C$ is blue, then it cannot contain $v$, because $v$ is isolated in the blue graph.
  If $C$ is red, then the underlying graph cycle of $C$ lies in the red graph.
  However, after deleting $v$, there are no red pairs between $X$ and $Y$, so $v$ is a cut vertex of the red graph.
  Hence the red graph contains no spanning cycle, a contradiction.

  The codegree and density events hold simultaneously with positive probability, and the construction follows.
\end{proof}

\section{Concluding remarks}\label{sec:concludingremark}

The arguments and constructions above leave several natural directions for further study.
One is to seek analogous Hamiltonicity results for uniformly dense $k$-graphs.
For larger uniformity, both the appropriate absorbing structure and the extremal examples appear less transparent.

Even in the $3$-uniform case, it remains open whether the positive linear slack in the codegree condition is necessary.
\Cref{thm:main_codegree} proves Hamiltonicity above $(1/3+\alpha)n$, while \cref{thm:codegree_lower_bound} shows that the constant $1/3$ cannot be improved asymptotically.
In our proof, the slack above $n/3$ is used at several points, in particular to obtain the large monochromatic set $V_0$ in \cref{lem:monosetexist}.
This leads to the following problem.

\begin{problem}
  Does every sufficiently large $(\rho,d)$-dense $3$-graph $H$ with $\delta_2(H)\ge n/3$ contain a tight Hamiltonian cycle?
\end{problem}



It would also be interesting to extend our result to $k$-uniform hypergraphs. A natural pseudorandomness assumption is full $(k-2)$-layout density, which coincides with uniform density when $k=3$. Under this assumption, it seems plausible that the asymptotically optimal codegree threshold is $n/k$, although a higher-dimensional analogue of the dominant-color argument is still needed.

\section*{Acknowledgments}

\begin{sloppypar}
  Y.C.~and X.L.~are grateful for the research visit to the second author at Beijing Institute of Technology in January 2026, where this work was conducted.
  X.L. was supported by the Excellent Young Talents Program (Overseas) of the National Natural Science Foundation of China.  Y.C. was supported by the Institute for Basic Science (IBS-R029-C4).
  J.H. was partially supported by the National Natural Science Foundation of China (12371341).
\end{sloppypar}

\bibliographystyle{abbrv}
\bibliography{hamiltoniancycle}

\begin{appendix}

  \renewcommand{\appendixpagename}{Appendix}
  \appendixpage

  \section{Auxiliary results and proofs for RCC-class absorption}
  \label{app:indexed-absorber}

  \subsection{cover lemmas}

  \begin{lemma}[Greedy RCC cover]
    \label{lem:greedy-rcc-cover}
    In the setting of \cref{def:rcc-cover}, suppose that
    $(\Omega,\kappa)$ is $(A,\chi,L)$-RCC-connectable, where $L$ is
    independent of $n$, and that
    \(
      |A\cap Z|+\frac{\theta L}{k-1}n\le\chi n,
      ~
      \theta\left(1+\frac{L}{k-1}\right)\le\varepsilon,
      ~
      \frac{\theta L}{k-1}\le\omega.
    \)
    Then $(\Omega,\kappa)$ has the
    $(A,Z;\theta,\varepsilon,\omega)$-RCC cover property.
  \end{lemma}

  \begin{proof}
    For a given family $\mathcal Q$, we have
    $|\mathcal Q|\le\theta n/(k-1)$.  While more than $r$ paths remain,
    reverse one end of each of $r+1$ paths and join two ends in the same
    $\kappa$-fibre through $A$, forbidding $A\cap Z$ and all previously used
    reservoir vertices.  The first inequality permits every join.  It
    preserves the two unchosen ends and each path in $\mathcal Q$ as a
    subpath.  At most $\theta n/(k-1)$ joins use at most
    $\theta Ln/(k-1)$ reservoir vertices, so the other two inequalities give
    the required size and reservoir bounds.
  \end{proof}

  For \(3\)-graphs we use a stronger cover lemma.

  \begin{lemma}[Path cover; cf.~{\cite[Lemma~3.1]{Han22}}]
    \label{connecter}
    Given $\alpha>0$, there exists $\zeta_0>0$ such that, for every
    $0<\zeta\le\zeta_0$, there exists $n_0=n_0(\alpha,\zeta)$ for which the
    following holds for all $n\ge n_0$.
    Let $H$ be an $n$-vertex $3$-graph with
    $\delta_2(H)\ge(1/3+\alpha)n$.
    Suppose that $q\ge1$ and that $P_1,\dots,P_q$ are vertex-disjoint tight
    paths in $H$ with
    $|V(P_1\cup\cdots\cup P_q)|\le \zeta n$.
    Then there are $s\in\{1,2\}$ and vertex-disjoint tight paths
    $P_1',\dots,P_s'$ such that all of the following hold:
    \begin{enumerate}[label=\rm(\roman*)]
      \item $P_1',\dots,P_s'$ contain $P_1,\dots,P_q$ as subpaths;
      \item $|V(P_1'\cup\cdots\cup P_s')|\le \sqrt{\zeta}n$;
      \item every end pair of $P_1',\dots,P_s'$ is an end pair of one of
        $P_1,\dots,P_q$.
    \end{enumerate}
  \end{lemma}

  \begin{proof}
    Take $\zeta_0\le1$.  The case $q=1$ is immediate.  For $q\ge2$, apply
    \cite[Lemma~3.1]{Han22} after designating two paths and one end of
    each.  Its successive joins preserve the unchosen ends by
    \cref{lem:gluing}\textup{(i)}, proving~\textup{(iii)}; the other
    assertions are those of that lemma.
  \end{proof}

  \begin{proposition}[RCC covers for $3$-graphs]
    \label{prop:3graph-rcc-cover}
    Let $H$ be an $n$-vertex $3$-graph, let $A,Z\subseteq V(H)$ with
    $|A\cup Z|\le n/2$, and put $H_0=H-(A\cup Z)$ and
    $N_0=|V(H_0)|$.  Suppose that
    \(
      \delta_2(H_0)\ge(1/3+\sigma)N_0.
    \)
    Let $\kappa:\Omega\to[r]$ be surjective, where $r\in\{1,2\}$ and
    $\Omega\subseteq\overrightarrow{\partial}_2(H)$ is reversal-closed.
    Suppose that $(\Omega,\kappa)$ is $(A,\chi,L)$-RCC-connectable and
    $|A\cap Z|\le\chi n$.  When $r=1$, suppose in addition that
    $L\le\omega n$.

    If $n$ is sufficiently large and
    \(
      2\theta\le\zeta_0(\sigma),
      ~
      \sqrt{2\theta}\le\varepsilon/2,
      ~
      \omega\le\varepsilon/2,
    \)
    where $\zeta_0(\sigma)$ is given by \cref{connecter}, then
    $(\Omega,\kappa)$ has the
    $(A,Z;\theta,\varepsilon,\omega)$-RCC cover property.  If
    $r=2$, the resulting paths can moreover be chosen inside $H_0$.
  \end{proposition}

  \begin{proof}
    Let $\mathcal Q$ be a nonempty path family satisfying the conditions in
    \cref{def:rcc-cover}.
    Since $N_0\ge n/2$, it has at most
    $\theta n\le2\theta N_0$ vertices.  Apply \cref{connecter} in $H_0$
    with parameters $\sigma$ and $2\theta$ to obtain at most two
    inherited-end paths on at most
    $\sqrt{2\theta}N_0\le\varepsilon n/2$ vertices.  Retain them if there is
    one path or if $r=2$.  If $r=1$ and there are two, choose one end from
    each and apply RCC-connectability to their reversals, forbidding
    $A\cap Z$; \cref{lem:gluing}\textup{(i)} joins the paths.  The join uses
    at most $L\le\omega n$ reservoir vertices, preserves the two other ends,
    and has total order at most $\varepsilon n/2+\omega n\le\varepsilon n$.
  \end{proof}

  \subsection{Reachability and local absorbers}

  \begin{proposition}\label{prop:absorbernumber}
    Suppose that $1/n\ll \eta \ll \beta \ll \alpha$.
    Let $H$ be an $n$-vertex $3$-graph with $\delta_2(H)\ge \left(\frac{1}{3}+\alpha\right)n$, and let $A\subseteq V(H)$ be a forbidden set of size at most $\eta n$.
    Then, for every $v\in V(H)$, we have $|\tilde N_{A,\beta,1}(v)| \ge \left(\frac{1}{3}+\frac{\alpha}{2}\right)n$.
  \end{proposition}

  \begin{proof}
    Choose $\beta\ll\gamma\ll\alpha$.  Let $L'(v)$ be the link graph of
    $v$ induced by $V(H)\setminus A$.  If
    $|E(L'(u)\cap L'(v))|\ge\gamma n^2$, supersaturation gives at least
    $\beta n^4$ four-vertex paths, each witnessing that $u$ and $v$ are
    $(\beta,1)$-reachable.  Thus every non-reachable $u$ satisfies
    $|E(L'(u)\cap L'(v))|<\gamma n^2$.

    Since $|E(L'(v))|\ge n^2/6$, double-counting pairs $(u,S)$ and splitting
    them according to reachability gives
    \[
      \left(\frac{1}{3}+\alpha\right)n|E(L'(v))|
      \le \sum_{S\in E(L'(v))}\deg_H(S)\le
      |\tilde N_{A,\beta,1}(v)||E(L'(v))|+\gamma n^3.
    \]
    The result follows from $\gamma\ll\alpha$.
  \end{proof}

  We use~\cite[Lemma~2.7]{DHS}, whose proof applies to the above
  reachability relation.

  \begin{lemma}
    [{\cite[Lemma~2.7]{DHS}}]
    \label{lem:partition}
    Given $\delta>0$ and $0<\eta\ll\beta\ll\delta$, the following holds for all sufficiently large $n$.
    Suppose that $H$ is an $n$-vertex $3$-graph and $A$ is a forbidden set of size at most $\eta n$ such that $|\tilde N_{A,\beta,1}(v)|\ge \delta n$ for every $v\in V(H)$.
    Then there is a partition $\mathcal P$ of $V(H)$ into $V_1,\dots,V_r$ with $r\le 1/\delta$ such that each $|V_i|\ge(\delta-\eta)n$ and each $V_i$ is $(\beta,2^{\lfloor 1/\delta\rfloor-1})$-closed in $H-A$.
  \end{lemma}

  Here \(K^{(k)}_{a_1,\dots,a_k}\) denotes the complete \(k\)-partite
  \(k\)-graph with class sizes \(a_1,\dots,a_k\).

  Given a $k$-graph $H$, a reversal-closed
  $\Omega\subseteq\overrightarrow{\partial}_{k-1}(H)$, and $b\in\mathbb N$, a
  labeled $b$-vertex absorber configuration consists of a set
  $W\in\binom{V(H)}b$ with a fixed labeling by $[b]$, together with a
  fixed vertex-disjoint family
  \(\mathcal Q^-(W)\) of \(\Omega\)-ended tight paths spanning \(W\).  It is
  \emph{incident with} a set \(S\) disjoint from \(W\) if there is a
  vertex-disjoint tight-path family \(\mathcal Q^+(W,S)\) spanning \(W\cup S\)
  with the same respective ordered ends.  Configurations intersect when their
  labeled vertex sets intersect; for each incident pair we fix one choice of
  \(\mathcal Q^+(W,S)\).

  \begin{lemma}[Same-index local absorbers]
    \label{lem:k-same-index-local-absorbers}
    Let $k\ge3$ and $m\ge1$ be fixed, and let
    $1/n\ll\nu\ll\lambda,\mu,1/k,1/m$.  Let $H$ be an $n$-vertex
    $k$-graph, let $A,Z\subseteq V(H)$ with $|Z|\le\nu n$, and put
    $H_0=H-(A\cup Z)$.  Let
    $\Omega\subseteq\overrightarrow{\partial}_{k-1}(H)$ be reversal-closed,
    let $\mathcal P=\{X_1,\dots,X_m\}$ be a partition of $V(H)$ whose
    parts are $(\lambda,2;\Omega)$-closed in $H-A$, and let
    $J\subseteq H_0$ be $\Omega$-supported.

    There is $c=c(k,m,\lambda,\mu)>0$ such that every
    $S\in\binom{V(H)}k$ with
    $\mathbf i_{\mathcal P}(S)\in\mathcal I_\mu(J,\mathcal P)$ has at
    least $cn^{b_k}$ configurations incident with $S$ whose labeled vertex
    sets lie in $H_0$, where $b_k=4k^2-2$.
  \end{lemma}

  \begin{proof}
    Fix $S=\{v_1,\dots,v_k\}$, put
    $\mathbf v=\mathbf i_{\mathcal P}(S)$, and let $J_{\mathbf v}$ contain
    the edges of index $\mathbf v$.  Supersaturation gives
    $c_0n^{3k-2}$ labeled copies of
    $F_k=K^{(k)}_{\underbrace{3,\dots,3}_{k-2},2,2}$ in $J_{\mathbf v}$;
    all but $O(n^{3k-3})$ avoid $S$.

    Label the classes of one remaining copy as
    \[
      C_i=\{u_i,u_{k+i},w_{i+1}\}\quad(1\le i\le k-2),
      \qquad
      C_{k-1}=\{u_{k-1},w_k\},
      \qquad
      C_k=\{u_k,w_1\}.
    \]
    Set
    \[
      Q=u_1\cdots u_{2k-2}
      \quad\text{and}\quad
      Q^+=u_1\cdots u_{k-1}w_1\cdots w_ku_k\cdots u_{2k-2}.
    \]
    Their class sequences are segments of
    $C_1,\dots,C_k,C_1,\dots$, so they are tight paths in $J_{\mathbf v}$
    with ends $(u_{k-1},\dots,u_1)$ and $(u_k,\dots,u_{2k-2})$.

    Comparing two transversal edges that differ only in $C_i$ shows that
    every $C_i$ lies in one part of $\mathcal P$.  Since
    $\{w_1,\dots,w_k\}$ has index $\mathbf v$, relabel $S$ so that $v_j$
    and $w_j$ lie in the same part for every $j\in[k]$.

    For each $j$, closedness gives $\lambda n^{4k-3}$ witnesses $T_j$
    replacing $w_j$ by $v_j$.  Choose them disjointly outside
    $Z\cup S\cup V(F_k)$.  At each step at most
    $O(\nu n^{4k-3})+O_k(n^{4k-4})$ choices are lost, leaving at least
    $(\lambda/2)n^{4k-3}$.  Fix the two
    $\Omega$-ended paths $R^w_{j,1},R^w_{j,2}$ spanning
    $T_j\cup\{w_j\}$ and their same-ended counterparts
    $R^v_{j,1},R^v_{j,2}$ spanning $T_j\cup\{v_j\}$.

    The families $\{Q\}\cup\{R^w_{j,h}\}$ and
    $\{Q^+\}\cup\{R^v_{j,h}\}$, indexed by $j\in[k]$ and $h\in[2]$, are
    vertex-disjoint, have the same respective ends, and differ in vertex set
    by $S$.  The number of configuration vertices is
    \[
      |V(F_k)|+\sum_{j=1}^k|T_j|
      =(3k-2)+k(4k-3)=4k^2-2=b_k.
    \]
    Since $J$ is $\Omega$-supported, $Q,Q^+$ are
    $\Omega$-ended.  Dividing by the bounded labeling multiplicity $D_k$
    leaves at least
    \[
      \frac{c_0}{D_k}n^{3k-2}
      \left((\lambda/2)n^{4k-3}\right)^k
      \ge cn^{4k^2-2}.
    \]
  \end{proof}

  \subsection{Proofs of the absorption lemmas}

  \begin{proof}[Proof of \cref{lem:indexed-absorber}]
    Call a $k$-set $S$ admissible if
    $\mathbf i_{\mathcal P}(S)\in\mathcal I$.  Put $b=4k^2-2$.  By
    \cref{lem:k-same-index-local-absorbers}, every admissible set is incident
    with at least $cn^b$ configurations for some $c>0$.  Choose
    $\gamma\ll\xi\ll\min\{c,\theta\}$ and sample each
    configuration with probability $p=\xi n^{-(b-1)}$.  There are
    $O(n^b)$ configurations and $O(n^{2b-1})$ intersecting pairs.  Chernoff's
    inequality, a union bound over the at most $n^k$ admissible sets, and
    Markov's inequality give a sample of size $O(\xi n)$ with
    $O(\xi^2n)$ intersecting pairs and at least $c\xi n/2$ members incident
    with every admissible set.  Deleting one configuration from each
    intersecting pair leaves a vertex-disjoint family $\mathcal F$ of
    size $O(\xi n)$ and at least $c\xi n/4$ choices for every admissible set.

    The families $\mathcal Q^-(W)$, for $W\in\mathcal F$, form an
    $\Omega$-ended path system in $H_0$
    on $O_k(\xi n)\le\theta n$ vertices.  RCC cover gives
    $1\le t\le r$ paths $P_1,\dots,P_t$ satisfying~\textup{(i)} and
    containing all these paths as subpaths.

    If $L$ satisfies~\textup{(ii)}, an expression of
    $\mathbf i_{\mathcal P}(L)$ in $\mathsf M(\mathcal I)$ partitions $L$
    into admissible $k$-sets by distributing each $L\cap X_j$ among the
    coordinate demands.  Greedily assign the at most
    $\gamma n/k\ll c\xi n$ blocks to distinct incident configurations and
    switch them, reversing a replacement whenever its non-absorbing path is
    embedded in the opposite direction.  This preserves all attachment ends,
    hence every joining path, and absorbs exactly $L$.
  \end{proof}

  \begin{proof}[Proof of \cref{absorber1}]
    Choose auxiliary constants so that
    \[
      \max\{\rho,\zeta,\eta,\nu,\gamma\}
      \ll\mu\ll\lambda\ll\theta\ll\varepsilon^2
      \ll\varepsilon\ll\beta\ll\min\{\alpha,d\}.
    \]
    By \cref{prop:absorbernumber,lem:partition}, with parameters
    $2\eta,\lambda$ and $\delta=1/3+\alpha/2$, there is a partition
    $\mathcal P=\{X_1,\dots,X_m\}$ with $m\in\{1,2\}$ whose parts have size
    at least $(1/3+\alpha/4)n$ and are $(\lambda,2)$-closed in $H-A$.

    Put $\Omega=\overrightarrow{\partial}_2(H)$ and $J=H-(A\cup Z)$, and let
    $\mathcal I$ contain all nonnegative $m$-vectors of coordinate sum $3$.
    For $\mathbf v\in\mathcal I$, choose parts $Y_1,Y_2,Y_3$ with
    multiplicities $\mathbf v$ and put
    $Y_\ell^\circ=Y_\ell\setminus(A\cup Z)$.  Uniform density gives
    \[
      e_H(Y_1^\circ,Y_2^\circ,Y_3^\circ)
      \ge d|Y_1^\circ||Y_2^\circ||Y_3^\circ|-\rho n^3
      \ge 6\mu n^3.
    \]
    Each edge is counted at most six times, so
    $\mathcal I\subseteq\mathcal I_\mu(J,\mathcal P)$.

    The hierarchy gives $|A\cup Z|\le n/2$ and
    $\delta_2(J)\ge(1/3+\alpha/2)|V(J)|$.  Moreover, $\Omega$ is
    reversal-closed, $J$ is $\Omega$-supported, and
    \cref{equivalentclass} makes $(\Omega,\kappa)$
    $(A,\zeta,L_{\rm RCC})$-RCC-connectable.  Set $\omega_r=\zeta^2$ for
    $r=1$ and $\omega_r=0$ for $r=2$.
    Apply \cref{prop:3graph-rcc-cover} to $H$ with $H_0=J$,
    $\sigma=\alpha/2$, $\chi=\zeta$, and $\omega=\omega_r$, keeping $\theta$
    and $\varepsilon$ as above.  Its hypotheses follow from the preceding codegree bound,
    $|A\cap Z|\le\zeta n/2$, and the auxiliary hierarchy; when $r=1$, the
    additional requirement is
    $L_{\rm RCC}\le\zeta^2n=\omega_rn$.  Thus $(\Omega,\kappa)$ has the
    $(A,Z;\theta,\varepsilon,\omega_r)$-RCC cover property.
    Applying \cref{lem:indexed-absorber} with $k=3$ and forbidden-set
    parameter $3\nu$ gives $1\le t\le r$ paths of total order at most
    $\varepsilon n\le\beta n$ and reservoir use at most $\omega_rn$.  If
    $r=2$, then $\omega_r=0$; since the paths lie in $H-Z$, they lie in $J$.
    Finally, every set $S$ specified in \cref{absorber1} splits into triples,
    so
    $\mathbf i_{\mathcal P}(S)\in\mathsf M(\mathcal I)$ and
    \cref{lem:indexed-absorber} absorbs it.
  \end{proof}

  \section[Proof of the monochromatic C10 lemma]{Proof of \cref{lem:monochromatic_c10}}
  \label{app:monochromatic-c10}

  Throughout this appendix, $\phi$ is the $A$-connected coloring, $W$ is the exceptional set given by \cref{prop:coloritem}, and $i\in\{1,2\}$ is the fixed color from \cref{lem:monosetexist}.
  Roughly speaking, $G_i[B]$ keeps the hyperedges whose three pairs all have
  color $i$, while $\partial_i[B]$ records all color-$i$ pairs.
  For $B\subseteq V(H)\setminus W$, let $G_i[B]$ be the $3$-graph on $B$ whose edges are those $xyz\in E(H[B])$ satisfying $\phi(xy)=\phi(xz)=\phi(yz)=i$.
  We write $\partial_i[B]$ for the graph on $B$ whose edges are the pairs $xy\in\binom B2$ with $\phi(xy)=i$.

  A \emph{homomorphism} from a $3$-graph $F$ to a $3$-graph $G$ is a map
  $f:V(F)\to V(G)$ such that $f(e)\in E(G)$ for every $e\in E(F)$.  Thus the
  three vertices of each edge of $F$ have distinct images, although vertices not
  lying together in an edge may be identified.  A homomorphic copy of $F$ in
  $G$ is the image of such a homomorphism.

  We establish the following strengthening of the result by Piga--Sanhueza-Matamala--Schacht~\cite{PIGA20261}.

  \begin{lemma}
    \label{lem:active-shadow-pss}
    For every $\gamma,\tau>0$, there exist $m_0\in\mathbb N$ and $\rho>0$ such that the following holds for all $m\ge m_0$.
    Let $G$ be an $m$-vertex $3$-graph, and put $R\coloneqq \partial G$.
    Suppose that $R$ is $(\rho,\tau)$-dense and $\deg_G(xy)\ge \left(\frac13+\gamma\right)m$ for every $xy\in E(R)$.
    Then $G$ contains a homomorphic copy of $C_{10}^{(3)}$.
  \end{lemma}

  \begin{proof}
    We adapt the proof of Piga, Sanhueza-Matamala and Schacht~\cite{PIGA20261} to the active shadow $R$.
    Choose a constant $\xi>0$ with $\xi\ll\min\{\gamma,\tau\}$, then choose $\rho\ll \xi^2\tau$, and finally choose $m_0$ sufficiently large.
    Suppose, for a contradiction, that $G$ contains no homomorphic copy of $C_{10}^{(3)}$.
    For a pair $xy\in E(R)$, write
    \[
      N_G(xy)=\{z\in V(G): xyz\in E(G)\}.
    \]

    We first show that every edge of $G$ is contained in a copy of $K_4^{(3)-}$ whose apex lies in that edge.
    Let $xyz\in E(G)$.
    Since $\partial G\subseteq R$, all three pairs $xy,xz,yz$ belong to $E(R)$.
    Hence
    \[
      |N_G(xy)|+|N_G(xz)|+|N_G(yz)|
      \ge (1+3\gamma)m.
    \]
    For sufficiently large $m$, some vertex $w\notin\{x,y,z\}$ lies in at least two of the three neighborhoods.
    For instance, if $w\in N_G(xy)\cap N_G(xz)$, then the three edges $xyz,xyw,xzw$ form a copy of $K_4^{(3)-}$.
    The other two cases are symmetric.

    In a copy of $K_4^{(3)-}$, call the unique vertex of degree $3$ the \emph{apex}.
    For a pair $uv\in E(R)$, call $uv$ an \emph{apex pair} if some copy of $K_4^{(3)-}$ in $G$ contains $u$ and $v$ and has one of $u,v$ as its apex.
    Call $uv$ a \emph{base pair} if some copy of $K_4^{(3)-}$ in $G$ contains $u$ and $v$ and has neither $u$ nor $v$ as its apex.
    Since $\deg_G(uv)>0$ for every $uv\in E(R)$, the preceding paragraph implies that every pair in $E(R)$ is either an apex pair or a base pair.

    A pair in $E(R)$ cannot be both an apex pair and a base pair.
    Indeed, suppose that $uv$ is both.
    Then there are two copies $K,K'$ of $K_4^{(3)-}$ containing $u$ and $v$, such that $v$ is the apex of $K$, while neither $u$ nor $v$ is the apex of $K'$.
    Write $V(K)=\{u,v,x,y\}$ and $V(K')=\{u,v,a,b\}$, where $a$ is the apex of $K'$.
    Then the cyclic sequence
    \[
      v,u,a,b,v,a,u,v,x,y
    \]
    maps every consecutive triple to an edge of $K\cup K'$.
    Thus it gives a homomorphic copy of $C_{10}^{(3)}$ in $G$, a contradiction.

    Define a directed graph $D$ on $V(G)$ by putting $(u,v)\in E(D)$ whenever $uv\in E(R)$ is an apex pair with apex $v$.
    Let $B$ be the graph consisting of all base pairs in $E(R)$.
    The directed graph $D$ has no directed $2$-cycle.
    Indeed, if both $(a,x)$ and $(x,a)$ are arcs, then there are two copies $K,K'$ of $K_4^{(3)-}$ containing $a$ and $x$, such that $a$ is the apex of $K$ and $x$ is the apex of $K'$.
    Write $V(K)=\{a,x,b,c\}$ and $V(K')=\{a,x,y,z\}$.
    Then the cyclic sequence
    \[
      x,a,y,x,z,a,x,b,a,c
    \]
    maps every consecutive triple to an edge of $K\cup K'$ and gives a homomorphic copy of $C_{10}^{(3)}$, again a contradiction.
    Consequently, for every pair $uv\in E(R)$, exactly one of the following alternatives holds:
    \[
      (u,v)\in E(D),\qquad (v,u)\in E(D),\qquad uv\in E(B).
    \]

    We next record three useful degree bounds.
    Put $c\coloneqq 1/3+\gamma$.
    First, if $d_B(v)>0$, then $d_D^+(v)\ge cm$.
    To see this, choose $u$ with $uv\in E(B)$.
    For every $w\in N_G(uv)$, the edge $uvw$ is contained in a copy of $K_4^{(3)-}$.
    In such a copy, the apex must be one of $u,v,w$.
    It cannot be $u$ or $v$, because then $uv$ would be an apex pair, contradicting that $uv\in E(B)$.
    Hence $w$ is the apex, and in particular $(v,w)\in E(D)$.
    Thus $N_G(uv)\subseteq N_D^+(v)$, so $d_D^+(v)\ge \deg_G(uv)\ge cm$.

    Second, if $d_D^+(v)>0$, then $d_B(v)\ge cm$.
    Choose $u$ with $(v,u)\in E(D)$.
    For every $w\in N_G(vu)$, take a copy of $K_4^{(3)-}$ containing the edge $vuw$.
    Its apex cannot be $v$, since then both $(v,u)$ and $(u,v)$ would be arcs of $D$.
    Its apex also cannot be $w$, since then $vu$ would be a base pair.
    Hence the apex is $u$, and therefore $vw\in E(B)$.
    Thus $N_G(vu)\subseteq N_B(v)$, and $d_B(v)\ge cm$.

    Third, if $d_D^-(v)>0$, then $d_D^-(v)\ge cm$.
    Choose $u$ with $(u,v)\in E(D)$.
    For every $w\in N_G(uv)$, the same argument shows that in any copy of $K_4^{(3)-}$ containing $uvw$, the apex must be $v$.
    Therefore $(w,v)\in E(D)$.
    Hence $N_G(uv)\subseteq N_D^-(v)$, and $d_D^-(v)\ge cm$.

    Let
    \[
      S=\{v\in V(G): d_D^+(v)=0\}.
    \]
    We claim that $|S|\le \xi m$.
    Indeed, there is no edge of $R$ inside $S$.
    For if $uv\in E(R[S])$, then the three cases above and the first degree implication force at least one of $u,v$ to have positive out-degree in $D$.
    If $|S|>\xi m$, split $S$ into two disjoint sets $S_1,S_2$ with $|S_1|,|S_2|\ge \xi m/3$.
    Since $R[S]$ has no edge, $e_R(S_1,S_2)=0$.
    This contradicts the $(\rho,\tau)$-density of $R$, because $\rho\ll \xi^2\tau$.

    Now put $T=V(G)\setminus S$.
    Since $|T|\ge (1-\xi)m$, we can split $T$ into two disjoint sets $T_1,T_2$ of size at least $m/3$.
    The $(\rho,\tau)$-density of $R$ gives an edge $xy$ with $x\in T_1$ and $y\in T_2$.
    As $\deg_G(xy)\ge cm$ and $|S|\le \xi m\ll cm$, we may choose $z\in T\setminus\{x,y\}$ such that $xyz\in E(G)$.
    By the first paragraph, the edge $xyz$ is contained in a copy of $K_4^{(3)-}$.
    Let $v^\ast\in\{x,y,z\}$ be the apex of this copy.
    Since $v^\ast\in T$, we have $d_D^+(v^\ast)>0$.
    Also, the two other vertices of the edge $xyz$ send arcs to the apex, so $d_D^-(v^\ast)>0$.
    By the second degree implication, $d_B(v^\ast)>0$.

    Applying the three degree implications to $v^\ast$, the three sets $N_D^+(v^\ast)$, $N_D^-(v^\ast)$, and $N_B(v^\ast)$ each have size at least $cm$.
    They are pairwise disjoint by the three cases for pairs in $E(R)$.
    Hence their union has size at least $3cm=(1+3\gamma)m$, which is impossible in an $m$-vertex graph.
    This contradiction proves that $G$ contains a homomorphic copy of $C_{10}^{(3)}$.
  \end{proof}

  We shall also use two standard lemmas to deduce a blow-up copy from a
  homomorphic copy.

  \begin{lemma}[Hypergraph removal lemma {\cite{Gowers2007}}]
    \label{lem:hypergraph-removal}
    Let $\mathcal F$ be a finite family of fixed $3$-graphs.
    For every $\delta>0$, there exists $\varepsilon>0$ such that the following holds for all sufficiently large $m$.
    If an $m$-vertex $3$-graph $G$ contains at most $\varepsilon m^{v(F)}$ labeled copies of each $F\in\mathcal F$, then one can delete at most $\delta m^3$ edges from $G$ and obtain a $3$-graph which contains no copy of any $F\in\mathcal F$.
  \end{lemma}

  \begin{lemma}[Hypergraph supersaturation lemma {\cite{Keevash2011}}]\label{lem:blowup-supersaturation}
    Let $F$ be a fixed $3$-graph and let $s\in\mathbb N$.
    For every $c>0$, there exists $n_0$ such that the following holds for all $n\ge n_0$.
    If an $n$-vertex $3$-graph $G$ contains at least $c n^{v(F)}$ labeled copies of $F$, then $G$ contains a copy of $F[s]$.
  \end{lemma}

  Combining the active-shadow lemma with removal and supersaturation gives the form used in the proof of \cref{lem:monochromatic_c10}.

  \begin{lemma}\label{lem:active-shadow-c10-blowup}
    For every $\gamma,\tau>0$, there exist $m_0\in\mathbb N$ and $\rho>0$ such that the following holds for all $m\ge m_0$.
    Let $G$ be an $m$-vertex $3$-graph, and let $R$ be a graph on $V(G)$ with $\partial G\subseteq R$.
    Suppose that $R$ is a $(\rho,\tau)$-dense $2$-graph and $\deg_G(xy)\ge \left(\frac13+\gamma\right)m$ for every $xy\in E(R)$.
    Then $G$ contains a copy of $C_{10}^{(3)}[2]$.
  \end{lemma}

  \begin{proof}
    Let $C=C_{10}^{(3)}$.
    Let $\mathcal Q(C)$ be the finite family of all quotient $3$-graphs of $C$ obtained by identifying vertices of $C$ in such a way that every edge of $C$ is mapped to a $3$-edge.
    Thus $G$ contains a homomorphic copy of $C$ if and only if $G$ contains an injective copy of some $Q\in\mathcal Q(C)$.

    Let $\rho_1>0$ and $m_1$ be the constants given by \cref{lem:active-shadow-pss} with parameters $\gamma/4$ and $\tau/4$.
    Choose constants $0<\delta\ll \kappa\ll \min\{\gamma,\tau,\rho_1\}$.
    Apply \cref{lem:hypergraph-removal} to the finite family $\mathcal Q(C)$ with this value of $\delta$, and let $\varepsilon>0$ be the resulting constant.

    For every $Q\in\mathcal Q(C)$, apply \cref{lem:blowup-supersaturation} with $F=Q$, $s=20$, and $c=\varepsilon$.
    We choose $m_0$ sufficiently large for all these applications, and also large compared with $m_1$.
    Finally choose $\rho>0$ sufficiently small in terms of $\rho_1,\tau,\kappa$.

    Suppose, for a contradiction, that $G$ contains no copy of $C_{10}^{(3)}[2]$.
    We claim that, for every $Q\in\mathcal Q(C)$, the graph $G$ contains fewer than $\varepsilon m^{v(Q)}$ labeled copies of $Q$.
    Indeed, otherwise \cref{lem:blowup-supersaturation} would give a copy of $Q[20]$ in $G$.
    Let $\pi:V(C)\to V(Q)$ be the quotient map.  Each vertex of $Q$ has at
    most $10$ preimages under $\pi$.  Since $Q[20]$ has $20$ vertices above
    each vertex of $Q$, we may assign two distinct vertices to each preimage,
    giving an embedding $C[2]\hookrightarrow Q[20]$.  Hence $Q[20]$ contains a copy of
    $C_{10}^{(3)}[2]$, a contradiction.

    By \cref{lem:hypergraph-removal}, there is a set $E_0\subseteq E(G)$ with $|E_0|\le \delta m^3$ such that $G'\coloneqq G\setminus E_0$ contains no copy of any $Q\in\mathcal Q(C)$.
    Equivalently, $G'$ contains no homomorphic copy of $C$.

    Call a pair $xy\in E(R)$ bad if $\deg_{G'}(xy)<(1/3+\gamma/2)m$.
    Since each bad pair has lost at least $\gamma m/2$ incident hyperedges, and each deleted hyperedge is incident with exactly three pairs, the number of bad pairs is at most $(6\delta/\gamma)m^2\le \kappa^2m^2$.
    Let $B_{\rm bad}$ be the graph of bad pairs.

    Delete every vertex whose degree in $B_{\rm bad}$ is larger than $\kappa m$, and let the remaining vertex set be $U$.
    Since $e(B_{\rm bad})\le \kappa^2m^2$, we have $|V(G)\setminus U|\le 2\kappa m$.
    Put $m'=|U|$.
    Then $m'\ge (1-2\kappa)m$.

    Let $R_U\coloneqq R[U]\setminus B_{\rm bad}$.
    For all sets $X,Y\subseteq U$, we have
    \[
      e_{R_U}(X,Y)
      \ge \tau|X||Y|-(\rho+2\kappa^2)m^2.
    \]
    Since $m'\ge(1-2\kappa)m$, our choice
    $\max\{\rho,\kappa\}\ll\min\{\rho_1,\tau\}$ implies
    \[
      e_{R_U}(X,Y)
      \ge \frac{\tau}{4}|X||Y|-\rho_1(m')^2.
    \]
    Thus $R_U$ is $(\rho_1,\tau/4)$-dense on $U$.

    Now define a $3$-graph $G_U$ on $U$ by keeping those edges $xyz\in E(G'[U])$ for which all three pairs $xy,xz,yz$ lie in $R_U$.
    Then $\partial G_U\subseteq R_U$.

    We claim that every pair $xy\in E(R_U)$ satisfies $\deg_{G_U}(xy)\ge (1/3+\gamma/4)m'$.
    Indeed, since $xy$ is not bad, $\deg_{G'}(xy)\ge (1/3+\gamma/2)m$.
    Passing from $V(G)$ to $U$ loses at most $2\kappa m$ possible third vertices.
    Moreover, because $x,y\in U$, each of $x$ and $y$ is incident with at most $\kappa m$ bad pairs.
    Thus, when we require also $xz,yz\in R_U$, we lose at most another $2\kappa m$ choices of $z$.
    Hence
    \[
      \deg_{G_U}(xy)
      \ge
      \left(\frac13+\frac{\gamma}{2}\right)m-4\kappa m
      \ge
      \left(\frac13+\frac{\gamma}{4}\right)m',
    \]
    using $\kappa\ll\gamma$.

    Since $\partial G_U\subseteq R_U$ and every pair in $E(R_U)$ has positive degree in $G_U$, we have $R_U=\partial G_U$.
    Therefore the pair $(G_U,R_U)$ satisfies the assumptions of \cref{lem:active-shadow-pss} with parameters $\gamma/4$ and $\tau/4$.
    Hence $G_U$ contains a homomorphic copy of $C$.
    But $G_U\subseteq G'$, contradicting the fact that $G'$ contains no homomorphic copy of $C$.

    This contradiction proves that $G$ contains a copy of $C_{10}^{(3)}[2]$.
  \end{proof}

  \begin{proof}[Proof of \cref{lem:monochromatic_c10}]
    We prove the following slightly stronger statement.
    For every set $Z\subseteq V(H)\setminus(A\cup W)$ with $|Z|\le 20$, the graph $H[V(H)\setminus(A\cup W\cup Z)]$ contains a copy of $C_{10}^{(3)}[2]$ whose shadow pairs all have color $i$.

    Fix such a set $Z$ and put $B'\coloneqq V(H)\setminus(A\cup W\cup Z)$.
    Let $G=G_i[B']$ and $R=\partial_i[B']$.
    Then $\partial G\subseteq R$ by definition.

    By \cref{lem:shadowlocaldense}, the color-$i$ shadow graph on $V(H)\setminus W$ is dense.
    Restricting this graph to $B'=V(H)\setminus(A\cup W\cup Z)$, and weakening the constants if necessary, we obtain that $R=\partial_i[B']$ is a $(\rho',\tau)$-dense $2$-graph for some constants satisfying $\rho'\ll \tau \ll \min\{\alpha,\beta,d\}$.
    Here $\tau$ is a fixed positive constant, and the restriction loses only the vertices in $A\cup Z$.

    Let $xy\in E(R)$.
    By \cref{prop:coloritem}~\ref{coloritem2}, there are at least $(1/3+\alpha/4)n$ vertices $z\in V(H)\setminus W$ such that $xyz\in E(H)$ and $\phi(xy)=\phi(xz)=\phi(yz)=i$.
    After deleting $A\cup Z$ and the vertices $x,y$, we still have $\deg_G(xy) \ge \left(\frac13+\frac{\alpha}{5}\right)n \ge \left(\frac13+\frac{\alpha}{6}\right)|B'|$ for sufficiently large $n$.

    Apply \cref{lem:active-shadow-c10-blowup} to $G$ and $R$ with $\gamma=\alpha/6$ and with the above value of $\tau$.
    We obtain a copy of $C_{10}^{(3)}[2]$ in $G$.

    By the definition of $G_i[B']$, every hyperedge of this copy is an edge of $H$ whose three shadow pairs have color $i$.
    Hence every shadow pair appearing in the copy has color $i$.
    Thus the stronger statement holds.

    Applying the stronger statement first with $Z=\emptyset$ and then with $Z$ equal to the vertex set of the first copy gives two vertex-disjoint copies of $C_{10}^{(3)}[2]$ in $H[V(H)\setminus(A\cup W)]$, both with all shadow pairs of color $i$.
    This finishes the proof.
  \end{proof}
\end{appendix}
\end{document}